\documentclass[a4paper,reqno,10pt]{amsart}
    \DeclareMathSizes{12}{12}{7}{6}
\usepackage{geometry}
\usepackage{amssymb,amsmath,mathrsfs,bm}
\usepackage{paralist}
\usepackage[usenames]{color}
\usepackage[all]{xy}
\usepackage{url}
\usepackage{braket}
\usepackage[all]{xy}
\usepackage{graphicx}
\usepackage{transparent}
\usepackage{xcolor}
\usepackage[shortlabels]{enumitem}
\usepackage{mathtools}
\usepackage{tikz}
\usepackage{tikz-cd}

\usepackage{wasysym} 
\usepackage{stmaryrd} 
\makeatletter

\@addtoreset{equation}{section}
\makeatother
\usepackage{amsthm}
\usepackage{aliascnt}
\newtheorem{theorem}{Theorem}[section]
\newaliascnt{lemma}{theorem}
\newtheorem{lemma}[lemma]{Lemma}
\aliascntresetthe{lemma}

\newaliascnt{proposition}{theorem}
\newtheorem{proposition}[proposition]{Proposition}
\aliascntresetthe{proposition}

\newaliascnt{corollary}{theorem}
\newtheorem{corollary}[corollary]{Corollary}
\aliascntresetthe{corollary}

\theoremstyle{definition}
\newaliascnt{definition}{theorem}
\newtheorem{definition}[definition]{Definition}
\aliascntresetthe{definition}

\theoremstyle{definition}
\newaliascnt{remark}{theorem}
\newtheorem{remark}[remark]{Remark}
\aliascntresetthe{remark}

\newaliascnt{setup}{theorem}
\theoremstyle{definition}
\newtheorem{setup}[setup]{Setup}
\aliascntresetthe{setup} 

\newaliascnt{example}{theorem}
\newtheorem{example}[example]{Example}
\aliascntresetthe{example}

\newaliascnt{condition}{theorem}

\aliascntresetthe{condition}

\newaliascnt{construction}{theorem}

\aliascntresetthe{construction}

\newaliascnt{question}{theorem}

\aliascntresetthe{question}

\newaliascnt{conjecture}{theorem}

\aliascntresetthe{conjecture}

\newtheorem{claim}{Claim}
\newtheorem{mainthm}{Theorem}

\usepackage[hyperfootnotes=false, hidelinks]{hyperref}
\usepackage[capitalise,noabbrev]{cleveref}
\crefname{theorem}{Theorem}{Theorems}
\crefname{lemma}{Lemma}{Lemmas}
\crefname{proposition}{Proposition}{Propositions}
\crefname{corollary}{Corollary}{Corollaries}
\crefname{definition}{Definition}{Definitions}
\crefname{remark}{Remark}{Remarks}
\crefname{example}{Example}{Examples}
\crefname{condition}{Condition}{Conditions}
\crefname{construction}{Construction}{Constructions}
\crefname{claim}{Claim}{Claims}
\crefname{mainthm}{Theorem}{Theorems}
\crefname{maincor}{Corollary}{Corollaries}
\crefname{setup}{Setup}{Setups}

\def\Mor{\operatorname{Mor}}

\def\Ext{\operatorname{Ext}}

\def\Hom{\operatorname{Hom}}
\def\mod{\operatorname{\mathsf{mod}}}
\def\proj{\operatorname{\mathsf{proj}}}
\def\Ab{\mathsf{Ab}}
\def\add{\operatorname{\mathsf{add}}}

\newcommand{\op}{\mathsf{op}}

\newcommand{\Ker}{\operatorname{Ker}}
\renewcommand{\Im}{\operatorname{Im}}

\newcommand{\id}{\mathsf{id}}
\newcommand{\xto}{\xrightarrow}

\newcommand{\HOM}{{\mathcal H}om}
\newcommand{\NAT}{{\mathcal N}at}
\newcommand{\dg}{{\rm dg}}
\newcommand{\tr}{\mathsf{tr}}
\newcommand{\pretr}{\mathsf{pretr}}
\newcommand{\coh}{\textnormal{coh}}
\newcommand{\rep}{\textnormal{rep}}
\newcommand{\pvd}{\mathsf{pvd}}
\newcommand{\per}{\mathsf{per}}

\newcommand{\lra}{\longrightarrow}
\newcommand{\dra}{\dashrightarrow}
\def\cone{\operatorname{\mathsf{Cone}}}
\def\cocone{\operatorname{\mathsf{CoCone}}}
\newcommand{\inc}{\mathsf{inc}}
\newcommand{\can}{\mathsf{can}}
\renewcommand{\sp}{\mathsf{sp}}

\newcommand{\Sn}{\mathcal{S}_{\CN}}

\newcommand\Inf{\operatorname{\mathsf{Inf}}}
\newcommand\Def{\operatorname{\mathsf{Def}}}

\newcommand{\wtil}[1]{\widetilde{#1}}
\newcommand{\ovl}[1]{\overline{#1}}

\def\cof{\operatorname{\mathsf{cof}}}

\def\Ho{\operatorname{\mathsf{Ho}}}
\def\Int{\operatorname{\mathsf{Int}}}

\newcommand\Cat{\mathsf{Cat}}
\newcommand\ET{\mathsf{ETcat}}

\newcommand\TR{\mathsf{TRcat}}
\newcommand\EX{\mathsf{EXcat}}
\newcommand\AB{\mathsf{ABcat}}
\newcommand\dgcat{\mathsf{dgcat}}
\newcommand\Hqe{\mathsf{Hqe}}

\newcommand{\Db}{\mathcal{D}^{\rm b}}

\newcommand{\Hb}{\mathcal{H}^{\rm b}}

\newcommand{\Ck}{\mathcal{C}(k)}
\newcommand{\Cdg}{\mathcal{C}_{\rm dg}}
\newcommand{\Ddg}{\mathcal{D}_{\rm dg}}
\newcommand{\thick}{\operatorname{\mathsf{thick}}}
\newcommand{\Tria}{\operatorname{\mathsf{Tria}}}

	\newcommand{\deff}{\coloneqq}
	\newcommand{\sse}{\subseteq}

    \newcommand{\fs}{\mathfrak{s}}
    \newcommand{\ft}{\mathfrak{t}}

	\newcommand{\BE}{\mathbb{E}}
	\newcommand{\BF}{\mathbb{F}}

	\newcommand{\BZ}{\mathbb{Z}}

	\newcommand{\CA}{\mathcal{A}}
	\newcommand{\CB}{\mathcal{B}}
	\newcommand{\CC}{\mathcal{C}}
	\newcommand{\CD}{\mathcal{D}}

	\newcommand{\CH}{\mathcal{H}}

	\newcommand{\CK}{\mathcal{K}}
	
	\newcommand{\CM}{\mathcal{M}}
	\newcommand{\CN}{\mathcal{N}}
	
	\newcommand{\CP}{\mathcal{P}}

	\newcommand{\CS}{\mathcal{S}}
	\newcommand{\CT}{\mathcal{T}}
	\newcommand{\CU}{\mathcal{U}}
	\newcommand{\CV}{\mathcal{V}}
	\newcommand{\CW}{\mathcal{W}}
	\newcommand{\CX}{\mathcal{X}}

    \newcommand{\A}{\mathscr{A}}
	\newcommand{\B}{\mathscr{B}}
    \newcommand{\C}{\mathscr{C}}
	\newcommand{\D}{\mathscr{D}}
	\newcommand{\N}{\mathscr{N}}
    \renewcommand{\P}{\mathscr{P}}
	\renewcommand{\SS}{\mathscr{S}}
    \newcommand{\T}{\mathscr{T}}

\usetikzlibrary{matrix,arrows,decorations.pathmorphing,positioning,decorations.pathreplacing}
\tikzset{commutative diagrams/.cd, 
mysymbol/.style = {start anchor=center, end anchor = center, draw = none}}
\tikzset{
labl/.style={anchor=north, rotate=90, inner sep=1mm}
}

\let\amph=& 
	
\tikzcdset{every label/.append style = {font = \footnotesize}}

\newcommand{\wPO}[1]{\arrow[mysymbol]{#1}[description]{\mathrm{(wPO)}}}

\newcommand{\wPB}[1]{\arrow[mysymbol]{#1}[description]{\mathrm{(wPB)}}}

\begin{document}
\setlength{\baselineskip}{15pt}
\title
[An enhanced extriangulated subquotient]{An enhanced extriangulated subquotient}

\author[Mochizuki]{Nao Mochizuki}
	\address{Graduate School of Mathematics, Nagoya University, Furocho, Chikusaku, Nagoya 464-8602, Japan}
	\email{mochizuki.nao.n8@s.mail.nagoya-u.ac.jp} %

\author[Ogawa]{Yasuaki Ogawa}
	\address{Faculty of Engineering Science, Kansai University, Suita-shi, Osaka 564-8680, Japan}
	\email{y\underline{ }ogawa@kansai-u.ac.jp} %



\begin{abstract}
Bondal-Kapranov's notion of enhanced triangulated categories behaves well in the framework of localization theory, in the sense that the Verdier quotient of triangulated categories can be lifted to the Drinfeld dg quotient of pretriangulated dg categories. In this paper, we develop a parallel enhancement for Nakaoka-Palu's notion of extriangulated categories, which unifies exact and triangulated categories. The enhancement of extriangulated categories was recently initiated by Xiaofa Chen under the name exact dg categories. Moreover, it is known that certain ideal quotients of extriangulated categories remain extriangulated, and that such ideal quotients admit dg enhancements via the dg quotient of the corresponding connective exact dg category. Motivated by Chen's framework of enhanced extriangulated categories, we introduce the concept of a cohomological envelope of an exact dg category and generalize his construction of the enhanced ideal quotient. We show that the dg quotient of exact dg categories, when passing to cohomological envelopes and their substructures---referred to as exact dg subquotients---is compatible with a broad class of extriangulated quotients in the sense of Nakaoka-Ogawa-Sakai. To further clarify the scope of our approach, we formulate the notion of an extriangulated subquotient, which enables the localization of any extriangulated category by extension-closed subcategories. This construction encompasses not only ideal and Verdier quotients, but also the quotient of an exact category by a biresolving subcategory. Notably, the extriangulated subquotient admits a natural lifting to the exact dg subquotient.

\end{abstract}
\maketitle
\tableofcontents

\section*{Introduction}\label{sec:intro}
The localization theory in the context of homological algebra including abelian/triangulated category theory can be traced back to the localization of commutative rings.
Following a general mathematical principle of decomposing a large object to smaller pieces, it has been gradually extended to localizations of spaces in topology, schemes in algebraic geometry and other things in many mathematical branches.
The abelian/triangulated category theory serves as the foundation to do homological algebra and their appropriate localizations, the so-called Serre/Verdier quotients \cite{Gab62,Ver96}, have been playing central roles and possess satisfactory generalities within their respective scopes.
In practice, we often need to move back and forth between abelian and triangulated categories: Given an abelian category $\CA$, a construction of the derived category $\CD(\CA)$ leads us to a triangulated category;
The (projectively) stable category of a Frobenius abelian/exact category \cite{Hap88} also carries a natural triangulated structure;
Conversely, a $t$-structure $(\CT^{\leq 0},\CT^{\geq 0})$ of a triangulated category $\CT$ was introduced in an attempt to recover an abelian origin category from which $\CT$ can be naturally obtained \cite{BBD}.
The authors proved that the intersection $\CH\deff\CT^{\leq 0}\cap\CT^{\geq 0}$, called the heart, is always abelian and is equipped with a cohomological functor $H\colon \CT\to\CH$, which brings us back from the triangulated circumstance to an abelian one.
As a unifying notion, Nakaoka-Palu defined extriangulated category \cite{NP19} which contains a more general notion of Quillen's exact category.
It not only provides a convenient setting in which to formulate arguments that simultaneously apply to both exact and triangulated categories, but also strengthens the growing evidence that many concepts and constructions in homological algebra are most naturally understood through the perspective of extriangulated categories, such as
\begin{itemize}
\item 
cotorsion pairs and their hearts \cite{LN19,HS20,AET23},
\item 
Grothendieck groups \cite{ZZ21,Hau21,ES22,PPPP23,OS24a},
\item 
tensor (ex)triangulated categories and Balmer spectrum \cite{BGLS25},
\item 
Auslander-Reiten theory \cite{INP24, Jin20, MP25},
\item 
silting objects and mutation \cite{GNP23,Bor24,Iit24,PZ25},
\item 
$n$-(ex)angulated categories and $n$-cluster tilting subcategories \cite{HLN21,HLN22,JS23} and
\item 
exact $\infty$-enhancement \cite{NP20,BT21}.
\end{itemize}
Most relevant to the present work is the localization theory of extriangulated categories with respect to certain classes of morphisms, established in \cite{NOS22}, which contains the Serre/Verdier quotients.
We will develop the extriangulated localization theory within the framework of exact dg enhancements, as introduced by Chen \cite{Che23}.

The use of differential graded categories (= dg categories) to enhance triangulated categories was initiated by Bondal-Kapranov \cite{BK90}.
A triangulated category $\CT$ is said to be algebraic if it admits a dg enhancement by a pretriangulated dg category $\T$, that is, if $\CT\simeq H^0\T$.
It is widely known that triangulated categories do not behave well to compute some well-established invariants, nor to describe triangle equivalences between derived categories over algebras.
In addition, triangulated categories do not reproduce, for instance, the functor category over $\CT$ is not expected to carry triangulated structures (see \cite{Nee18, Mat24} for further developments).
Bondal-Kapranov’s notion of enhanced triangulated categories precisely compensates for such missing information inherent in triangulated categories.
We mention a few notable examples here:
Under the notion of a Frobenius pair (= a localization pair), the $K$-theory and the cyclic homology were made compatible with localizations of algebraic triangulated categories in \cite{TT90,Sch06} and \cite{Kel99} respectively;
In \cite{Kel94}, Keller leveraged Rickard's Morita theory \cite{Ric89} to the setting of derived categories over dg categories, providing a concrete description of derived equivalences between algebras;
In a similar spirit to Rickard, but with a distinctly different purpose, To{\"e}n initiated the Morita theory for dg algebras \cite{Toe07}.
A synthesis of To{\"e}n's Morita theory and Auslander-Reiten theory was studied as well \cite{HI22};
In contrast to the triangulated categories, the (derived) functor category $\rep_{\dg}(\A,\B)$ remains pretriangulated if $\A$ is pretriangulated.
This rich interplay between dg and triangulated categories motivates us to expand the coverage of the dg enhancement from triangulated categories to extriangulated categories, which has been done under the name of exact dg category \cite{Che23}.
We still say such an extriangulated category to be algebraic.
Our particular focus is the Drinfeld dg quotient of dg categories $\A$ by its dg full subcategory $\N$ \cite{Dri04, Tab10}, a significant feature of which provides an enhancement of the Verdier quotient of the associated homotopy category $H^0\A$ by the thick closure of $H^0\N$ provided $\A$ is pretriangulated.
In light of the current situation, one may naturally ask the following question: ``Can any extriangulated localization be understood as a dg quotient?''
Our results give a partial answer to this question and reveal intricate interactions between the quotient of extriangulated categories and the dg quotient of their exact dg enhancements.
We also note that there has been a growing body of recent work on exact dg categories \cite{Che23b,Che23c,Che24c,Bor24,Moc25,MP25}, to which the present paper contributes.

Now we give a brief summary of the main results presented in the current paper.
Our first theorem shows that a weakly idempotent complete (WIC) extriangulated category can be localized by any extension-closed subcategory which is an extriangulated version of the main result in \cite{Oga24}.

\begin{mainthm}[\cref{thm:subquotient_of_extri_cat}]\label{mainthm:A}
Let $(\CA,\BE,\fs)$ be a WIC extriangulated category together with an extension-closed subcategory $(\CN,\BE|_\CN,\fs|_\CN)$ which is closed under direct summands.
Then, there exists a natural extriangulated substructure $(\CA_\CN,\BE_\CN,\fs_\CN)$ such that the extriangulated quotient of $\CA_\CN$ by $\CN$ fits into the following diagram: 
\begin{equation}\label{eq:A}
\begin{tikzcd}
(\CN,\BE|_\CN,\fs|_\CN)\arrow[hook]{r}{}
&
(\CA_\CN,\BE_\CN,\fs_\CN)\arrow{r}{(Q,\mu)}
&
(\CA_\CN/\CN,\wtil{\BE_\CN},\wtil{\fs_\CN}).
\end{tikzcd}
\end{equation}
\end{mainthm}

We call this construction the \emph{extriangulated subquotient} of $\CA$ by $\CN$, where the prefix `sub' and the notation `$(-)_\CN$' indicate we first pass to the substructure of $\CA$ determined by $\CN$.
This notion has a broad spectrum of connections to many known localizations, including the Verdier quotient of triangulated categories, the ideal quotient of extriangulated categories and the quotient of exact categories by biresolving subcategories in the sense of \cite{Rum21}.
Moreover, in these examples, taking the substructure of $\CA$ is unnecessary, that is, we see $\CA=\CA_\CN$.
This formulation of the extriangulated subquotient is further justified by the following result.

\begin{mainthm}[\cref{cor:algebraic_subquotient}]\label{mainthm:B}
The extriangulated subquotient $(\CA_\CN/\CN,\wtil{\BE_\CN},\wtil{\fs_\CN})$ in \cref{mainthm:A} is still algebraic.
\end{mainthm}

Note that this generalizes the prominent fact that algebraic triangulated categories are closed under taking Verdier quotients \cite[\S 3.6]{Kel06}.
As mentioned before, the dg quotient gives an enhancement of the Verdier quotient.
Our second aim, therefore, is to construct an enhanced extriangulated subquotient, which we achieve by introducing the notion of a \emph{cohomological envelope}.
Before describing our construction, let us recall that, for any exact dg category $(\A,\SS)$, taking its connective cover $(\tau_{\leq 0}\A,\SS)$ (= truncation at the degree zero) does not alter the homotopy category: $H^0\A=H^0(\tau_{\leq 0}\A)$.
In line with this philosophy of discarding `irrelevant' positive degrees, Chen has successfully formulated an enhanced ideal quotient of extriangulated categories by projective-injective objects which is realized as a dg quotient of $\tau_{\leq 0}\A$ \cite[\S 3.5]{Che24b}.
Our approach extends this direction further: an \emph{enhanced extriangulated subquotient} can be viewed as obtained by adding the `appropriate' positive degrees to $\tau_{\leq 0}\A$.
We emphasize that our construction relies on the existence of the bounded dg derived category $\Db(\tau_{\leq 0}\A)$ of the connective cover of $\A$ due to Chen, see \cref{thm:universal_embedding}.
Indeed, the cohomological envelope $\A^\coh$ is defined to be the quasi-essential image of $\tau_{\leq 0}\A$ in $\Db(\tau_{\leq 0}\A)$, see \cref{def:cohomological_envelope}.
We can now state our second main theorem, which establishes the compatibility of extriangulated subquotients with certain dg quotients.

\begin{mainthm}[\cref{thm:dg_subquotient}]\label{mainthm:C}
Let $(\A,\SS)$ be an exact dg category with an extension-closed subcategory $(\N,\SS|_\N)$.
Then, there exists a natural exact dg substructure $(\A_\N,\SS_\N)$ of $(\A,\SS)$ such that we have a dg quotient
\[
\begin{tikzcd}
(\N,\SS|_\N)\arrow[hook]{r}{}
&
(\A_\N^\coh,\SS_\N)\arrow{r}{Q}
&
(\A_\N^\coh/\N,\wtil{\SS_\N})
\end{tikzcd}
\]
which enhances an extriangulated subquotient as in \eqref{eq:A}, where $\A_\N^\coh$ denotes the cohomological envelope of $(\A_\N,\SS_\N)$.
\end{mainthm}

We refer to this the \emph{exact dg subquotient} of $\A$ by $\N$ following the name in \cref{mainthm:A}.
It includes the known dg enhancements as summarized in \cref{fig:summary}.

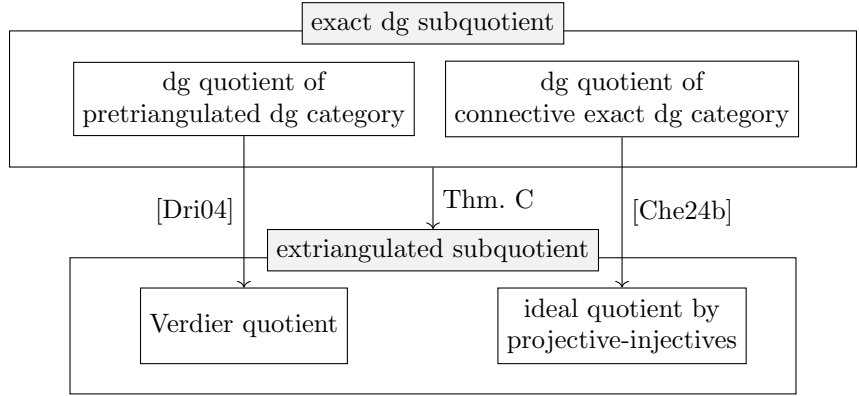
\begin{figure}[ht]
\begin{tikzpicture}
    \node[rectangle, draw=black, text centered,align=center] at (0,3.0) (pretr) 
    {dg quotient of\\
    pretriangulated dg category};
    \node[rectangle, draw=black, text centered, minimum height=1.0cm, minimum width=1.0cm,align=center] at (0,0) (Verdier) {
    Verdier quotient};
  \node[rectangle, draw=black, text centered, minimum height=1.0cm, minimum width=1cm,align=center] at (5,3.0) (Chen) {
    dg quotient of\\
    connective exact dg category};
  \node[rectangle, draw=black, text centered, minimum height=1.0cm, minimum width=1cm,align=center] at (5,0) (ideal) {
    ideal quotient by\\
    projective-injectives};
    \node[rectangle,draw=black, minimum height=1.8cm, minimum width=11.2cm] at (2.5,3.0) (B) {};
  \node[rectangle, draw=black, fill=lightgray!20, text centered, minimum height=0.5cm, minimum width=1cm,align=center] at (2.5,4.0) (dg subquot) {exact dg subquotient};
    \node[rectangle,draw=black, minimum height=1.8cm, minimum width=9.6cm] at (2.5,0) (A) {};
    \node[rectangle, draw=black, fill=lightgray!20, text centered, minimum height=0.5cm, minimum width=1cm,align=center] at (2.5,1.0) (ET) {extriangulated subquotient};
  \draw[->] (pretr) -- node[left] {\cite{Dri04}} (Verdier);
  \draw[->] (Chen) -- node[right] {\cite{Che24b}} (ideal);
  \draw[->] (B) -- node[right] {Thm.~\ref{mainthm:C}} (ET);
\end{tikzpicture}
\caption{Connections to other dg enhancements.}
\label{fig:summary}
\end{figure}

This article is organized as follows.
The sections \S 1 and \S 2 are devoted to recall needed foundations
of the dg category theory. We focus on the exact dg categories in \S 2. The key notion of bounded dg derived category is reviewed there.
In \S 3, we formulate the extriangulated subquotient $\CA_\CN/\CN$ of $(\CA,\BE,\fs)$ with respect to any extension-closed subcategory $(\CN,\BE|_\CN,\fs|_\CN)$.
We also show that if $\CA$ is algebraic, then so is the subquotient $\CA_\CN/\CN$ by proving \cref{mainthm:B}.
In \S 4, we introduce the notion of cohomological envelope and place the extriangulated subquotient in the exact dg context such as \cref{mainthm:C}.
Last, in \S\ref{sec:enhanced_heart_construction}, as an application, we understand Liu-Nakaoka's heart construction of cotorsion pairs in extriangulated categories as an exact dg subquotient.
In turn, we have an enhanced cohomological functor to the associated enhanced heart.

\medskip
\noindent
{\bf Notation and convention.}
Throughout the article, $k$ denotes a commutative ring with unity and all the categories are considered to be $k$-linear.
We will work with both differential graded (=dg) categories and ordinary additive categories.
Depending on which framework is under consideration, the same notation may be used with slightly different meanings, and we will follow the corresponding conventions in each context.

Let $\CA$ be an additive category.
An \emph{additive subcategory} $\CB$ of $\CA$ means a full subcategory which is closed under finite direct sums and isomorphisms.
The \emph{ideal quotient} is denoted by $\CA/[\CB]$.
For any additive functor $F\colon \CA\to \CC$, its (essential) \emph{image} $\Im F$ and the \emph{kernel} $\Ker F$ are defined as
\begin{eqnarray*}
\Im F &=& \{C\in\CC\mid FA\cong C\textnormal{\ for some\ }A\in\CA\},\\
\Ker F &=& \{A\in\CA\mid FA=0\textnormal{\ in\ }\CC\}
\end{eqnarray*}
respectively.

Let $\A$ be a dg category.
Slightly different from the above, it is said to be \emph{additive} if the homotopy category $H^0\A$ is additive in the usual sense.
For a dg functor $F\colon \A\to\B$, the symbol $\Im F$ denotes the \emph{quasi-essential image} of $F$, namely, the canonical enhancement of the class $\{FA\in\B\mid A\in\A\}$.
Throughout the paper, we assume that a dg category is small to avoid the set theoretical obstruction.

\section{Elements of dg categories}
\label{sec:dg_cat}
The section \S \ref{sec:dg_cat} will be devoted to just recalling needed foundations of dg category which is used in this article.
Those readers who are familiar with the dg category may skip this section.
In \S \ref{subsec:monoidal_cat}, we will review some monoidal structures closely related to dg categories.
The subsection \S \ref{subsec:Der_cat_and_pretr} will refer to basic structures concerning dg enhancements of triangulated categories.
In more details, we continue reviewing in \S \ref{subsec:Hqe}  the category $\Hqe$ of dg categories and the Drinfeld dg quotient in $\Hqe$ as a dg enhancement of the Verdier quotient.

\subsection{Monoidal structures associated to dg categories}
\label{subsec:monoidal_cat}
The reason why we start the section with the monoidal structure is coming from that it is a convenient framework to understand dg categories: Actually they are categories enriched by the monoidal category $\CC(k)$ of chain complexes of $k$-modules; In addition, both categories $\dgcat$ and $\Hqe$ consisting of dg categories admit monoidal structures respectively, see \S\ref{subsec:Hqe} for $\Hqe$.
In this article, a monoidal category $\CK$ is always assumed to be \emph{symmetric}, i.e., there exists a (bi-)functorial isomorphism $X\otimes Y\xto{\sim} Y\otimes X$ in any $X,Y\in\CK$.
We refer to \cite{Kel82} for details of monoidal categories.

\begin{definition}\label{def:monoidal_cat}
Let $\CK$ be a symmetric monoidal category.
Suppose that, for any object $Y$, the functor $-\otimes Y\colon \CK\to\CK$ admits a right adjoint $[Y,-]$.
We thus have a natural isomorphism
\[
[X\otimes Y, Z]\xto{\sim} [X,[Y,Z]]
\]
in any objects $X,Y,Z\in\CK$.
In this case, the monoidal category $\CK$ is said to be \emph{closed} and the object $[X,Y]\in\CK$ is called \emph{internal hom} of $\CK$.
\end{definition}

The category $\CC(k)$ of chain complexes admits a monoidal structure with respect to the tensor product $A\otimes B$ for chain complexes $A$ and $B$.
Recall that the \emph{tensor complex} $A\otimes B$ is structured by
$
(A\otimes B)^n = \bigoplus_{i+j=n}A^i\otimes_k B^j
$
together with the differential
\[
d_{A\otimes B}(x\otimes y)=d_A(x)\otimes y+(-1)^{|x|}x\otimes d_B(y)
\]
for homogeneous elements $x\in A^n$ and $y\in B^m$, where $|x|=n$ denotes the degree of $x$.
The right adjoint of tensor functor $-\otimes B$ is the internal hom written as $\HOM(B,-)$.
It is defined by
$
\HOM(B,C)^n = \prod_{i\in\BZ}\Hom(B^i,C^{i+n}),
$
namely, a set $\{f^i\colon B^i\to C^{i+n}\}_{i\in\BZ}$ of $k$-module homomorphisms of degree $n$.
The differential $d=d_{\HOM(B,C)}$ is determined by $df=d_Cf-(-1)^nfd_B$ for homogeneous elements $f\in\HOM(B,C)^n$.
It turns out that they make $\CC(k)$ to be a closed monoidal category.

A monoidal category $\CK$ permits us to define the enriched category whose morphism space belongs to $\CK$, called the \emph{$\CK$-enriched} category.
The \emph{differential graded category} (\emph{dg category}, for short) is defined to be just the $\CC(k)$-enriched category, which is unpacked explicitly as below, e.g. \cite[\S 2.3]{Toe11}.

\begin{definition}\label{def:dg_cat}
A $k$-linear dg category $\A$ is a category equipped with the following datum:
\begin{enumerate}[label=\textup{(\arabic*)}]
\item
For any pair $(A,B)$ of objects in $\A$, the $\Hom$-space forms a chain complex of $k$-modules $\A(A,B)$;
\item 
For any $A\in\A$, a chain map $k\to\A(A,A)$ exists and represents the identity, where $k$ is regarded as a stalk complex of degree $0$; and
\item 
The composition is defined as a chain map $\A(B,C)\otimes \A(A,B)\to \A(A,C)$ in $\C(k)$ and satisfies the obvious associativity and unity.
\end{enumerate}

A morphism $f\in\A(A,B)$ is said to be \emph{closed} if $df=0$.
By the definition of the tensor complex, we deduce from the item (3) the following equation called the \emph{Leibniz rule}:
\[
d(gf)=(dg)f+ (-1)^{n}g(df)
\]
for homogeneous elements $g\in \A(B,C)^n$ and $f\in \A(A,B)^m$.

Also, we associate the following categories to a given dg category $\A$:
\begin{itemize}
    \item 
    The \emph{strict category} $Z^0\A$ is the category with the same objects as $\A$ and whose $\Hom$-spaces are $(Z^0\A)(A,B)=Z^0(\A(A,B))$, where $Z^0$ is the kernel of the differential $d\colon \A(A,B)^0\to \A(A,B)^1$.
    \item 
    Similarly, the \emph{homotopy category} $H^0\A$ is the category with the same objects as $\A$ and whose $\Hom$-spaces are $(H^0\A)(A,B)=H^0(\A(A,B))$, where $H^0$ is the $0$-th cohomology of $\A(A,B)$.
\end{itemize}
Two objects in $\A$ are said to be \emph{isomorphic} (resp. \emph{homotopy equivalent}) if they are isomorphic in $Z^0\A$ (resp. $H^0\A$).
\end{definition}

\begin{example}\label{ex:dg_cat_of_chain_cpx}
As we have already mentioned, the category $\CC(k)$ is a closed monoidal category.
The internal $\HOM$ with respect to the structure on $\CC(k)$ makes it into a dg category $\CC_\dg(k)$.
In this case, we see that its strict category recovers the abelian category $\CC(k)$ and its homotopy category is the ordinary homotopy category $\CH(k)$, the ideal quotient of $\CC(k)$ by the contractible complexes.
\end{example}

In the remainder, we will see that the category of small dg categories admits a natural monoidal structure.
A morphism of dg categories is defined to be a functor between $\CC(k)$-enriched categories.
Explicitly, we have this:

\begin{definition}
\label{def:dg_functor}
Let $\A$ and $\B$ be dg categories.
A functor $F\colon \A\to\B$ is called a \emph{dg functor} if, for any $A,B\in\A$, the morphism
\[
F_{A,B}\colon \A(A,B)\to \B(FA,FB)
\]
is a morphism in $\Ck$.
We denote by $\dgcat$ the category of small dg categories and dg functors.
\end{definition}

The category $\dgcat$ is closed under many categorical operations:
The \emph{opposite category} $\A^{\op}$ of a dg category $\A$ is naturally equipped with a dg category structure;
A full subcategory $\B\sse\A$ possesses a dg structure as well.
With this view, we call $\B$ a \emph{dg full subcategory} of $\A$;
The category $\dgcat$ is closed under taking the tensor product.

\begin{definition}
\label{def:tensor_of_dg_cat}
Let $\A$ and $\B$ be dg categories.
The \emph{tensor product} $\A\otimes\B$ is the dg category defined by the following datum:
\begin{itemize}
\item 
The objects are the pairs $(A,B)$ of $A\in\A$ and $B\in\B$.
\item 
The hom-space is defined as
\[
\A\otimes\B((A,B),(A',B'))=\A(A,A')\otimes\B(B,B')
\]
together with the composition
\[
(f'\otimes g')\circ (f\otimes g)=(-1)^{nm}(f'\circ f)\otimes (g'\circ g)
\]
for homogeneous elements $f'\in\A(A',A'')^{n}, f\in\A(A,A'), g'\in\B(B',B'')^{m}$ and $g\in\B(B,B')$.
\end{itemize}
\end{definition}

Also, the corresponding internal $\HOM$ exists:
we need to consider the complex of the natural transformations between dg functors.

\begin{definition}\label{def:graded_natural_trans}
Let us consider dg functors  $F,G\colon \A\to\B$ between dg categories.
We define the complex $\NAT_{\dg}(F,G)$ of (\emph{graded}) \emph{natural transformations} in such a way:
\begin{itemize}
\item 
For each object $A\in\A$, it consists of a family of morphisms $\varphi_A\in\B(FA,GA)^n$ in degree $n$;
\item 
For any morphism $f\in\A(A,B)^m$, it requires the diagram
\[
\begin{tikzcd}[row sep=0.6cm]
F(A) \arrow{d}[swap]{F(f)}\arrow{r}{\varphi_A}& G(A)\arrow{d}{G(f)} \\
F(B) \arrow{r}{\varphi_B} &G(B)
\end{tikzcd}
\]
to be commutative up to sign $(-1)^{nm}$ which we call the \emph{graded naturality condition}; The differential is defined objectwise.
We can check that $d\varphi_A\in\B(FA,GA)^{n+1}$ fulfills the graded naturality condition as well.
Thus, $\NAT_{\dg}(F,G)$ actually sits in $\CC(k)$.
\end{itemize}
Furthermore, the above construction permits us to think of $\NAT_{\dg}(F,G)$ as the $\Hom$-space of the dg category of dg functors from $\A$ to $\B$.
It actually equips the class $\HOM(\A,\B)$ of dg functors with a dg category structure.
\end{definition}

\begin{lemma}\label{lem:Hom-tensor_adjunction_on_dgcat}
For any $\A,\B,\C\in\dgcat$, there exists a natural isomorphism in $\dgcat$.
\[
\HOM(\A\otimes\B,\C)\cong\HOM(\A,\HOM(\B,\C))
\]
In turn, the category $\dgcat$ admits a closed monoidal structure.
\end{lemma}

\subsection{Derived categories and pretriangulated hulls}
\label{subsec:Der_cat_and_pretr}
For a dg category $\A$, we provide a quick review on a construction of its derived category $\CD(\A)$ which is indeed a localization of the category $\CC(\A)$ of dg $\A$-module.
Localizations of categories are very difficult to understand in general.
The model category is one of the most powerful tool to describe them.
The first model category will be $\CC(k)$.
This is structured by taking the fibrations to be degree-wise surjective morphisms, and the weak equivalence to be the quasi-isomorphism.
The homotopy category $\Ho\CC(k)$ is the usual derived category $\CD(k)$, see \cite{Hov98} for more details.

Now we will recall that $\CC(\A)$ naturally carries the model structure from $\CC(k)$.

\begin{definition}
\label{def:dg_module}
For a small dg category $\A$, we introduce the following notions.
\begin{enumerate}[label=\textup{(\arabic*)}]
\item 
A \emph{dg $\A$-module} $M$ is a dg functor $M\colon \A^{\op}\to\Cdg(k)$.
\item 
A dg $\A$-module is said to be \emph{representable} (resp. \emph{quasi-representable}) if it is isomorphic (resp. quasi-isomorphic) to a dg $\A$-module of the form $A^{\wedge}\deff \A(-,A)$ for some $A\in\A$.
\end{enumerate}
We denote by $\Cdg(\A)\deff \HOM(\A^{\op},\CC_{\dg}(k))$ the dg category of dg $\A$-modules.
Following \cref{def:dg_cat}, we have the strict category $\CC(\A)\deff Z^0\Cdg(\A)$.
\end{definition}

We can define a model structure on $\CC(\A)$ by defining fibrations (resp. weak equivalences) to be the morphisms $f\colon M\to M'$ such that the induced morphism $f(A)\colon M(A)\to M'(A)$ is a fibration (resp. a weak equivalence) in $\CC(k)$ for each $A\in\A$, see \cite[\S 3]{Toe07}.
Thus, the category $\CC(\A)$ admits not only a dg structure but also a model structure.
Under these notions, we define the derived category of $\A$ equipped with a dg enhancement.

\begin{theorem}
\label{thm:derived_category}
The \emph{derived category} $\CD(\A)$ of a dg category $\A$ is the homotopy category of $\CC(\A)$, that is, $\CD(\A)\deff \Ho\CC(\A)$.
In addition, we denote by $\Ddg(\A)$ the dg full subcategory of cofibrant complexes in $\Cdg(\A)$ and call it the \emph{dg derived category} of $\A$.
The following linkage from $\Ddg(\A)$ to $\CD(\A)$ explains the name:
\[
H^0(\Ddg(\A))\simeq \CD(\A).
\]
\end{theorem}

\begin{remark}\cite[Ch. 4]{Hov98}
There is a notion of \emph{$\CC(k)$-enriched model category} which was introduced to deal with the model category ``enriched'' by $\CC(k)$ such as the category $\CC(\A)$, see also \cite[\S 3]{Toe07}.

Let $\C$ be a $\CC(k)$-enriched model category and $\Int\C$ the dg full subcategory of fibrant-cofibrant objects of $\C$.
Then, by the general model category theory, we have $H^0(\Int\C)=\Ho\C$.
This explains the above linkage, since any object in $\CC(\A)$ is fibrant.
\end{remark}

\begin{example}
If we consider a usual $k$-algebra $A$ as a dg category, then the model category $\CC(A)$ in \cref{thm:derived_category} is known as the projective model structure \cite{Hov02}.
It's homotopy category $\Ho\CC(A)$ is equivalent to the usual derived category $\CD(A)$ by definition.
\end{example}

We will recall other relevant structures in connection with the derived category $\CD(\A)$, which provides an elementary understanding to it.
In the same fashion as $\CC(k)$, the strict category $\CC(\A)$ is an abelian category.
By tweaking the structure, we can impose a Frobenius exact category on $\CC(\A)$, see \cite[\S 2.2]{Kel94} for details, in which the class of contractible dg $\A$-modules forms the full subcategory of all projective-injective objects.
Hence, the ideal quotient $\CH(\A)$ of $\CC(\A)$ with respect to the contractible objects is equipped with a natural triangulated structure.
Notice that the homotopy category $H^0\Cdg(\A)$ coincides with this ideal quotient $\CH(\A)$.

We denote by $M[1]$ the shifted $\A$-module of $M$ (i.e., $(M[1])(A)\deff M(A)[1]$) and by $C(f)$ the \emph{cone} of $f$.
The associated inclusion $N(A)\xto{\iota}C(f)(A)$ and the projection $C(f)(A)\xto{p}M(A)[1]$ are both chain maps, which yield an exact sequence $0\to N\xto{\iota} C(f)\xto{p} M[1]\to 0$ in $\CC(\A)$.
In a similar manner, the \emph{cocone} $C(f)[-1]$ of $f$ is defined as well.
The following explains how the triangulated structure on $\CH(\A)$ is determined.

\begin{lemma}\label{lem:triangulated_str_on_HA}
The homotopy category $\CH(\A)$ admits a triangulated structure in the following way:
the translation $[1]$ is the objectwise shifts as complexes and the triangles are isomorphic to a sequence of the form
$
M\xto{f}N\xto{\iota}C(f)\xto{p}M[1]
$.
\end{lemma}

The derived category $\CD(\A)$ is defined to be a Verdier quotient of $\CH(\A)$ as well and hence it is triangulated:
A dg $\A$-module $M\in\CC(\A)$ is said to be \emph{acyclic} if $M(A)$ is an acyclic complex for any $A\in\A$.
We denote by $\CH_{\mathsf{ac}}(\A)$ the full subcategories of acyclic dg $\A$-modules in $\CH(\A)$.
Then $\CH_{\mathsf{ac}}(\A)$ is a (co)localizing subcategory of $\CH(\A)$ which gives rise to the following Verdier quotient admitting a left adjoint:
\[
\begin{tikzcd}
\CH_{\mathsf{ac}}(\A)\arrow[hook]{r}
&\CH(\A)\arrow{r}{q}\arrow[bend right = 50]{l}
&\CD(\A)\arrow[bend right = 50]{l}[swap]{\lambda}
\end{tikzcd}
\]
We have thus defined as $\CD(\A)\deff\CH(\A)/\CH_{\mathsf{ac}}(\A)$ and see that the quotient restricts to a triangle equivalence $\CH_{\mathsf{prj}}(\A)\deff\Im\lambda\xto{\sim}\CD(\A)$.
Therefore, the dg derived category $\Ddg(\A)$ is a canonical enhancement of the corresponding dg full subcategory of $\Cdg(\A)$.

We recall the Yoneda embedding for dg categories, where the notion of algebraic triangulated category arises.

\begin{theorem}[derived Yoneda embedding]
\label{thm:dg_Yoneda_embedding}
Let $\A$ be a dg category.
Then we have an isomorphism
\[
\CD(\A)(A^{\wedge},M)\cong H^0(M(A))
\]
in $\Ck$ which is functorial in $A\in\A$ and $M\in\CD(\A)$.
Moreover, the embedding $\A\hookrightarrow \Cdg(\A)$ induces a fully faithful functor $H^0\A\hookrightarrow\CD(\A)$.
\end{theorem}

We recall the notion of \emph{pretriangulated hull} $\pretr(\A)$ of a dg category $\A$.
It is defined to be the smallest dg full subcategory of $\CC_{\dg}(\A)$ which contains $\A$ and is closed under taking cones and shifts.
A dg category $\A$ is \emph{pretriangulated} if it is quasi-equivalent to the pretriangulated hull of a dg category.
A triangulated category is \emph{algebraic} if it is triangle equivalent to $H^0\A$ for some pretriangulated dg category $\A$.
We put $\tr(\A)\deff H^0(\pretr(\A))$ as usual.

\begin{lemma}\label{lem:pretriangulated_hll}
The triangulated category $\tr(\A)$ is the smallest triangulated subcategory of $\CD(\A)$ which contains $H^0\A$.
In particular, the derived Yoneda embedding $H^0\A\hookrightarrow \CD(\A)$ extends to the fully faithful embedding $\tr(\A)\hookrightarrow \CD(\A)$.
\end{lemma}

\subsection{The Drinfeld dg quotient}
\label{subsec:Hqe}
The homotopy category $H^0(\A)$ is considered as the genuine category for a given dg category $\A$, so an appropriate framework to deal with morphisms in $\dgcat$ is the category $\Hqe$ that we will explain in \cref{def:quasi-equivalence}.
The aim of this subsection is to recall the notion of Drinfeld dg quotient, a dg version of Verdier quotient in the sense of \cref{prop:enhancement_of_Verdier_quotient}.
We also record some easy observations for the dg quotient which will be used in \S\ref{sec:subquotient}.

\begin{definition}\label{def:quasi-equivalence}
A dg functor $F\colon \A\to\B$ is called a \emph{quasi-equivalence} if the following conditions are fulfilled:
\begin{enumerate}[label=\textup{(\arabic*)}]
\item 
\emph{quasi-fully faithful}, namely, the morphism
\[
F_{A,B}\colon \A(A,B)\to \B(FA,FB)
\]
is a quasi-isomorphism in $\CC(k)$ for $A,B\in\A$; and
\item 
\emph{quasi-essentially surjective}, namely, the induced functor $H^0(F)\colon H^0(\A)\to H^0(\B)$ is essentially surjective.
\end{enumerate}
Note, if the above conditions are satisfied, we have an equivalence $H^0F\colon H^0\A\xto{\sim}H^0\B$.
The category $\Hqe$ is defined as the localization of $\dgcat$ at the quasi-equivalences.
\end{definition}

\begin{remark}\label{rem:quasi-equivalence}
If $\A$ is pretriangulated, an equivalence $H^0F\colon H^0\A\xto{\sim}H^0\B$ implies that $F$ is quasi-fully faithful.
In particular, we have the quasi-equivalence $\A\hookrightarrow \pretr(\A)$ for any pretriangulated dg category $\A$.
\end{remark}

Thanks to the Dwyer-Kan model structure on $\dgcat$ whose weak equivalences are quasi-equivalences, we can understand $\Hqe$ as the homotopy category of $\dgcat$.

\begin{theorem}\label{thm:Hqe}
\cite{Tab05, Toe07}
We have a model structure on $\dgcat$ by defining the weak equivalences to be the quasi-equivalences and the fibrations to be the morphisms $F\colon \A\to\B$ satsifying the following two properties:
\begin{enumerate}[label=\textup{(\arabic*)}]
\item 
For any objects $A,B\in\A$, the induced morphism
\[
F_{A,B}\colon \A(A,B)\to \B(FA,FB)
\]
is a fibration in $\CC(k)$.
\item 
For any isomorphism $u'\colon A'\to B'$ in $H^0\A$ and $B\in H^0\A$ such that $F(B)=B'$, there is an isomorphism $u\colon A\to B$ in $H^0(\A)$ such that $H^0(F)(u)=u'$.
\end{enumerate}
The homotopy category of $\dgcat$ is denoted by $\Hqe$.
\end{theorem}

We know $\Hqe$ has small hom-spaces by the general model category theory.
Since the fibrations $\CC(k)$ are surjective morphisms, we see that any object in $\dgcat$ is fibrant.
Thus a morphism $\A\to\B$ in $\Hqe$ is represented as a roof in $\dgcat$
\[
\A\xleftarrow{\sigma}\A_{\cof}\xto{\alpha}\B,
\]
where $\A\xleftarrow{\sigma}\A_{\cof}$ is the cofibrant replacement for $\A$.
In particular, we have a morphism $H^0(\alpha)H^0(\sigma)^{-1}\colon H^0\A\to H^0\B$.
Also we notice that, if $\A$ is flat (e.g. cofibrant), each morphism $\A\to\B$ in $\Hqe$ can be taken as a morphism in $\dgcat$.

The cofibrant replacement plays an essential role to equip $\Hqe$ with a monoidal structure in the same spirit as \S\ref{subsec:monoidal_cat}.
Note that the tensor product $-\otimes -$ over $\dgcat$ does not preserve quasi-equivalences, so such a naive approach does not work well.
This leads us to pass to the cofibrant replacement.
Actually, a cofibrant dg category $\B$ is \emph{flat} in the sense that $-\otimes\B$ preserves the quasi-equivalences.
Thus putting $\A\otimes^\mathbf{L}\B\deff \A\otimes(\B_{\cof})$, we have a bifunctor $-\otimes^\mathbf{L}-\colon \Hqe\times \Hqe\to \Hqe$ which moreover defines a monoidal structure on $\Hqe$.
The tensor product $-\otimes^{\mathbf{L}}-$ admits an internal $\HOM$ which makes $\Hqe$ to be a closed monoidal category.
Note that the dg categories $\A\otimes\B$ and $\A\otimes^\mathbf{L}\B$ are \emph{Morita equivalent}, that is, their dg derived categories $\Ddg(\A\otimes\B)$ and $\Ddg(\A\otimes^\mathbf{L}\B)$ are quasi-equivalent.

\begin{definition}\label{def:internal_Hom_in_Hqe}
Let us consider the derived category $\CD(\A^{\op}\otimes^\mathbf{L} \B)$  of $\A$-$\B$-bimodules.
We denote by $\rep(\A,\B)$ the full subcategory of bimodules $X$ such that $X(-,B)$ is quasi-representable for each object $B\in\B$, namely, the tensor functor
\[
-\otimes_{\B}^\mathbf{L} X\colon \CD(\B)\to \CD(\A)
\]
takes the representable $\B$-modules to objects isomorphic to representable $\A$-modules.
Also, we denote by $\rep_{\dg}(\A,\B)$ the canonical dg enhancement which is defined as the full dg subcategory of $\CD_{\dg}(\A^{\op}\otimes^\mathbf{L} \B)$ consisting of the objects in $\rep(\A,\B)$.
\end{definition}

We thus have a closed monoidal structure on $\Hqe$ in the sense that there exists a natural bijection $\Hqe(\A\otimes^\mathbf{L}\B,\C)\cong \Hqe(\A,\rep_{\dg}(\B,\C))$ proved in \cite[Thm.~1.3]{Toe07}, see also \cite[Thm.~1.1]{CS15}.

In the remainder, we recall the notion of the Drinfeld dg quotient and collect some basic properties.
The flatness of dg categories is essential there too.
\emph{In the rest, we assume that any dg category is {\rm additive} in the sense that its homotopy category is additive, but nothing is lost for our purpose.}
Let now $\A$ be a dg category and $\N\sse\A$ a full dg subcategory (or just a class of objects in $\A$).
We say that a morphism $F\colon \A\to\C$ in $\Hqe$ \emph{annihilates $\N$} if the induced functor $H^0F$ vanishes on $H^0\N$.

\begin{definition}\label{def:dg_quotient}
For the given pair $(\A,\N)$, the \emph{dg quotient} $\A/\N$ is a dg category equipped with a morphism $Q\colon \A\to\A/\N$ in $\Hqe$ which is universal with respect to the property that annihilates $\N$.
We often present such a situation by
\[
\N\lra \A\overset{Q}{\lra} \A/\N .
\]
\end{definition}

We know from \cite{Dri04} that the dg quotient $\A/\N$ exists and the induced functor $H^0Q$ is essentially surjective.
Furthermore, the dg quotient of a flat dg category still remains flat.
Due to the existence of the pretriangulated hull, the defining universality of the dg quotient slightly reduces to the following.

\begin{lemma}\label{lem:universality_of_dg_quotient}
For the given pair $(\A,\N)$, we consider a dg category $\B$  equipped with a morphism $Q\colon \A\to\B$ in $\Hqe$.
Then the following are equivalent.
\begin{enumerate}[label=\textup{(\arabic*)}]
\item 
$Q\colon \A\to\B$ is a dg quotient.
\item 
If there exists a morphism $F\colon \A\to\T$ to a pretriangulated category which annihilates $\N$, then we have a unique morphism $F'\colon \B\to\T$ such that $F=F'\circ Q$.
\end{enumerate}
\end{lemma}
\begin{proof}
The implication $(1)\Rightarrow (2)$ is obvious.
To show the converse, we take a morphism $F\colon \A\to\C$ which annihilates $\N$ together with the pretriangulated hull $\C\xto{\inc}\T\deff\pretr(\C)$.
By (2), we have a unique morphism $F'\colon \B\to\T$ depicted as a dotted arrow in the following commutative diagram.
\[
\begin{tikzcd}[column sep=1.0cm, row sep=0.5cm]
    {}
    &{\T}
    &{}
    \\
    {}
    &\C\arrow{u}{\inc}
    &{}
    \\
    \N \arrow{r}
    &\A\arrow{r}{Q}\arrow{u}{F}
    &\B\arrow[dotted]{uul}[swap]{F'}
\end{tikzcd}
\]
Thus it suffices to show that the quasi-essential image $\Im F'$ of $F'$ is contained in the quasi-essential image $\Im\inc$.
Note however that this follows from the essential surjectivity of $H^0Q$.
Since the $\inc\colon\C\to\T$ is a fully faithful dg functor, the uniqueness also follows.
\end{proof}

As already mentioned, the dg quotient has an aspect as an enhancement of the Verdier quotient.

\begin{proposition}\label{prop:enhancement_of_Verdier_quotient}
Let us consider an algebraic triangulated category $\CA$ and a triangulated subcategory $\CN\sse\CA$.
If we consider a pretriangulated dg enhancement $\A$ of $\CA$ together with the canonical dg enhancement $\N$ of $\CN$,
then the dg quotient $\A/\N$ induces the Verdier quotient $\CA/\CN$ under $H^0$ which is indicated as below.
\[
\begin{tikzcd}[column sep=1.0cm]
    \N \arrow{r}&
    \A\arrow{r}{Q'} &
    \A/\N
\end{tikzcd}
\qquad
\overset{H^0}{\rightsquigarrow}
\qquad
\begin{tikzcd}[column sep=1.0cm]
    \CN \arrow{r}&
    \CA\arrow{r}{H^0Q'} &
    \CA/\CN
\end{tikzcd}
\]
\end{proposition}

We proceed observations on the dg quotient $\N\to\A\xto{Q'}\A/\N$ for any pair $(\A,\N)$ in connection with their pretriangulated hulls.
Let us consider the pretriangulated hull $F\colon \A\hookrightarrow \pretr(\A)$ and the \emph{triangulated closure} $\Tria\CN$ of $\CN=H^0\N$ in $\tr(\A)$, namely, the smallest triangulated subcategory of $\tr(\A)$ containing $\CN$.
Also we denote by $\Tria_{\dg}\N$ its canonical dg enhancement in $\pretr(\A)$.

By \cref{prop:enhancement_of_Verdier_quotient}, we have the dg quotient which enhances the Verdier quotient of $\tr(\A)$ by $\Tria\CN$ as the following diagram indicates.
\begin{equation}\label{seq:ambient_Verdier_quotient}
\begin{tikzcd}[column sep=1.0cm]
    \Tria_{\dg}\N \arrow{r}&
    \pretr(\A)\arrow{r}{Q} &
    \dfrac{\pretr(\A)}{\Tria_{\dg}\N}
\end{tikzcd}
\quad
\overset{H^0}{\rightsquigarrow}
\quad
\begin{tikzcd}[column sep=1.0cm]
    \Tria\CN \arrow{r}&
    \tr(\A)\arrow{r}{H^0Q} &
    \dfrac{\tr(\A)}{\Tria\CN}
\end{tikzcd}
\end{equation}

Now, we restrict the above dg quotient $Q$ to $\A$ along the pretriangulated hull $F\colon \A\hookrightarrow \pretr(\A)$ as in \eqref{diag:restriction_of_dg_quotient}.
The next corollary shows that the dg quotient $\A/\N$ emerges as a subsequence associated to $\frac{\pretr(\A)}{\Tria_{\dg}(\N)}$.

\begin{corollary}\label{cor:restriction_of_dg_quotient}
Let us denote by $\Im(Q|_{\A})$ the quasi-essential image of $Q\circ F$ and consider the following commutative diagram in $\Hqe$.
\begin{equation}\label{diag:restriction_of_dg_quotient}
\begin{tikzcd}[column sep=1.5cm]
    \N\arrow[hook]{r}\arrow[hook]{d}[swap]{F|_{\N}} & \A \arrow{r}{Q|_{\A}}\arrow[hook]{d}[swap]{F}&
    \Im(Q|_{\A}) \arrow[hook]{d}{\inc} \\
    \Tria_{\dg}\N \arrow[hook]{r}&
    \pretr(\A)\arrow{r}{Q} &
    \dfrac{\pretr(\A)}{\Tria_{\dg}\N}
\end{tikzcd}
\end{equation}
Then the following assertions hold.
\begin{enumerate}[label=\textup{(\arabic*)}]
\item 
$\inc\colon \Im(Q|_{\A})\hookrightarrow\frac{\pretr(\A)}{\Tria_{\dg}\N}$ is the pretriangulated hull of $\Im(Q|_{\A})$.
\item 
The restricted quasi-functor $Q|_{\A}\colon \A\to \Im(Q|_{\A})$ is a dg quotient of $\A$ by $\N$.
In particular, we have an isomorphism $\Im(Q|_\A)\cong \A/\N$ in $\Hqe$.
\item 
$F|_\N\colon \N\hookrightarrow\Tria_{\dg}\N$ is the pretriangulated hull of $\N$.
\end{enumerate}
\end{corollary}
\begin{proof}
The following claim is a crucial observation to proceed our proof.

\begin{claim}\label{claim:restriction_of_dg_quotient}
The composite $Q\circ F$ is a morphism to a pretriangulated dg category which is universal with respect  the property of annihilating $\N$.
\end{claim}
\begin{proof}
Let us consider a morphism $\A\xto{G} \C$ to a pretriangulated dg category which annihilates $\N$.
By the universality of the pretriangulated hull $\A\hookrightarrow\pretr(\A)$, we have a morphism $\pretr(\A)\xto{G'}\C$ with $G=G'\circ F$.
Note that $H^0G'$ is a triangle functor and thus $G'$ annihilates $\Tria_{\dg}\N$.
Again by the universality of the dg quotient $Q$, we have the following commutative diagram in $\Hqe$.
\begin{equation*}
\begin{tikzcd}[column sep=1.5cm, row sep=0.8cm]
    {}
    &\A \arrow{r}{G}\arrow[hook]{d}[swap]{F}
    &\C
    \\
    \Tria_{\dg}\N\arrow{r}{}
    &\pretr(\A)\arrow{r}[swap]{Q}\arrow{ru}[description]{G'}
    &\frac{\pretr(\A)}{\Tria_{\dg}\N}\arrow{u}[swap]{G''}
\end{tikzcd}
\end{equation*}
We still have to show the uniqueness of $G''$.
However, it is straightforward from the universality of $Q$ and $F$.
\end{proof}

(1)
We put $\B\deff\Im(Q|_{\A})$ until we finish the proof.
We consider the pretriangulated hull $\B\xto{G_0} \pretr(\B)$.
Applying \cref{claim:restriction_of_dg_quotient} to the morphism $G_0\circ Q|_{\A}$, we have the following diagram of solid arrows
\begin{equation*}
\begin{tikzcd}[column sep=1.5cm]
    {}
    &\A \arrow{r}{Q|_{\A}}\arrow[hook]{d}[swap]{F}
    &\B \arrow[hook]{d}{\inc}\arrow{r}{G_0}
    &\pretr(\B)\arrow[dotted, bend left]{ld}{G_2}
    \\
    \Tria_{\dg}\N\arrow{r}{}
    &\pretr(\A)\arrow{r}{Q}
    &\frac{\pretr(\A)}{\Tria_{\dg}\N}\arrow{ru}{G_1}
    &{}
\end{tikzcd}
\end{equation*}
with $G_0\circ Q|_{\A}=G_1\circ Q\circ F$.
Furthermore by the universality of $G_0$, we get the fully faithful dotted arrow $G_2$ with $G_2\circ G_0=\inc$.
It remains to show that $H^0G_2$ is essentially surjective.
By the aforementioned commutativity, we can easily check an equality $G_2\circ G_1\circ Q\circ F=Q\circ F$ is true.
Again the universality stated in \cref{claim:restriction_of_dg_quotient} guarantees the identity $G_2\circ G_1=\id$.
Thus, we have a desired isomorphism $G_2$ in $\Hqe$.

(2)
We shall check the needed universality of $Q|_{\A}\colon \A\to\B$.
Let us consider a morphism $\A\xto{F_0}\C$ in $\Hqe$ to a pretriangulated dg category $\C$ which annihilates $\N$, see \cref{lem:universality_of_dg_quotient}.
Thanks to \cref{claim:restriction_of_dg_quotient}, the arguments end up with the following diagram
\begin{equation*}
\begin{tikzcd}[column sep=1.5cm, row sep=0.8cm]
    {}
    &\C
    \\
    \A \arrow{r}[swap]{Q|_{\A}}\arrow[hook]{d}[swap]{F}\arrow[bend left = 20]{ru}{F_0}
    &\B \arrow[hook]{d}[swap]{\inc}\arrow{u}{F_1\circ\inc}
    \\
    \pretr(\A)\arrow{r}{Q}
    &\dfrac{\pretr(\A)}{\Tria_{\dg}\N}\arrow[bend right = 50]{uu}[swap]{F_1}
\end{tikzcd}
\end{equation*}
with $F_0=F_1\circ Q\circ F$.
Thus, we get a desired commutativity $F_0=(F_1\circ \inc)\circ Q|_{\A}$.
The uniqueness of $F_1\circ \inc$ is also ensured by \cref{claim:restriction_of_dg_quotient}.
Actually, if there is another such morphism $\B\xto{F_2}\C$ with $F_0=F_2\circ Q|_{\A}$, we have $\frac{\pretr(\A)}{\Tria_{\dg}\N}\xto{F'_1}\C$ with $F_2=F'_1\circ \inc$ by the universality of $\inc$.
The commutativity $F_0=F'_1\circ Q\circ F$ is immediate and results in $F_1=F'_1$.

(3)
If we consider a pretriangulated hull $\N\overset{F'}{\hookrightarrow}\pretr(\N)$, we get a fully faithful dg functor $\pretr(\N)\overset{\iota}{\hookrightarrow}\Tria_{\dg}\N$.
Since $\Tria(\CN)$ is the smallest triangulated subcategory of $\tr(\A)$ containing $\CN$, the induced functor $H^0\iota$ is a triangle equivalence which forces $\iota$ to be a quasi-equivalence.
\end{proof}

Consequently, any dg quotient $\A/\N$ admits the ambient dg quotient of their pretriangulated hulls:
\begin{equation*}
\begin{tikzcd}[column sep=1.5cm]
    \N\arrow[hook]{r}\arrow[hook]{d}[swap]{} & \A \arrow{r}{Q|_{\A}}\arrow[hook]{d}[swap]{}&
    \A/\N \arrow[hook]{d}{} \\
    \pretr(\N) \arrow[hook]{r}&
    \pretr(\A)\arrow{r}{Q} &
    \pretr(\A/\N)
\end{tikzcd}
\end{equation*}
where a natural isomorphism $\frac{\pretr(\A)}{\pretr(\N)}\cong \pretr(\A/\N)$ exists and the vertical arrows denote the pretriangulated hulls (cf. \cite[Thm.~3.4]{Dri04}).
The discussions so far are well-known for specialists.
We do not find a proper reference though and have thus included details.

\section{Exact dg categories}
\label{sec:exact_dg}
In this section \S\ref{sec:exact_dg}, we recall several aspects of exact dg categories following \cite{Che24a,Che24b}; see also \cite{Che23}.
In \S\ref{subsec:alg_ET}, we collect some basic facts on exact dg categories and briefly remind the reader that an exact dg category $(\A,\SS)$ gives rise to an algebraic extriangulated category $H^0\A$.
For a more comprehensive treatment of extriangulated categories, we refer the reader to \cite{NP19, HLN21, NOS22, GNP23}; see also \S\ref{subsec:extriangulated_quotient}.
In \S\ref{subsec:universal_embedding}, we review the key notion of universal embedding which provides the `universal' ambient triangulated category of the extriangulated category $H^0\A$ compatible with their higher extensions.
Some notions concerning higher extensions are recalled in \S\ref{subsec:higher_extensions}.
Last, we review the relative theory of the algebraic extriangulated category in \S\ref{subsec:relative_theory}.

\subsection{Algebraic extriangulated categories}
\label{subsec:alg_ET}
We begin with recalling basic notions and terminology to define exact dg categories.
Let $M\in\CC(k)$ be a complex of $k$-module and put the ($0$-th) \emph{truncation} as
\[
\tau_{\leq 0}M=(\cdots \to M^{-2}\to M^{-1}\to Z^0M\to 0\to \cdots),
\]
where $Z^0M$ is the kernel of the differential $M^0\to M^1$.
Since a dg category $\A$ is enriched by $\CC(k)$, we naturally define the \emph{truncation} $\tau_{\leq 0}\A$ to be the dg category consisting of the same objects as $\A$ with the morphisms given by
\[
(\tau_{\leq 0}\A)(A,B)=\tau_{\leq 0}(\A(A,B)).
\]
A dg category $\A$ is said to be \emph{connective} if the cohomology of every Hom-space $\A(A,B)$ vanishes in all positive degrees.
Accordingly, the canonical (not full) inclusion $\tau_{\leq 0}\A\to \A$ is called the \emph{connective cover} of $\A$ \cite[p.4]{Che24a}.
Note that this truncation $\tau_{\leq 0}$ does not change the homotopy category at all, in the sense that the canonical functor $H^0(\tau_{\leq 0}\A)\xto{\sim}H^0(\A)$ is an equivalence of additive categories (Recall that $\A$ is always assumed to be additive!).

The notion of $3$-term homotopy complexes ($3$-term h-complexes, for short) is the key ingredient to define the exact dg category.
The conceptual definition of them is originally given as a certain $A_\infty$-functor \cite[Def.~3.14]{Che24a}.
However, since we do not deal with the axioms of exact dg categories, we only provide a more elementary description that we can find in \cite[p.20]{Che24a}.

The \emph{$3$-term h-complex} $X$ is a diagram in $\A$
\begin{equation}\label{diag:3-term_h-complex}
\begin{tikzcd}[row sep=0.6cm]
A_0 \arrow{r}{f}\arrow[bend right]{rr}[swap]{h}& A_1\arrow{r}{g} & A_2,
\end{tikzcd}
\end{equation}
where $|f|=|g|=0, |h|=-1$ and $d(f)=0, d(g)=0$ and $d(h)=-gf$.
A morphism from $X$ to another $3$-term h-complex $X'$ is an equivalence class which can be explained by using the following diagram in $\A$,
\[
\begin{tikzcd}[row sep=0.8cm]
A_0 \arrow{r}{f}\arrow[bend left]{rr}{h}\arrow{d}[swap]{r_0}\arrow[draw=red]{rd}[swap]{\textcolor{red}{s_1}}\arrow[draw=blue]{rrd}{\textcolor{blue}{t}}
& A_1\arrow{r}{g}\arrow{d}[swap]{r_1} \arrow[draw=red]{rd}{\textcolor{red}{s_2}}
& A_2\arrow{d}{r_2},
\\
A'_0 \arrow{r}[swap]{f'}\arrow[bend right]{rr}[swap]{h'}
& A'_1\arrow{r}[swap]{g'} 
& A'_2,
\end{tikzcd}
\]
where $X'$ corresponds to the bottom one:
It consists of the following morphisms
\begin{align*}
&r_i\colon A_i\to A'_i, |r_i|=0, d(r_i)=0,\\
&s_j\colon A_{j-1}\to A'_j, |s_j|=-1, d(s_1)=f'r_0-r_1f, d(s_2)=g'r_1-r_2g
\end{align*}
and
\[
t\colon A_0\to A'_2, |t|=-2, d(t)=r_2\circ h-h'\circ r_0-s_2\circ f-g'\circ s_1.
\]
We denote by $\CH_{{\rm 3}t}(\A)$ the class of certain equivalent classes of $3$-term h-complexes in $\A$, cf. \cite[\S 3.2]{Che24a}.

\begin{remark}\label{rem:comparison_cone_and_h-cokernel}
If we consider the Yoneda embedding $\A\to\CC_{\dg}(\A)$, we notice that a $3$-term h-complex yields canonical morphisms in $\CC(\A)$
\[
u\deff\begin{bsmallmatrix}
h^{\wedge}\amph g^{\wedge}
\end{bsmallmatrix}\colon
C(f^{\wedge})\to A_2^{\wedge},\quad
v\deff\begin{bsmallmatrix}
f^{\wedge}\\
h^{\wedge}
\end{bsmallmatrix}\colon
A_0^{\wedge} \to C(g^{\wedge})[-1],
\]
due to the existence of $A_0\xto{h}A_2$ and $d(h)=-gf$.
Also, we notice that a $3$-term h-complex \eqref{diag:3-term_h-complex} can be regarded as a diagram in $\tau_{\leq 0}\A$.
\end{remark}

\begin{definition}\label{def:homtopy_short_exact}
A $3$-term h-complex \eqref{diag:3-term_h-complex} is \emph{homotopy left exact} if $\tau_{\leq 0}(v)$ is a quasi-isomorphism in $\CC(\tau_{\leq 0}\A)$.
Dually, we define the notion of \emph{homotopy right exact sequence}.
A both right and left homotopy h-complex is called a \emph{homotopy short exact sequence}.
\end{definition}

To define the exact structure on a dg category $\A$, dg versions of pushout and pullback are necessary.
These concepts are also defined in the obvious fashion such as the homotopy short exact sequence.
A \emph{homotopy square} \cite[Def.~3.21]{Che24a} is a diagram in $\A$
\begin{equation}\label{diag:homotopy_square}
\begin{tikzcd}[row sep=0.8cm]
A \arrow{r}{a}\arrow{d}[swap]{b}\arrow{rd}{h}
&B \arrow{d}{c}
\\
B' \arrow{r}[swap]{d}
& C
\end{tikzcd}
\end{equation}
where $a,b,c$ and $d$ belong to $Z^0(\A)$ and $h$ is a morphism of degree $-1$ such that $d(h)=ca-db$ and the folded sequence
\begin{equation}\label{diag:3-term_h-complex_folded}
\begin{tikzcd}[row sep=0.6cm]
A \arrow{r}{\begin{bsmallmatrix}
a\\
b
\end{bsmallmatrix}}\arrow[bend right]{rr}[swap]{h}& B\oplus B'\arrow{r}{
\begin{bsmallmatrix}
c \amph
-d
\end{bsmallmatrix}
} & C
\end{tikzcd}
\end{equation}
is a $3$-term h-complex.
The homotopy square \eqref{diag:homotopy_square} is a \emph{homotopy pullback square} if \eqref{diag:3-term_h-complex_folded} is homotopy left exact.
Dually, we define the notion of \emph{homotopy pushout square}.

Exact dg categories are a dg analogue of Quillen's exact category, cf. \cite[\S A.1]{Kel90} and \cite[\S 2]{Buh10}.
Let us, in advance, consider an isomorphism-closed subclass $\SS\sse\CH_{{\rm 3}t}(\A)$ consisting of homotopy short exact sequences.
Borrowing the terminology from Quillen's exact category, a sequence \eqref{diag:3-term_h-complex} belonging to $\SS$ is called a \emph{conflation} as well as the morphisms $f$ and $g$ are called an \emph{inflation} and a \emph{deflation}, respectively.

\begin{definition}\label{def:exact_dg_cat}
\cite[Def.~4.1]{Che24a}
An \emph{exact dg category} is defined to be the pair $(\A,\SS)$ which satisfies the following axioms:
\begin{enumerate}[label=\textup{(Ex\arabic*)}]
\setcounter{enumi}{-1}
\item \label{Ex0}
$\id_0$ is a deflation.
\item \label{Ex1}
The class of deflations are closed under composition.
\item \label{Ex2}
A diagram $B\xto{p}C\xleftarrow{c}C'$ in $Z^0(\A)$ with $p$ being a deflation is completed into a homotopy pullback
\[
\begin{tikzcd}[row sep=0.8cm]
B' \arrow{r}{p'}\arrow{d}[swap]{b}\arrow{rd}{s}
&C' \arrow{d}{c}
\\
B \arrow{r}[swap]{p}
& C
\end{tikzcd}
\]
where $p'$ is a deflation.
\end{enumerate}
\begin{enumerate}[label=\textup{(Ex\arabic*$^{\op}$)}]
\setcounter{enumi}{1}
\item \label{Ex2op}
A diagram $A'\xleftarrow{a}A\xto{i}B$ in $Z^0(\A)$ with $i$ being an inflation is completed into a homotopy pushout
\[
\begin{tikzcd}[row sep=0.8cm]
A \arrow{r}{i}\arrow{d}[swap]{a}\arrow{rd}{s}
&B \arrow{d}{b}
\\
A' \arrow{r}[swap]{i'}
& B'
\end{tikzcd}
\]
where $i'$ is a inflation.
\end{enumerate}
If the exact structure is clear from the context, we simply denote it by $\A$.
\end{definition}

The following ensures that $\Hqe$ is an appropriate framework when dealing with exact dg categories.

\begin{lemma}\label{lem:quasi-equiv_preserve_ex_dg_str}
\cite[Rem.~4.5(b)]{Che24a}
Let $F\colon \A\to\B$ be a quasi-equivalence of dg categories.
Then the natural map $\CH_{{\rm 3}t}(\A)\to \CH_{{\rm 3}t}(\B)$ preserves and reflects the property of being homotopy short exact sequences.
Moreover, $F$ induces a bijection between exact dg structures on $\A$ and those on $\B$.
\end{lemma}

Thanks to \cref{lem:quasi-equiv_preserve_ex_dg_str}, the following notion of exact morphisms between exact dg categories is well-defined, that is, it does not depend on the choice of representatives for a given morphism in $\Hqe$:

\begin{definition}\label{def:exact_morphisms_in_Hqe}
\cite[Def.~4.3]{Che24a}
Let us consider a morphism $F\colon (\A,\SS)\to (\A',\SS')$ between exact dg categories.
It is said to be \emph{exact} if the induced functor $\CH_{{\rm 3}t}(\A)\to \CH_{{\rm 3}t}(\A')$ restricts to $\SS\to\SS'$.
We denote by $\Hqe_{\rm ex}$ the subcategory of $\Hqe$ consisting of small exact dg categories and exact morphisms.
Also, an exact morphism $F$ is called an \emph{exact quasi-equivalence} if it is an isomorphism in $\Hqe_{\rm ex}$.
\end{definition}

As the following examples show, an exact dg category is a unification of Quillen's exact category and the pretriangulated dg category.

\begin{example}\label{ex:Quillen_exact_cat_is_exact_dg}
\cite[Ex.~4.6]{Che24a}
Let $\A$ be an additive category and regard it as a dg category concentrated in degree $0$.
Then, an exact dg structure on $\A$ defines the Quillen's exact structure on $\A$ and vice versa.
Also, an exact functor in \cref{def:exact_morphisms_in_Hqe} performs in the usual sense.
\end{example}
\begin{proof}
Since any morphism in $\A$ of degree $-1$ is zero, the axioms in \cref{def:exact_dg_cat} is nothing other than that of Quillen's exact category.
Thus the assertion holds.
\end{proof}

\begin{example}\label{ex:pretri_is_exact_dg}
\cite[Ex.~4.7]{Che24a}
A pretriangulated dg category $\A$ admits a maximal exact dg structure, that is, putting $\SS$ to be the class of all homotopy short exact sequences, we have the corresponding exact dg structure $(\A,\SS)$.
\end{example}
\begin{proof}
To clarify what the homotopy short exact sequences are like, we include a detailed discussion.
Thanks to \cref{rem:quasi-equivalence}, it suffices to show that $\A'\deff\pretr(\A)$ is an exact dg category with respect to the class $\SS'$ of homotopy short exact sequences.
We shall use the following easy observation.

\begin{claim}\label{claim:pretri_is_exact_dg}
Let us consider a $3$-term h-complex in $\A'$ as below.
Then, it is homotopy right exact if and only if it is homotopy left exact.
\begin{equation}\label{seq:claim:pretri_is_exact_dg}
\begin{tikzcd}[row sep=0.6cm]
A \arrow{r}{f}\arrow[bend right]{rr}[swap]{h}& B\arrow{r}{g} & C
\end{tikzcd}
\end{equation}
\end{claim}
\begin{proof}
As we have seen in \cref{rem:comparison_cone_and_h-cokernel}, a natural morphism $v\deff\begin{bsmallmatrix}
f^{\wedge}\\
h^{\wedge}
\end{bsmallmatrix}\colon
A^{\wedge} \to C(f^{\wedge})[-1]$ in $\CC(\A')$ exists.
If \eqref{seq:claim:pretri_is_exact_dg} is homotopy left exact, by \cref{rem:quasi-equivalence}, we see $v$ is a quasi-isomorphism.
Since \eqref{seq:claim:pretri_is_exact_dg} thus falls into a triangle in $\CD(\A')$, another natural morphism $u\colon
C(f^{\wedge})\to C^{\wedge}$ in \cref{rem:comparison_cone_and_h-cokernel} is a quasi-isomorphism as well, which proves \eqref{seq:claim:pretri_is_exact_dg} is a homotopy short exact sequence.
\end{proof}

We will see that any morphism $B\xto{g}C$ in $\A'$ is a deflation.
Recall that the inclusion $\A'\hookrightarrow \CC_{\dg}(\A')$ preserves taking (co)cones and $H^0\A'$ is a triangulated subcategory of $\CD(\A')$.
Since the canonical diagram
\[
\begin{tikzcd}[row sep=0.6cm]
C(g)[-1] \arrow{r}{p}\arrow[bend right]{rr}[swap]{
\begin{bsmallmatrix}
0\\
\id_C
\end{bsmallmatrix}
}& B\arrow{r}{g} & C,
\end{tikzcd}
\]
in $\A'$ is a homotopy short exact sequence by \cref{claim:pretri_is_exact_dg}, we see $g$ is a deflation.
Hence the conditions \ref{Ex0} and \ref{Ex1} are obvious.
To show \ref{Ex2}, we consider a cospan $B\xto{p}C\xleftarrow{c}C'$ in $Z^0\A'$ of morphisms of degree $0$.
Taking the cocone of the induced morphism $B\oplus C'\xto{
\begin{bsmallmatrix}
c \amph
-d
\end{bsmallmatrix}
} C$, we get a desired homotopy pullback such as the diagram in \ref{Ex2}.
Besides, the obtained morphism $p$ is a deflation as we have mentioned.
We can proceed the dual argument for \ref{Ex2op}.

Thanks to \cref{lem:quasi-equiv_preserve_ex_dg_str}, it turns out that $\A$ is an exact dg category with respect to the class $\SS$ of homotopy short exact sequences.
\end{proof}

An exact dg category $(\A,\SS)$ provides an enhancement of an extriangulated category in the following sense.

\begin{definition}\label{def:alg_ET}
\cite[Thm.~4.26]{Che24a}
An exact dg category $(\A,\SS)$ produces a natural extriangulated category $(H^0\A,\BE_\SS,\fs_\SS)$.
If an extriangulated category $(\CC,\BE,\fs)$ is exact equivalent to some $(H^0\A,\BE_\SS,\fs_\SS)\in\ET$, it is called an \emph{algebraic} extriangulated category.
Also, $(\A,\SS)$ is called an \emph{exact dg enhancement} of $(\CC,\BE,\fs)$.
\end{definition}

\begin{remark}\label{rem:alg_ET_interpretation}
The extriangulated structure $(H^0\A,\BE_\SS,\fs_\SS)$ can be interpreted via the universal embedding as stated in \cite[Prop-Def.~3.14]{Che23}, see \cref{thm:universal_embedding}.
Such an interpretation is useful and enough for our purpose, so we do not recall the constructions of the original $\BE_\SS$ and $\fs_\SS$.
\end{remark}

As \cref{ex:pretri_is_exact_dg} shows, we should mention that the notion of the exact dg enhancement can be traced back to Bondal-Kapranov's notion of enhanced triangulated category \cite{BK90}.

In closing the subsection, we review an exact dg structure is stable under several categorical operations: 
taking the connective cover; considering extension-closed subcategories; passing to a substructure using relative theory. 
We now describe only the parts of the first two basic operations, and the last will be treated separately in relevance for our intension, see \S\ref{subsec:relative_theory}.

\begin{example}\label{ex:connective_cover}
Let $(\A,\SS)$ be an exact dg category.
Then, the connective cover admits a natural exact dg structure, we will write it as $(\tau_{\leq 0}\A,\SS)$ and still call it the \emph{connective cover} of $(\A,\SS)$.
Moreover, we can easily see the identity $H^0(\A)=H^0(\tau_{\leq 0}\A)$ as extriangulated categories, because the axioms \ref{Ex0}--\ref{Ex2op} of exact dg structure do not refer to any morphism of positive degree.
\end{example}

\begin{example}\label{ex:extension_closed_subcat}
\cite[Ex.-Def.~4.8]{Che24a}
A full dg subcategory $\B$ of an exact dg category $(\A,\SS)$ is \emph{extension-closed} if we have $B\in\B$ for any conflation \eqref{diag:3-term_h-complex} in $\A$ with $A,C\in\B$.
Let us consider the subclass $\SS|_\B\sse\SS$ consisting of conflations whose all terms belong to $\B$. 
Then the pair $(\B,\SS|_\B)$ forms an exact dg category the canonical inclusion from which is exact.
\end{example}

\subsection{The universal embedding}
\label{subsec:universal_embedding}
By using Neeman's notion of bounded derived category $\Db(\CA)$ of an exact category $\CA$ \cite{Nee90}, we have an embedding $\CA\to\Db(\CA)$ which is exact in $\ET$.
In this subsection, we review Chen's notion of universal embedding which brings Neeman's embedding to the dg world and shows its universality as the name shows.

Thanks to \cref{lem:quasi-equiv_preserve_ex_dg_str}, we know any algebraic extriangulated category $\CA$ admits a cofibrant dg category $\A$ such that $\CA$ is exact equivalent to $H^0\A$.
Thus we will work under the following setup which is fundamental to study an algebraic extriangulated category in terms of the exact dg enhancement.

\begin{setup}\label{setup:universal_embedding}
Let $(\A,\SS)$ be an exact dg category which is cofibrant.
\end{setup}

We have to begin with recalling the construction of a pretriangulated dg category $\Db_{\dg}(\A,\SS)$.
Let us consider a conflation below.
\begin{equation}\label{diag:conflation_for_univ_emb}
\begin{tikzcd}[row sep=0.6cm]
A \arrow{r}{f}\arrow[bend right]{rr}[swap]{h}& B\arrow{r}{g} & C
\end{tikzcd}
\end{equation}

For the readability purpose, we omit $\wedge$ for representable dg module $A^\wedge$.
As mentioned in \cref{rem:comparison_cone_and_h-cokernel}, the conflation \eqref{diag:conflation_for_univ_emb} induces a morphism $\begin{bsmallmatrix}
-h\amph g
\end{bsmallmatrix}\colon \cone(f)\to C$ and set $M\deff \cone\begin{bsmallmatrix}
-h\amph g
\end{bsmallmatrix}$ in $\CC_{\dg}(\A)$ as below.

\begin{equation}\label{diag:definition_of_deffect}
\begin{tikzcd}[row sep=0.8cm]
A\arrow{r}{f}
&B\arrow{r}{\begin{bsmallmatrix}
0\\
1
\end{bsmallmatrix}}
&U \arrow{r}{\begin{bsmallmatrix}
1\amph 0
\end{bsmallmatrix}}\arrow{d}[swap]{\begin{bsmallmatrix}
-h\amph g
\end{bsmallmatrix}}
&A[1]
\\
{}
&{}
&C
\arrow{d}[swap]{\begin{bsmallmatrix}
0\\
0\\
1
\end{bsmallmatrix}}
&
\\
{}
&{}
&M
&{}
\end{tikzcd}
\end{equation}

Regarding the object $M\in\CC_{\dg}(\A)$ above as in $\CD(\A)$, we call it a \emph{defective} object.
Since the row and column of \eqref{diag:definition_of_deffect} correspond to triangles in $\CD(\A)$, we thus consider the full triangulated subcategory $\CM$ of $\tr(\A)$ generated by the defective objects $M$.
In addition, $\CM_{\dg}$ is the canonical dg enhancement of $\CM$, namely, the full dg subcategory of $\pretr(\A)$ consisting of the objects in $\CM$.

\begin{definition}\label{def:defective_object}
Let $Q\colon \pretr(\A)\to\pretr(\A)/\CM_{\dg}$ be the dg quotient. Putting $\Db_{\dg}(\A,\SS)\deff \pretr(\A)/\CM_{\dg}$, we call it the \emph{bounded dg derived category} of $(\A,\SS)$.
Also, it falls into the Verdier quotient $H^0(Q)\colon \tr(\A)\to\Db(\A,\SS)$ under the functor $H^0$, where we put $\Db(\A,\SS)\deff H^0(\Db_{\dg}(\A,\SS))$ and call it the \emph{bounded derived category} of $H^0\A$ accordingly.
\end{definition}

Since the specified exact dg structure should be clear from the context, we often dorp the symbol $\SS$ for simplicity, that is, the symbols $\Db_{\dg}(\A)$ and $\Db(\A)$ will be used instead of $\Db_{\dg}(\A,\SS)$ and $\Db(\A,\SS)$ respectively.

\begin{theorem}\label{thm:universal_embedding}
\cite[Thm.~3.1]{Che24b}
Let $(\A,\SS)$ be an exact dg category and consider the canonical morphism $F\colon \A\to\Db_{\dg}(\A)$ in $\Hqe$.
Then, it is a universal exact morphism from $\A$ to a pretriangulated category.

Furthermore, if $\A$ is connective, it induces an exact quasi-equivalence $\tau_{\leq 0}\A\to \tau_{\leq 0}\D'$ for an extension-closed dg subcategory $\D'\sse\Db_{\dg}(\A)$. 
In particular, we have an exact fully faithful embedding $H^0\A\to\Db(\A)$ whose essential image is extension-closed and restricts to the exact equivalence in $\ET$.
\end{theorem}

The above constructed morphism $F\colon \A\to\Db_{\dg}(\A)$ in $\Hqe$ is called the \emph{universal embedding} of $\A$ with respect to the exact dg structure $\SS$.
Since $\A$ is assumed to be cofibrant, the morphism $F$ is taken to be a dg functor.
Note that $F$ is not necessarily fully faithful though.
To understand what it is like, we observe the two extremal cases.

\begin{example}\label{ex:univ_emb_exact_cat}
Let $(\A,\SS)$ be an exact dg category concentrated in degree $0$.
Then, the universal embedding induces Neeman's embedding $H^0\A\to\Db(H^0\A)=\Db(\A)$.
\end{example}
\begin{proof}
We may regard $\A=H^0\A$ as a Quillen's exact category due to \cref{ex:Quillen_exact_cat_is_exact_dg}.
It is well-known that $\tr(\A)$ is the usual homotopy category $\Hb(\A)$ of bounded complexes in $H^0\A$.
Note that the subcategory $\CM$ is generated by conflations which are regarded as complexes in $H^0\A$.
Thus, $\CM$ is noting other than the full subcategory $\Hb_{\rm ac}(\A)$ of acyclic complexes in $\Hb(\A)$.
The associated Verdier quotient $\Hb_{\rm ac}(\A)\to \Hb(\A)\to \Db(\A)$ shows the assertion.
\end{proof}

\begin{example}\label{ex:univ_emb_pretr}
Let $\A$ be a pretriangulated dg category and consider its connective cover $\tau_{\leq 0}\A\to\A$.
Then, we have an isomorphism $\pretr(\tau_{\leq 0}\A)\xto{\simeq} \Db_{\dg}(\tau_{\leq 0}\A)$ in $\Hqe$ and hence triangle equivalences $H^0\A\xleftarrow{\sim}\tr(\tau_{\leq 0}\A)\xto{\sim}\Db(\tau_{\leq 0}\A)$.
\end{example}
\begin{proof}
Since $\A$ is pretriangulated, for any conflation \eqref{diag:conflation_for_univ_emb} in $\tau_{\leq 0}\A$, the induced morphism $\begin{bsmallmatrix}
-h\amph g
\end{bsmallmatrix}$ in $\CC(\tau_{\leq 0}\A)$ such as \eqref{diag:definition_of_deffect} is a homotopy equivalence, which shows that the subcategory $\CM$ generated by defective objects is zero.
Hence the canonical dg quotient $\pretr(\tau_{\leq 0}\A)\xto{Q} \Db_{\dg}(\tau_{\leq 0}\A)$ is an isomorphism in $\Hqe$.
The latter assertion is  immediate from the fact that $\A$ is pretriangulated.
\end{proof}

Due to the universal embedding, any algebraic extriangulated category $H^0\A$ is an extension-closed subcategory of the ambient triangulated category $\Db(\A)$.
Thus, $H^0\A$ possesses an extriangulated structure inherited from $\Db(\A)$ and that of \cref{def:alg_ET}.
Both extriangulated structures are exact equivalent by \cref{thm:universal_embedding}.

In addition, as a consequence of the existence of the universal embedding, we recall the following characterizations for an extriangulated category to be algebraic.

\begin{proposition}\label{prop:characterization_for_alg_ET}
\cite[Prop.-Def.~3.14]{Che24b}
Let $(\CA,\BE,\fs)$ be an extriangulated category.
Then the following are equivalent.
\begin{enumerate}[label=\textup{(\arabic*)}]
\item 
$\CA$ is exact equivalent to an extension-closed subcategory of an algebraic triangulated category.
\item 
$\CA$ is exact equivalent to $\CB/[\CP]$, the ideal quotient of a Quillen's exact category $\CB$ by a class $\CP$ of projective-injective objects.
\item 
$\CA$ is exact equivalent to $(H^0\A,\BE_\SS,\fs_\SS)$ for an exact dg category $(\A,\SS)$.
\end{enumerate}
\end{proposition}

\subsection{Higher extensions}
\label{subsec:higher_extensions}
The viewpoint of realizing $(H^0\A,\BE_\SS,\fs_\SS)$ as an extension-closed subcategory of $\Db(\A,\SS)$ gives us a great advantage to deal with the higher extensions of $H^0\A$.

We recall some basics of the higher extensions over an (arbitrary) extriangulated category $(\CA,\BE,\fs)$.
The \emph{higher extension} $\BE^n(?,-)$ is defined as the $n$-th tensor power of $\BE$ for any $n\geq 1$ \cite[Def.~3.1]{GNP23}.
We remark that, if $\CA$ has enough projectives and enough injectives, it is isomorphic to that in \cite[\S 5.1]{LN19} by \cite[Cor.~3.21]{GNP23}.
The aim of this subsection is to explain that the higher extensions over $H^0\A$ is interpreted as the shifted $\Hom$-spaces in $\Db(\A)$.
To this end, we have to recall the notion of $\delta$-functors
 over an extriangulated category $(\CA,\BE,\fs)$.

\begin{definition}\label{def:delta_functor}
\cite[Def.~4.5]{GNP23}\cite[Def.~2.5]{Che24b}
Let $(T,\epsilon)$ be the pair of
\begin{itemize}
\item 
$T=\{T^i\}_{i\geq 0}$: a sequence of right $\CA$-modules; and
\item 
$\epsilon$: a collection of \emph{connecting morphisms} $\epsilon^i_{\delta}\colon T^i(A)\to T^{i+1}(C)$ for each $\delta\in\BE(C,A)$ and $i\geq 0$,
\end{itemize}
which is natural with respect to morphisms of $\BE$-extensions.
It is called a \emph{right $\delta$-functor} if, for any $\fs$-triangle $A\overset{f}{\lra} B\overset{g}{\lra} C\overset{\delta}{\dra} $, the associated sequence
\[
\cdots\to T^n(C)\xto{T^n(g)}T^n(B)\xto{T^n(f)}T^n(A)\xto{\epsilon^n_{\delta}}T^{n+1}(C)\to\cdots
\]
is exact.
A \emph{morphism} $\theta\colon (T,\epsilon)\to (R,\eta)$ of right $\delta$-functors is a family of morphisms of right $\CA$-modules $\theta^i\colon T^i\to R^i$ which is compatible with the connecting morphisms, i.e., for each $\delta\in\BE(C,A)$ and each $i\geq 0$, the following diagram is commutative.
\[
\begin{tikzcd}
T^i(A)\arrow{r}{\epsilon^i_{\delta}}\arrow{d}[swap]{\theta^i(A)} & T^{i+1}(C)\arrow{d}{\theta^{i+1}(C)}\\
R^i(A)\arrow{r}{{\eta}^i_{\delta}} &R^{i+1}(C)
\end{tikzcd}
\]
Also, the notions of \emph{left $\delta$-functors} and \emph{morphisms} amongst them are defined dually.

A \emph{$\delta$-functor} is defined to be the triplet $(T,\epsilon,\eta)$ where $T=\{T^i\}_{i\geq 0}$ is a sequence of $\CA$-$\CA$-bimodules such that
\begin{itemize}
\item $(T(?,X),\epsilon)$: a right $\delta$-functor; and
\item $(T(X,-),\eta)$: a left $\delta$-functor,
\end{itemize}
for each object $X\in\CA$.
A \emph{morphism} $\theta\colon (T,\epsilon,\eta)\to (R,\epsilon',\eta')$ is a family of morphisms of $\CA$-$\CA$-bimodules $\theta^i\colon T^i\to R^i$ which is compatible with the connecting morphisms.
\end{definition}

We need more terminology of effaceability to put the notation of $\delta$-functors to use.
It was introduced for abelian categories by Grothendieck \cite[p.141]{Gro57} and was placed in extriangulated context \cite{Eno21,Oga21}.

\begin{definition}\label{def:effaceable}
A right $\CA$-module $F\colon \CA^{\op}\to\Ab$ is \emph{effaceable} if, for any $C\in\CA$ and $x\in F(C)$, there exists a deflation $B\xto{p} C$ such that $F(p)(x)=0$.
Dually, we define the notion of effaceable left $\CA$-module.

In addition, an $\CA$-$\CA$-bimodule $\BF\colon \CA^{\op}\otimes \CA\to\Ab$ is \emph{effaceable} if both $\BF(?,A)$ and $\BF(A,-)$ are effaceable for any $A\in\CA$.
\end{definition}

As below, the effaceablity determines the higher part of (right) $\delta$-functors in a sense.

\begin{proposition}\label{prop:effaceablity_determines_the_higher_degrees}
\cite[Prop.~2.4]{Che24b}
Let $(T,\epsilon)$ and $(R,\eta)$ be right $\delta$-functors such that $T^i$ is effaceable for each $i>0$.
Then each morphism $\theta^0\colon T^0\to R^0$ extends uniquely to a morphism of right $\delta$-functors $\theta\colon (T,\epsilon)\to (R,\eta)$.
\end{proposition}

It is obvious that the defining bifunctor $\BE(?,-)$ of an extriangulated category $\CA$ is effaceable, so are the higher ones $\BE^n(?,-)$ for any $n>0$ by the construction.
Moreover, putting $\BE^0(?,-)=\Hom_{\CA}(?,-)$, we see that the sequence $\{\BE^i(?,-)\}_{i\geq 0}$ gives rise to a $\delta$-functor together with the obvious connecting morphisms.
To be precise, we make a space to confirm the assertion.

\begin{lemma}\label{lem:higher_extensions_are_effaceable}
The sequence $\{\BE^i(?,-)\}_{i\geq 0}$ of bifunctors forms a $\delta$-functor and each $\BE^i(?,-)$ is effaceable for any $i>0$.
\end{lemma}
\begin{proof}
Let $\BF(?,-)$ be an effaceable $\CA$-$\CA$-bimodule.
We shall show that the tensor product $\BF^2(?,X)\deff\BF\otimes_{\CA}\BF(?,X)$ is an effaceable right $\CA$-module as well.
An element $\delta\in \BF^2(C,X)$ is represented by the finite sum $\sum x_i\otimes y_i$ of some $x_i\in\BF(C,Y_i)$ and $y_i\in\BF(Y_i,X)$.
Then, considering the corresponding element $x\in\BF(C,\coprod Y_i)$, we get a deflation $B\xto{p}C$ such that $\BF(p)(x)$=0.
We also have a desired equation $\BF^2(p,X)(\sum x_i\otimes y_i)=\sum px_i\otimes y_i=0$.
Combining the dual argument, we have checked that $\BE^i(?,-)$ is effaceable.
\end{proof}

We have mentioned in \cref{prop:characterization_for_alg_ET} that any algebraic extriangulated category $\CA$ is realized as an extension-closed subcategory of an algebraic triangulated category $\CT$.
In this case, the shifted $\Hom$-space $\Ext^i_{\CT}(?,-)=\Hom_{\CT}(?,-[i])$ is a bifunctor over $\CA$.
If we employ the ambient triangulated category $\CT$ as the universal embedding $\CA\sse\Db(\A)$, the shifted $\Hom$-spaces are effaceable as well.
Precisely, we have the following assertion which will be a key ingredient in \S\ref{sec:subquotient}.

\begin{proposition}\label{prop:higher_extension_via_univ_emb}
\cite[Prop.~3.17]{Che24b}
Let us assume that $\A$ is connective and consider the universal embedding $F\colon \A\to\Db(\A)$.
Then the bifunctor $\Ext^i_{\Db(\A)}(?,-)$ over $\CA$ is a effaceable $\delta$-functor with a coincidence $\Hom_{\CA}(?,-)=\Ext^0_{\Db(\A)}(?,-)$.
Furthermore, we have a canonical isomorphism of $\delta$-functors $\theta^i\colon \BE^i(?,-)\xto{\simeq}\Ext^i_{\Db(\A)}(?,-)$ for any $i\geq 0$.
\end{proposition}

We would like to emphasize that passing to the connective cover is an essential step in ensuring that $\Db(\A)$ behaves compatibly with higher extensions.
Needless to say, the incarnation of \cref{prop:higher_extension_via_univ_emb} is a usual abelian category $\CA$ together with the embedding $\CA\hookrightarrow\Db(\CA)$ to its bounded derived category.
As is well-known, the group $\Ext^i_\CA(A,B)$ is understood as the shifted $\Hom$-space $\Hom_{\Db(\CA)}(A,B[i])$.

\subsection{Exact dg substructures and extriangulated substructures}
\label{subsec:relative_theory}
We review a remarkable interaction between substructures of an exact dg category $(\A,\SS)$ and those of the associated extriangulated category $\CA=(H^0\A,\BE,\fs)$, following \cite[\S 4.4]{Che24a}.
They bijectively correspond to each other even in connection with their poset structures.

Recall from \cite[\S 5.1]{INP24} that an additive subbifunctor $\BF\sse\BE$ is said to be \emph{closed} if $(\CA,\BE,\fs)$ satisfies (one of) the following equivalent conditions:
\begin{itemize}
\item 
$(\CA,\BF,\fs|_{\BF})$ is an extriangulated category;
\item 
$\fs|_\BF$-inflations are closed under composition; and
\item 
$\fs|_\BF$-deflations are closed under composition.
\end{itemize}
This may permit us to place some arguments in \cite{AS93,DRSSK99} in extriangulated contexts.
We call the above $(\CA,\BF,\fs|_{\BF})$ an \emph{extriangulated subcategory/substructure} of $(\CA,\BE,\fs)$.
We should mention that our notion of extriangulated subcategory is slightly stronger than that in \cite[Def.~3.7]{Hau21}.

Since the intersections of (possibly infinitely many) closed subfunctors is still closed \cite[Cor.~3.14]{HLN21}, the class of closed subfunctors of $\BE$ forms a poset with respect to the canonical inclusion.
Similarly, exact substructures of $(\A,\SS)$ form a poset given by inclusion as well.
What we want to review in this subsection is this.

\begin{proposition}\label{prop:bijection_between_substructures}
We have a poset isomorphism between the following classes:
\begin{enumerate}[label=\textup{(\arabic*)}]
\item 
The poset of exact substructures of $(\A,\SS)$; and
\item 
The poset of closed subbifunctors of $\BE$.
\end{enumerate}
\end{proposition}

As witnessing in \cref{prop:characterization_for_alg_ET} and \cref{prop:bijection_between_substructures}, we know that the algebraic extriangulated categories are closed under taking extension-closed subcategories and extriangulated substructures.

The following closed subfunctors will play an important role in \S\ref{subsec:extriangulated_subquotient}.
This generalizes a way to construct Quillen's exact substructure on Quillen's exact category $\CA$.

\begin{lemma}\label{lem:Quillen_substructure}
\cite[Lem. A.2]{Che23}
Let $(\CA,\BE,\fs)$ be an arbitrary extriangulated category.
For any objects $A,C\in\CA$, we consider subsets of $\BE(C,A)$:
\begin{enumerate}[label=\textup{(\arabic*)}]
\item 
$\displaystyle \BE_{[\CN]}(C,A)\ = \sum_{f\colon C\to N, N\in\CN}\Im(\BE(N,A)\to\BE(C,A))$; and
\item 
$\displaystyle \BE^{[\CN]}(C,A)\ = \sum_{f\colon N\to A, N\in\CN}\Im(\BE(C,N)\to\BE(C,A))$.
\end{enumerate}
Then,  both $\BE_{[\CN]}$ and $\BE^{[\CN]}$ give rise to \emph{closed} bifunctors of $\BE$.
The associated extriangulated categories are denoted by $(\CA,\BE_{[\CN]},\fs_{[\CN]})$ and $(\CA,\BE^{[\CN]},\fs^{[\CN]})$, respectively.
\end{lemma}

These substructures possess a useful stability that will be used in \S\ref{subsec:extriangulated_subquotient}.

\begin{lemma}\label{lem:stability_of_Quillen_substructure}
Consider the extriangulated subcategories $(\CA,\BE_{[\CN]},\fs_{[\CN]})$ and $(\CA,\BE^{[\CN]},\fs^{[\CN]})$.
Then we have equalities $\BE_{[\CN]}=(\BE_{[\CN]})_{[\CN]}$ and $\BE^{[\CN]}=(\BE^{[\CN]})^{[\CN]}$.
\end{lemma}
\begin{proof}
We place the argument only on the first equality.
To show the inclusion $\BE_{[\CN]}\sse(\BE_{[\CN]})_{[\CN]}$, we consider an element $\delta\in\BE_{[\CN]}(C,A)$ corresponding to an $\fs_{[\CN]}$-conflation $A\overset{f}{\lra} B\overset{g}{\lra} C\overset{\delta}{\dra} $.
By definition, there exists a finite set of morphisms $C\xto{h_i}N_i$ with $N_i\in\CN$ and $\eta_i\in\BE(N_i,A)$ such that $\sum_{i} \BE(h_i,A)(\eta_i)=\delta$.
Since $\eta_i\in\BE_{[\CN]}(N_i,A)$ is obvious, we can rewrite $\sum_{i} \BE_{[\CN]}(h_i,A)(\eta_i)=\delta$ showing $\delta\in(\BE_{[\CN]})_{[\CN]}(C,A)$.
\end{proof}

\section{Extriangulated categories}
\label{sec:ET}
Throughout the section, let $\CA=(\CA,\BE,\fs)$ be a skeletally small extriangulated category.
We begin by recalling two types of localizations of $\CA$ with respect to an extension-closed subcategory $\CN\sse\CA$.
The first operation, reviewed in \S\ref{subsec:extriangulated_quotient}, is the extriangulated quotient \cite{NOS22} which exists under a certain condition.
It provides an exact functor $Q\colon \CA\to \CA/\CN$ with an appropriate universality and unifies the Serre/Verdier quotients of abelian/triangulated categories.
The subsection \S\ref{subsec:auxiliary_exact_sequence} serves as an auxiliary preparation to the subsequent discussions.
In \S\ref{subsec:extriangulated_subquotient}, we introduce the second operation called the extriangulated subquotient which---by contrast---always exists.
Finally in \S\ref{subsec:subquotients_are_algebraic}, we verify that algebraic extriangulated categories are stable under the subquotients.

Although we refer the reader to \cite[\S 2]{NP19} for details of extriangulated category, we record only a few lemmas to proceed with our discussions.
An \emph{extriangulated category} is defined to be the triplet $(\CA,\BE,\fs)$ of
\begin{itemize}
\item an additive category $\CA$;
\item an (bi-)additive bifunctor $\BE\colon \CA^{\op}\times\CA\to \Ab$, where $\Ab$ is the category of abelian groups;
\item a \emph{realization} $\fs$, that is, it associates each equivalence class of  a sequence  $A\overset{f}{\lra} B\overset{g}{\lra} C$ in $\CA$ to an element in $\BE(C,A)$ for any $C,A\in\CA$,
\end{itemize}
which satisfies some axioms laid out in \cite[Def.~2.12]{NP19}.
An element $\delta\in\BE(C,A)$ with an associated sequence is denoted by $A\overset{f}{\lra} B\overset{g}{\lra} C\overset{\delta}{\dra}$ called an \emph{$\fs$-triangle}.
Borrowing the terminology from \cite{Kel90}, we call the sequence $A\overset{f}{\lra} B\overset{g}{\lra} C$ an \emph{$\fs$-conflation}.
The morphisms $f$ and $g$ appearing there are called an \emph{$\fs$-inflation} and an \emph{$\fs$-deflation}, respectively.
The next lemma will be used in many places. 
It says that, like in the triangulated case, extriangulated categories admit weak pushouts and weak pullbacks.

\begin{lemma}\label{lem:wPO_wPB}
\cite[Prop.~1.20]{LN19}
The following properties hold.
\begin{enumerate}[label=\textup{(\arabic*)}]
\item\label{item:wPO}
For any $\fs$-triangle $A\overset{f}\lra B\overset{g}{\lra} C\overset{\delta}{\dra} $ together with a morphism $a\in\CA(A, A')$ and an associated $\fs$-triangle $A'\overset{f'}\lra B'\overset{g'}{\lra} C\overset{a_{*}\delta}{\dra} $,
there exists $b\in\CA(B,B')$ which gives a morphism of $\fs$-triangles 
\[
\begin{tikzcd}[row sep=0.6cm]
    A 
        \arrow{r}{f}
        \arrow{d}[swap]{a}
        \wPO{dr}
    & B 
        \arrow{r}{g}
        \arrow{d}{b}
    & C 
        \arrow[dashed]{r}{\delta}
        \arrow[equals]{d}{}
    & {}\\
    A' 
        \arrow{r}{f'}
    & B'
        \arrow{r}{g'}
    & C 
        \arrow[dashed]{r}{a_{*}\delta}
    & {}
\end{tikzcd}
\]
and makes 
$
\begin{tikzcd}[column sep=1.1cm, cramped]
A
    \arrow{r}{\begin{bsmallmatrix}
        f\\a
    \end{bsmallmatrix}}
& B\oplus A'
    \arrow{r}{\begin{bsmallmatrix}
        b \amph -f'
    \end{bsmallmatrix}}
& B'
    \arrow[dashed]{r}{(g')^{*}\delta}
&{}
\end{tikzcd}
$
an $\fs$-triangle.
We call the square ${\rm (wPO)}$ is a \emph{weak pushout} of $f$ along $a$.
\item\label{item:wPB}
For any $\fs$-triangle $A\overset{f}\lra B\overset{g}{\lra} C\overset{\delta}{\dra} $ together with a morphism $c\in\CA(C',C)$ and an associated $\fs$-triangle $A\overset{f'}\lra B'\overset{g'}{\lra} C'\overset{c^{*}\delta}{\dra} $,
there exists $b\in\CA(B',B)$ which gives a morphism of $\fs$-triangles
\[
\begin{tikzcd}[row sep=0.6cm]
    A 
        \arrow{r}{f'}
        \arrow[equals]{d}{}
    & B' 
        \arrow{r}{g'}
        \arrow{d}[swap]{b}
        \wPB{dr}
    & C' 
        \arrow[dashed]{r}{c^{*}\delta}
        \arrow{d}{c}
    & {}\\
    A 
        \arrow{r}{f}
    & B
        \arrow{r}{g}
    & C 
        \arrow[dashed]{r}{\delta}
    & {}
\end{tikzcd}
\]
and makes 
$
\begin{tikzcd}[column sep=1.1cm, cramped]
B'
    \arrow{r}{\begin{bsmallmatrix}
        -g'\\b
    \end{bsmallmatrix}}
& C'\oplus B
    \arrow{r}{\begin{bsmallmatrix}
        c \amph g
    \end{bsmallmatrix}}
& C
    \arrow[dashed]{r}{f'_{*}\delta}
&{}
\end{tikzcd}
$
an $\fs$-triangle.
We call the square ${\rm (wPB)}$ a \emph{weak pullback} of $g$ along $c$.
\end{enumerate}
\end{lemma}

Let us recall \ref{ET4} which is a part of the original axioms for $(\CA,\BE,\fs)$ in \cite[Def.~2.12]{NP19}.
\begin{enumerate}[label=\textup{(ET\arabic*)}]
\setcounter{enumi}{3}
\item\label{ET4}
Let $A\overset{f}{\lra}B\overset{f'}{\lra}D\overset{\delta}{\dra}$ and $B\overset{g}{\lra}C\overset{g'}{\lra}F\overset{\rho}{\dra}$ be $\fs$-triangles in $(\CA,\BE,\fs)$. 
Then there exists a commutative diagram made of $\fs$-triangles in $\C$,
\[
\begin{tikzcd}[column sep=1.2cm, row sep=0.6cm]
    A \arrow{r}{f}\arrow[equals]{d}{}& B \arrow{r}{f'}\arrow{d}{g}& D \arrow[dashed]{r}{\delta}\arrow{d}{d}& {} \\
    A \arrow{r}[swap]{h}& C\arrow{r}[swap]{h'}\arrow{d}{g'} & E \arrow[dashed]{r}{\delta'}\arrow{d}{e}& {} \\
      & F \arrow[equals]{r}{}\arrow[dashed]{d}{\rho}& F,\arrow[dashed]{d}{f'_{*}\rho} & {} \\
      & {}  &{}   &
\end{tikzcd}
\]
where 
$A\overset{h}{\lra}C\overset{h'}{\lra}E\overset{\delta'}{\dra}$ and $D\overset{d}{\lra}E\overset{e}{\lra}F\overset{f'_{*}\rho}{\dra}$ are $\fs$-triangles, 
and
$d^{*}\delta' = \delta$ and $e^{*}\rho=f_{*}\delta'$. 
\end{enumerate}

\subsection{The extriangulated quotient}
\label{subsec:extriangulated_quotient}
We recall the localization of extriangulated categories with respect to a suitable class of morphisms, whose study was initiated by Nakaoka-Ogawa-Sakai~\cite{NOS22}.
This construction encompasses the Verdier/Serre quotients of triangulated/abelian categories, as well as other ``quotient-like'' operations, see \cref{ex:extri_quotient_by_biresolving}.

An \emph{exact} functor $(F,\phi)\colon (\CA,\BE,\fs)\to (\CA',\BE',\fs')$ between extriangulated categories is the pair consisting of an additive functor $F\colon \CA\to\CA'$ and a natural transformation $\phi\colon \BE\rightarrow\BE'\circ(F^{\op}\times F)$ which satisfies
\[
\mathfrak{s}'(\phi_{C,A}(\delta))=[F(A)\overset{F(x)}{\lra}F(B)\overset{F(y)}{\lra}F(C)]
\]
for any $\fs$-triangle $A\overset{x}{\lra}B\overset{y}{\lra}C\overset{\delta}{\dashrightarrow}$ in $\CA$, see \cite[Def.~2.32]{B-TS21}.

\begin{lemma}\label{lem:exact_equivalence}
\cite[Prop.~2.13]{NOS22}
An exact functor $(F,\phi)\colon (\CA,\mathbb{E},\mathfrak{s})\to(\CA',\mathbb{E}',\mathfrak{s}')$ is called an \emph{exact equivalence}, if the following equivalent conditions are satisfied:
\begin{enumerate}[label=\textup{(\arabic*)}]
\item
There exists an exact functor $(G,\psi)\colon (\CA',\mathbb{E}',\mathfrak{s}')\to(\CA,\mathbb{E},\mathfrak{s})$ and a natural isomorphisms $\eta\colon (G,\psi)\circ (F,\phi)\cong\id_{\CA}$ and $\eta'\colon (F,\phi)\circ(G,\psi)\cong\id_{\CA'}$;
\item
$F$ is an equivalence of categories and $\phi$ is a natural isomorphism.
\end{enumerate}  
\end{lemma}

We recall the notion of exact sequences in the category $\ET$ of skeletally small extriangulated categories and exact functors.

\begin{definition}\label{def:exact_seq_ET}
\cite[Def.~4.10]{OS24b}
We consider a morphism $(\CA,\BE,\fs)\xto{(Q,\mu)}(\CB,\BF,\ft)$ in $\ET$ together with
an extension-closed subcategory $(\CN,\BE|_\CN,\fs|_\CN)\sse(\CA,\BE,\fs)$.
We say that the sequence $\Theta\colon$
\begin{equation}\label{eqn:exact-seq-in-ET}
\begin{tikzcd}[column sep=1.2cm]
(\CN,\BE|_\CN,\fs|_\CN)
	\arrow{r}{(\inc,\iota)}
& (\CA,\BE,\fs)
	\arrow{r}{(Q,\mu)}
& (\CB,\BF,\ft)
\end{tikzcd}
\end{equation}
in $\ET$ is an \emph{exact sequence} of extriangulated categories, 
if the following conditions are fulfilled.
\begin{enumerate}[label=\textup{(\arabic*)}]
    \item\label{item:image-equals-kernel}
    $\CN = \Im(\inc) = \Ker Q$ holds.

    \item\label{item:universality}
    For any exact functor 
    $(G,\psi)\colon (\CA,\BE,\fs)\to (\CA',\BE',\fs')$ 
    in $\ET$ with $G|_\CN = G \circ \inc = 0$, 
    there is a unique exact functor
    $(G',\psi')\colon (\CB,\BF,\ft) \to (\CA',\BE',\fs')$ 
    such that $(G,\psi)=(G',\psi')\circ (Q,\mu)$.
	
\end{enumerate}
Let us consider another exact sequence $\Theta'$ which appears in the first row of the following commutative diagram in $\ET$.
\begin{equation*}
\begin{tikzcd}[column sep=1.5cm]
    (\CN',\BE'|_{\CN'},\fs'|_{\CN'}) \arrow{r}\arrow{d}{}& (\CA',\BE',\fs') \arrow{r}{(Q',\mu')}\arrow{d}&
    (\CB',\BF',\ft') \arrow{d} \\
    (\CN,\BE|_\CN,\fs|_\CN) \arrow{r}&
    (\CA,\BE,\fs)\arrow{r}{(Q,\mu)} &
    (\CB,\BF,\ft)
\end{tikzcd}
\end{equation*}
If the above vertical arrows are fully faithful, then we call it an \emph{exact subsequence} of \eqref{eqn:exact-seq-in-ET} and we write $\Theta'\sse\Theta$ to present the situation.
Also, exact (sub)sequences in $\TR$, $\EX$ and $\AB$ are defined similarly.
As one expects, the Serre/Verdier quotients give rise to exact sequences in $\ET$ \cite[\S 4.2]{NOS22}.
\end{definition}

Now we are going to review a construction of exact sequences.
Let us fix a class $\CS$ of morphisms in $\CA$ satisfying the following condition.

\begin{enumerate}[label=\textup{(M\arabic*)}]
\setcounter{enumi}{-1}
\item\label{M0}
$\CS$ contains all isomorphisms in $\CA$, and is closed under compositions and taking finite direct sums.
\end{enumerate}

We would like to assume from the beginning that $\CA$ is skeletally small to consider the Gabriel-Zisman localization $\widetilde{\CA}=\CA[\CS^{-1}]$ in the sense of \cite{GZ67} without set-theoretic obstructions.
We shall recall a sufficient condition for $\CS$ to impose a natural extriangulated structure on the localization $\widetilde{\CA}$.

\begin{definition}
\label{def:null_system_extri_quotient}
We associate a full subcategory $\CN_\CS\sse\CA$ to the above class $\CS$ which consists of objects $N\in\CA$ such that both $N\to 0$ and $0\to N$ belong to $\CS$.
\end{definition}

Since $\CS$ is closed under direct sums, $\CN_\CS$ is an additive subcategory of $\CA$.
In the rest of this subsection, we denote by $p\colon \CA\to\ovl{\CA}=\CA/[\CN_\CS]$ the ideal quotient as well as $\ovl{f}=p(f)$ for any $f\in\Mor\CA$.
Also, let $\ovl{\CS}$ be the closure of $p(\CS)$ with respect to compositions with isomorphisms in $\ovl{\CA}$.

Recall from \cite[Lem.~3.2, Rem.~3.3]{NOS22} that if $\CA$ is \emph{weakly idempotent complete} in the sense that any section admits a cokernel in $\CA$, we have $\ovl\CS=p(\CS)$.
It is shown in \cite[Thm.~C, Prop.~2.7]{Kla22} that the {\rm (WIC)} property for $(\CA,\BE,\fs)$ in \cite[Condition~5.8]{NP19} is equivalent to weakly idempotent completeness for the underlying additive category $\CA$.
Thus, we prefer to say {\rm (WIC)} than weakly idempotent complete because of its shortness.

We now provide a construction of exact sequences in $\ET$, which will play a fundamental role in what follows.

\begin{theorem}\label{thm:mult_loc}
\cite[Thm. 3.5]{NOS22}
Let $(\CA,\BE,\fs)$ be an extriangulated category and $\CS$ a class of morphisms satisfying \ref{M0}.
\begin{enumerate}[label=\textup{(\arabic*)}]
\item
Suppose that $\CS$ satisfies the following conditions \ref{MR1}--\ref{MR4} and consider the localization $\widetilde{\CA}=\CA[\CS^{-1}]$.
Then we obtain an extriangulated category $(\wtil{\CA},\wtil{\BE},\wtil{\fs})$ together with an exact functor $(Q,\mu)\colon (\CA,\BE,\fs)\to (\wtil{\CA},\wtil{\BE},\wtil{\fs})$.
\begin{enumerate}[label=\textup{(MR\arabic*)}]
\setcounter{enumi}{0}
\item\label{MR1}
$\ovl{\CS}$ satisfies the $2$-out-of-$3$ property with respect to compositions.
\item\label{MR2}
$\ovl{\CS}$ is a multiplicative system in $\ovl{\CA}$.
\item\label{MR3}
Let $\langle A\overset{x}{\lra}B\overset{y}{\lra}C,\delta\rangle$, $\langle A'\overset{x'}{\lra}B'\overset{y'}{\lra}C',\delta'\rangle$ be any pair of $\fs$-triangles, and let $a\in\CA(A,A'),c\in\CA(C,C')$ be any pair of morphisms satisfying $a_*\delta=c^*\delta'$. If $\ovl{a}$ and $\ovl{c}$ belong to $\ovl{\CS}$, then there exists $\ovl{b}\in\ovl{\CS}(B,B')$ which satisfies $\ovl{b}\circ\ovl{x}=\ovl{x}'\circ\ovl{a}$ and $\ovl{c}\circ\ovl{y}=\ovl{y}'\circ\ovl{b}$.
\item\label{MR4} $\ovl{\CM}_{\mathsf{inf}}:=\{ \ovl{v}\circ \ovl{x}\circ \ovl{u}\mid x\ \text{is an}\ \mathfrak{s}\text{-inflation}, \ovl{u},\ovl{v}\in\ovl{\CS}\}$ is closed under composition in $\ovl{\CA}$.
Dually, $\ovl{\CM}_{\mathsf{def}}:=\{ \ovl{v}\circ \ovl{y}\circ \ovl{u}\mid y\ \text{is an}\ \mathfrak{s}\text{-deflation}, \ovl{u},\ovl{v}\in\ovl{\CS}\}$ is closed under compositions.
\end{enumerate}
\setcounter{enumi}{1}
\item
If $\ovl{\CS}=p(\CS)$, we have an equality $\CN_\CS=\Ker Q$.
\item
The exact functor $(Q,\mu)\colon (\CA,\BE,\fs)\to (\wtil{\CA},\wtil{\BE},\wtil{\fs})$ obtained in {\rm (1)} has the following universality:
for any exact functor $(F,\phi)\colon (\CA,\BE,\fs)\to (\CA',\BE',\fs')$ such that $F(s)$ is an isomorphism for any $s\in\CS$, there exists a unique exact functor $(\wtil{F},\wtil{\phi})\colon (\wtil{\CA},\wtil{\BE},\wtil{\fs})\to (\CA',\BE',\fs')$ with $(F,\phi)=(\wtil{F},\wtil{\phi})\circ (Q,\mu)$.
\end{enumerate}
The quotient functor  $(Q,\mu)\colon(\CA,\BE,\fs)\to(\wtil{\CA},\wtil{\BE},\wtil{\fs})$is called an \emph{extriangulated localization} of $\CA$ at $\CS$.
\end{theorem}
\begin{proof}
(1)(3) These assertions are contained in \cite[Thm.~3.5]{NOS22}.

(2) For any $N\in\CN_\CS$ together with a zero map $0\to N$ in $\CS$, we have an isomorphism $0\xto{\simeq}QN$.
Conversely, for any $N\in\Ker Q$, both zero maps $0\xto{a} N$ and $N\xto{b} 0$ are sent to isomorphisms in $\ovl{\CA}$.
By the assumption $\ovl{\CS}=p(\CS)$, we see $a,b\in\CS$ and $N\in\CN_\CS$.
\end{proof}

The following lemma gives useful descriptions of $\widetilde{\mathfrak{s}}$-conflations.

\begin{lemma}
\label{lem:inf_from_inf}
\cite[Lem. 3.32]{NOS22}
The following holds for any morphism $\alpha$ in $(\widetilde{\CA},\widetilde{\mathbb{E}},\widetilde{\mathfrak{s}})$.
\begin{enumerate}[label=\textup{(\arabic*)}]
\item $\alpha$ is an $\widetilde{\mathfrak{s}}$-inflation in $\widetilde{\CA}$ if and only if $\alpha=s\circ Q(f)\circ t$ holds for some $\mathfrak{s}$-inflation $f$ in $\CA$ and isomorphisms $s,t$ in $\widetilde{\CA}$.
\item $\beta$ is an $\widetilde{\mathfrak{s}}$-deflation in $\widetilde{\CA}$ if and only if $\beta=s\circ Q(f)\circ t$ holds for some $\mathfrak{s}$-deflation $f$ in $\CA$ and isomorphisms $s,t$ in $\widetilde{\CA}$.
\end{enumerate}
\end{lemma}

Next, we review when the pair $(\CA,\CN)$ of an extriangulated category $\CA$ and a subcategory $\CN\sse\CA$ yields an extriangulated localization.
Note that the kernel of an exact functor $F\colon \CA\to\CB$ in $\ET$ is a thick subcategory in the following sense.

\begin{definition}
\label{def:thick}
A full subcategory $\CN$ of $\CA$ is called a \emph{thick} subcategory if it satisfies the following conditions:
\begin{itemize}
\item
$\CN$ is closed under direct summands and isomorphisms; and
\item
$\CN$ satisfies $2$-out-of-$3$ for $\mathfrak{s}$-conflations. Namely, for any $\mathfrak{s}$-conflation $A\lra B\lra C$, if two of $\{A, B, C\}$ belong to $\CN$, then so does the third.
\end{itemize}
In addition, a thick subcategory $\CN$ is said to be \emph{biresolving}, if, for any object $C\in\CA$, there exist an $\mathfrak{s}$-inflation $C\to N$ and an $\mathfrak{s}$-deflation $N'\to C$ with $N,N'\in\CN$, cf. \cite[p.~403]{Rum21}.
\end{definition}

We associate a class $\Sn$ of morphisms to a thick subcategory $\CN$ as below.

\begin{definition}
\label{def:Sn_from_thick}
For a thick subcategory $\CN$, we consider the following classes
of morphisms:
\begin{enumerate}[label=\textup{(\arabic*)}]
\item
$\Inf_{\CN} = \{f\in\Mor\C \mid f \ \textnormal{is an $\mathfrak{s}$-inflation with}\ \cone(f)\in\CN\}$; and 
\item
$\Def_{\CN} = \{g\in\Mor\C \mid g \ \textnormal{is an $\mathfrak{s}$-deflation with}\ \cocone(g)\in\CN\}$.
\end{enumerate}
Define $\Sn$ to be the smallest subclass closed under compositions containing both $\Inf_\CN$ and $\Def_\CN$.
It is obvious that $\Sn$ satisfies the condition \ref{M0}.
Note that, since $\CN$ is closed under extensions, all the above classes are closed under compositions.
\end{definition}

As a direct consequence of Theorem \ref{thm:mult_loc}, if $\Sn$ satisfies the needed conditions, we obtain an extriangulated localization.

\begin{corollary}
\label{cor:mult_loc}
Let $(\CA,\CN)$ be a pair of an extriangulated category $(\CA,\mathbb{E},\mathfrak{s})$ and a thick subcategory $\CN$ of $\CA$.
If the class $\Sn$ satisfies the conditions \ref{MR1}--\ref{MR4} and $\ovl{\Sn}=p(\Sn)$, then, we have an exact sequence
\begin{equation}\label{seq:mult_loc}
(\CN,\mathbb{E}|_\CN,\mathfrak{s}|_\CN)\lra (\CA,\mathbb{E},\mathfrak{s})\overset{(Q,\mu)}{\lra} (\CA/\CN,\widetilde{\mathbb{E}},\widetilde{\mathfrak{s}})
\end{equation}
in $\ET$.
\end{corollary}
\begin{proof}
The extriangulated localization exists by \cref{thm:mult_loc}.
It remains to show that the sequence \eqref{seq:mult_loc} is an exact sequence in $\ET$.
As we see $\Ker Q=\CN$ from \cref{thm:mult_loc}(2) and \cite[Lem. 4.5]{NOS22}, we have to only check the universality of $(Q,\mu)$.
It actually follows from \cite[Thm.~3.8]{OS24a} and \cite[Prop.~4.3]{ES22}.
\end{proof}

\begin{definition}\label{def:extriangulated_quotient}
Let us consider the exact sequence $\eqref{seq:mult_loc}$ coming from the pair $(\CA,\CN)$ as in \cref{cor:mult_loc}.
Then we call the quotient functor $(Q,\mu)\colon (\CA,\BE,\fs)\to (\CA/\CN,\wtil{\BE},\wtil{\fs})$ the \emph{extriangulated quotient} (or just \emph{quotient}) of $\CA$ by $\CN$, and the pair $(\CA,\CN)$ is said to be \emph{admissible} to the extriangulated quotient.
\end{definition}

We can find some examples of admissible pairs $(\CA,\CN)$ in \cite[\S 4.2]{NOS22}.
For later use, we only recall one of them, that is, any biresolving subcategory $\CN$ forms an admissible pair $(\CA,\CN)$.

\begin{example}\label{ex:extri_quotient_by_biresolving}
\cite[Prop.~4.26, Cor.~4.27]{NOS22}
Let $\CN\sse\CA$ be a biresolving subcategory.
Then, we have an exact sequence \eqref{seq:mult_loc} such that the resulting category $\CA/\CN$ corresponds to a triangulated category.
We should remark that, in the case that $\CA$ is an exact category, this result reduces to \cite[Thm.~5]{Rum21}.
\end{example}

A prototypical example of biresolving subcategories is the subcategory $\CN$ of projectives in a Frobenius exact category $\CA$.
In fact, the associated extriangulated quotient \eqref{seq:mult_loc} coincides with the ideal quotient $\CN\lra\CA\overset{p}{\lra} \CA/[\CN]$, see \cite[Rem.~3.35]{NOS22} and the extriangulated quotient $\CA/\CN=\CA/[\CN]$ is the stable triangulated category.
Also, we should notice that a thick subcategory of a triangulated category is biresolving.
Thus \cref{ex:extri_quotient_by_biresolving} contains the Verdier quotient.

\subsection{Auxiliary exact sequences}
\label{subsec:auxiliary_exact_sequence}
In this preliminary subsection, we develop the extriangulated quotient which matches well with our context.
More precisely, we show that an extriangulated quotient restricts to an exact subsequence under a certain technical condition.
The condition is stated below and will be justified in view of our notion of extriangulated subquotient introduced in \S\ref{subsec:extriangulated_subquotient}.

\begin{setup}
\label{setup:spade}
Let $(\CA,\BE,\fs)$ be an extriangulated category and consider an extension-closed subcategory $\CN\sse\CA$ which is closed under direct summands.
In addition, we assume $\CA$ is ${\rm (WIC)}$ and the pair $(\CA,\CN)$ satisfies the following extra condition:
\begin{enumerate}[label=\textup{($\spadesuit$)}]
\item \label{spade}
For any $\fs$-inflation $N_1\overset{f}{\lra} B$ from $N_1\in\CN$, there exists a morphism $B\xto{b} N_2$ such that $b\circ f$ is an $\fs|_\CN$-inflation.
Dually, for any $\fs$-deflation $B\overset{g}{\lra} N_3$ to $N_3\in\CN$, there exists a morphism $N_2\xto{b} B$ such that $g\circ b$ is an $\fs|_\CN$-deflation.
\end{enumerate}
\end{setup}

This condition makes the pair $(\CA,\CN)$ to be admissible.
So the extriangulated quotient under \ref{spade} always exists.

\begin{proposition}\label{prop:extriangulated_quotient_under_spade}
We have an exact sequence \eqref{seq:mult_loc}
\[
\Theta \colon 
\begin{tikzcd}
    (\CN ,\BE|_\CN ,\fs|_\CN )
        \arrow{r}
    & (\CA ,\BE ,\fs )
        \arrow{r}{(Q,\mu)}
    & (\CB,\BF ,\ft)
\end{tikzcd}
\]
in $\ET$,
where we put $(\CB,\BF ,\ft)= (\CA/\CN,\wtil{\BE},\wtil{\fs})$.
\end{proposition}

Because of the lengthy proof, we divide it in several steps.
We have to check the thickness of $\CN\sse\CA$.

\begin{lemma}\label{lem:spade_implies_thickness}
$\CN$ is a thick subcategory of $\CA$.
\end{lemma}
\begin{proof}
We consider an $\fs$-triangle $N_1\overset{f}{\lra} N_2\overset{g}{\lra} C\overset{\delta}{\dra} $ with $N_1,N_2\in\CN$.
By \ref{spade}, we get a morphism $N_2\xto{f'}N'_2$ such that the composite $f'\circ f$ is an $\fs|_{\CN}$-inflation.
In a few moment, we have a morphism of $\fs$-triangles as follows.
\[
\begin{tikzcd}[column sep=1.0cm]
    N_1 \arrow{r}{f}\arrow[equal]{d}{}& N_2 \arrow{r}{g}\arrow{d}[swap]{f'}\wPB{rd}&
    C \arrow{d}{}\arrow[dashed]{r}{\delta} &
    {}\\
    N_1 \arrow{r}&
    N'_2\arrow{r}{} &
    N_3\arrow[dashed]{r} &
    {}
\end{tikzcd}
\]
Thanks to \cref{lem:wPO_wPB}, we can think of the folded sequence of ${\rm (wPB)}$ is an $\fs$-triangle by replacing $f'$ if necessary.
Since $\CN\sse\CA$ is closed under extensions and direct summands, $N_3,N_2\in\CN$ implies $C\oplus N'_2\in\CN$ and $C\in\CN$.
We omit the dual part of the proof.
\end{proof}

Following \cref{def:Sn_from_thick}, we consider the defining classes $\Inf_\CN, \Def_\CN$ and $\Sn$ associated with the pair $(\CA,\CN)$.
Note that $\ovl{\Sn}\sse\Mor\ovl{\CA}$ coincides with $p(\Sn)$ owing to $\rm (WIC)$ of $\CA$.
As a preparation for verifying the conditions \ref{MR1}--\ref{MR4}, we reveal the following useful factorization of $\Sn$ which serves as a source of the good behavior of the extriangulated quotient under \ref{spade}.
We denote by $\Inf_\CN^{\sp}$ the class of sections in $\Inf_\CN$ and by $\Def_\CN^{\sp}$ the class of retractions in $\Def_\CN$.

\begin{lemma}\label{lem:factorization_of_Sn}
We have $\Sn=\Def_\CN^{\sp}\circ\Inf_\CN=\Def_\CN\circ\Inf_\CN^{\sp}$.
\end{lemma}
\begin{proof}
To show $\Def_\CN\sse \Def_\CN^{\sp}\circ\Inf_\CN$, consider an $\fs$-triangle $N_1\overset{f}{\lra} B\overset{s}{\lra} C\overset{\delta}{\dra} $ with $N_1\in\CN$.
The condition \ref{spade} shows the existence of a morphism $B\xto{f'}N_2$ such that the composite $f'\circ f$ is an $\fs|_{\CN}$-inflation.
Thus, we have a morphism of $\fs$-triangles,
\[
\begin{tikzcd}[column sep=1.0cm]
    N_1 \arrow{r}{f}\arrow[equal]{d}{}& B \arrow{r}{s}\arrow{d}[swap]{f'}\wPB{rd}&
    C \arrow{d}{}\arrow[dashed]{r}{\delta} &
    {}\\
    N_1 \arrow{r}&
    N_2\arrow{r}{} &
    N_3\arrow[dashed]{r} &
    {}
\end{tikzcd}
\]
the square {\rm (wPB)} of which yields the folded $\fs$-triangle by \cref{lem:wPO_wPB}.
Note that $\begin{bsmallmatrix}s \\
f'
\end{bsmallmatrix}\in\Inf_\CN$ by $N_3\in\CN$.
Hence we have a desired factorization
$
s\colon B\overset{\begin{bsmallmatrix}s \\
f'
\end{bsmallmatrix}}\lra C\oplus N_2\overset{\begin{bsmallmatrix}\id_C & 0
\end{bsmallmatrix}}{\lra} C
$.
We next show $\Inf_\CN\circ \Def_\CN\sse\Def_\CN\circ\Inf_\CN$ by using the containment that we have just verified.
In addition to the above morphism $B\xto{s}C$, we consider a morphism $C\xto{t}D$ in $\Inf_{\CN}$.
By the former argument, the composite $t\circ s$ admits a factorization as follows:
\[
t\circ s= \bigl[
B\overset{\begin{bsmallmatrix}s \\
f'
\end{bsmallmatrix}}\lra C\oplus N\overset{\begin{bsmallmatrix}\id_C & 0
\end{bsmallmatrix}}{\lra} C\overset{t}{\lra} 
D \bigr] 
=
\bigl[
B\overset{\begin{bsmallmatrix}s \\
f'
\end{bsmallmatrix}}\lra C\oplus N\overset{\begin{bsmallmatrix}
t & 0\\
0 & \id_N
\end{bsmallmatrix}}{\lra} D\oplus N\overset{\begin{bsmallmatrix}\id_D & 0
\end{bsmallmatrix}}{\lra} 
D 
\bigr],
\]
which shows $\Inf_\CN\circ\Def_\CN\sse \Inf_\CN\circ\Def_\CN^{\sp}\circ\Inf_\CN\sse \Def_\CN^{\sp}\circ\Inf_\CN$.
Lastly, by combining the dual argument, we have $\Inf_\CN\circ\Def_\CN=\Def_\CN\circ\Inf_\CN$.

Since $\Sn$ consists of the finite compositions of morphisms in $\Inf_\CN\cup \Def_\CN$, we have $\Sn=\Def_\CN\circ\Inf_\CN=\Def_\CN^{\sp}\circ\Inf_\CN$.
The second equality is also true.
\end{proof}

\begin{proof}[Proof of \cref{prop:extriangulated_quotient_under_spade}]
Now we proceed to verify the conditions \ref{MR1}--\ref{MR4}.

\ref{MR1}:
To check the $2$-out-of-$3$ for $\Sn$, let us consider a sequence $A\xto{s}A'\xto{t}B$ in $\CA$ and assume $s,ts\in\Sn$.
We want to explain that we may employ a more strong assumption of $s, ts\in\Inf_\CN$ from the beginning:
Thanks to a factorization of $ts=\begin{bsmallmatrix}\id_{B}\amph 0 
\end{bsmallmatrix}\circ\begin{bsmallmatrix}ts \\
f 
\end{bsmallmatrix}$ by $\Sn=\Def_\CN^{\sp}\circ \Inf_\CN
$,
we can construct the following commutative diagram with $N\in\CN$.
\[
\begin{tikzcd}[column sep=0.7cm]
    &A'\oplus N\arrow{rd}{
    \begin{bsmallmatrix}t \amph 0 \\
    0 \amph \id_{N}
    \end{bsmallmatrix}
    }&&\\
    A\arrow{ru}{
    \begin{bsmallmatrix}s \\
    f
    \end{bsmallmatrix}
    }\arrow{rr}{
    \begin{bsmallmatrix}ts \\
    f
    \end{bsmallmatrix}
    }&&B\oplus N\arrow{r}{
    \begin{bsmallmatrix}\id_B \amph 0
    \end{bsmallmatrix}
    }&B ,
\end{tikzcd}
\]
Due to ${\rm (WIC)}$ of $\CA$, we know $\begin{bsmallmatrix}s \\
f
\end{bsmallmatrix}\in\Inf_\CN$ from $\begin{bsmallmatrix}ts \\
f
\end{bsmallmatrix}\in\Inf_\CN$.

Now, we return to a discussion on the sequence $A\xto{s}A'\xto{t}B$ with $s,ts\in\Inf_\CN$.
Applying \cite[dual of Prop.~3.15]{NP19} to both $\fs$-inflations starting from $A$ to get the following commutative diagram of $\fs$-triangles
\[
\begin{tikzcd}[column sep=1.0cm]
    A \wPO{rd}\arrow{r}{ts}\arrow{d}[swap]{s}& B \arrow{r}{}\arrow{d}{b}&
    N \arrow[equal]{d}{}\arrow[dashed]{r}{} &
    {}\\
    A' \arrow{r}{s'}\arrow{d}{}&
    B'\arrow{r}{}\arrow{d}{} &
    N\arrow[dashed]{r} &
    {}\\
    N'\arrow[equal]{r}\arrow[dashed]{d}{}&N'\arrow[dashed]{d}{}&&\\
    {}&{}&&
\end{tikzcd}
\]
with $N,N'\in\CN$.
Since $b$ splits, we have the corresponding retraction $B'\xto{b'}B$.
We can easily check that $\ovl{t-b's'}=0$ in $\ovl{\CA}$.
By $s'\in\Inf_\CN$ and $b'\in\Def_\CN^{\sp}$, we see $t$ still sits in $\Def_\CN^{\sp}\circ\Inf_\CN=\Sn$.
The other condition can be checked dually.

\ref{MR2}:
We confirm that $\ovl{\Sn}$ is a multiplicative system.
First, we consider a diagram $B\xleftarrow{f}A\xto{s}A'$ in $\CA$ with $s\in\Sn$.
A factorization $s\colon A\xto{s_1}A''\xto{s_2}A'$ exists by $\Sn=\Def_\CN^{\sp}\circ \Inf_\CN$ in \cref{lem:factorization_of_Sn}.
Taking a weak pushout of the $\fs$-inflation $s_1$ along $f$, we get the commutative square ${\rm (wPO)}$:
\[
\begin{tikzcd}[column sep=1.0cm]
    A \wPO{rd}\arrow{r}{s_1}\arrow{d}[swap]{f}& A'' \arrow{r}\arrow{d}{f'}&
    N \arrow[equal]{d}\arrow[dashed]{r} &
    {}\\
    B \arrow{r}{s'}&
    B'\arrow{r}{} &
    N\arrow[dashed]{r} &
    {}
\end{tikzcd}
\]
as a part of a morphism of $\fs$-triangles, where $N\in\CN$. 
This gives rise to a desired square, since $f'\circ s'_2\circ s = f'\circ s'_2\circ s_2\circ s_1 =s'\circ f$ holds in $\ovl{\CA}$ where $s'_2$ denote the section corresponding to $s_2\in \Def_\CN^{\sp}$.

Next, we consider a sequence $A\xto{f}B\xto{s}B'$ with $\ovl{s\circ f}=0$ in $\ovl{\CA}$.
Take a factorization $s=s_2\circ s_1$ by $\Def_\CN\circ\Inf_\CN^{\sp}$ similarly to the former argument. 
We consider the commutative square below on the left, in which $N\in\CN$.
\[
\begin{tikzcd}[column sep=1.0cm]
    A \arrow{r}{a}\arrow{d}[swap]{s_1f}& N \arrow{d}{b}\\
    B'' \arrow{r}{s_2}&
    B' ,
\end{tikzcd}
\qquad
\qquad
\begin{tikzcd}[column sep=1.0cm]
    &A\arrow[dotted]{d}{}\arrow[bend left]{rd}{a}\arrow[bend right = 50]{dd}[swap, yshift=2.5mm]{s_1f}&&
    \\
    N' \arrow[crossing over]{r}{}\arrow[equal]{d}& N'' \arrow{r}\arrow{d}{}&
    N \arrow{d}{b}\arrow[dashed]{r} &
    {}
    \\
    N' \arrow{r}{}&
    B''\arrow{r}{s_2} &
    B'\arrow[dashed]{r} \wPB{ul}&
    {}
\end{tikzcd}
\]
By taking a weak pullback of $s_2$ along $b$, we have the right  and the dotted arrow by the property of ${\rm (wPB)}$ which shows $s_1f$ factors through $N''\in\CN$.
Since $s_1$ is a section having the cokernel in $\CN$, we see $\ovl{f}=0$ in $\ovl{\CA}$.
The remaining conditions that we have to check are just the dual, so we omit them.

\ref{MR3}: We consider a morphism of $\fs$-triangles with $\ovl{a},\ovl{b}\in\ovl{\Sn}$ as in \ref{MR3} and factorize it as follows by \cite[Prop.~2.3]{Eno21}.
\[
\begin{tikzcd}[column sep=1.0cm]
    A \wPO{rd}\arrow{r}{f}\arrow{d}[swap]{a}& B \arrow{r}{g}\arrow{d}{b_1}&
    C \arrow[equal]{d}\arrow[dashed]{r}{\delta} &
    {}
    \\
    A' \arrow{r}{f''}\arrow[equal]{d}&
    B''\arrow{r}{g''}\arrow{d}[swap]{b_2} &
    C\arrow[dashed]{r}{\delta''}\arrow{d}{c} &
    {}
    \\
    A' \arrow{r}{f'}&
    B'\arrow{r}{g'} &
    C'\wPB{lu}\arrow[dashed]{r}{\delta'} &
    {}
\end{tikzcd}
\]
We will focus only on the upper morphism and show that $b_1$ can be replaced by a morphism in $\Sn$, because the needed argument for $b_2\in\Sn$ is dual.
Consider a factorization $a\colon A\xto{a_1}A''\xto{a_2}A'$ with $a_1\in\Inf_\CN, a_2\in\Def_\CN^{\sp}$ by \cref{lem:factorization_of_Sn}.
To replace the morphism $(a,b_1,\id_C)$,
we take accurate weak pushouts of $f$ along the factorization in succession as explained below:
\[
\begin{tikzcd}[column sep=1.0cm]
    A \wPO{rd}\arrow{r}{f}\arrow{d}[swap]{a_1}& B \arrow{r}{g}\arrow{d}{b'}&
    C \arrow[equal]{d}\arrow[dashed]{r}{\delta} &
    {}
    \\
    A'' \wPO{rd}\arrow{r}{f_0}\arrow{d}[swap]{a_2}&
    B_0\arrow{r}{g_0}\arrow{d}{b''} &
    C\arrow[dashed]{r}{\delta_0}\arrow[equal]{d} &
    {}
    \\
    A' \arrow{r}{f''}&
    B''\arrow{r}{g''} &
    C\arrow[dashed]{r}{\delta''} &
    {}
\end{tikzcd}
\]
At the first step to constitute the upper morphism, we have taken it such that $b'$ is an $\fs$-inflation sharing the cone isomorphic that of $a$ by \cite[dual of Prop.~3.15]{NP19}.
Thus we get $b'\in\Inf_\CN$.
At the second step, we consider the splitting cocone of the retraction $a_2$, say $N\xto{a'_2}A''$, and apply the axiom {\rm (ET4)} to the composition $f_0\circ a'_2$ to produce the bottom row.
The cocones of $a_2$ and $b''$ are isomorphic to each other, so $a_2\in\Def_\CN^{\sp}$ implies $b''\in\Def_\CN$.
Thus $b_1$ has been replaced by a desired morphism $b''\circ b'\in\Sn$.

\ref{MR4}: We only prove that $\ovl{\mathcal{M}}_{\mathsf{inf}}:=\{ \ovl{v}\circ \ovl{f}\circ \ovl{u}\mid f\ \text{is an}\ \mathfrak{s}\text{-inflation}, \ovl{u},\ovl{v}\in\ovl{\Sn}\}$ is closed under compositions in $\ovl{\CA}$.
Since $\Sn$ is closed under compositions, we are reduced to show that $\ovl{f_2}\circ\ovl{s}\circ\ovl{f_1}\in\ovl{\CM}_{\mathsf{inf}}$ for $\fs$-inflations $f_1,f_2$ and $s\in\Sn$.
Let us denote the sequence by $A\xto{f_1}B\xto{s}B'\xto{f_2}C$.
Since we have a factorization $
s\colon B\overset{\begin{bsmallmatrix}s \\
g
\end{bsmallmatrix}}\lra B'\oplus N\overset{\begin{bsmallmatrix}\id_{B'} & 0
\end{bsmallmatrix}}{\lra} B'
$ with $N\in\CN$ by $\Sn=\Def_\CN^{\sp}\circ\Inf_\CN$,
the sequence $f_2\circ s\circ f_1$ can be factorized as below:
\[
A
\xto{f_1}
B
\xto{\begin{bsmallmatrix}s \\
g
\end{bsmallmatrix}}
B'\oplus N
\xto{\begin{bsmallmatrix}f_2 \amph 0 \\
0 \amph \id_N
\end{bsmallmatrix}}
C\oplus N
\xto{\begin{bsmallmatrix}
\id_C \amph 0
\end{bsmallmatrix}}
C,
\]
where we notice that the first three morphisms are $\fs$-inflations and the last one sits in $\Def_\CN^{\sp}$.
Thus we conclude $\ovl{f_2\circ s\circ f_1}\in\ovl{\mathcal{M}}_{\mathsf{inf}}$.

As a conclusion, the extriangulated quotient of $\CA$ by $\CN$ exists by \cref{cor:mult_loc}.
\end{proof}

\begin{example}\label{ex:spade}
We list a few prototypical examples of extension-closed subcategories $\CN\sse\CA$ satisfying \ref{spade}.
\begin{enumerate}[label=\textup{(\arabic*)}]
\item 
Let $\CN$ be a subcategory projective-injective objects in $\CA$ which is closed under direct summands.
Then it satisfies \ref{spade}:
Actually, a given $\fs$-inflation $N_1\xto{f}B$ splits and the corresponding retraction $B\xto{b}N_1$ is a desired morphism.
In this case, we have the extriangulated quotient $\CA/\CN$ which is exact equivalent to the ideal quotient $\CA/[\CN]$, see \cite[Rem.~3.35]{NOS22}.
\item 
A biresolving subcategory $\CN\sse\CA$ satisfies \ref{spade}:
Actually, for a given $\fs$-inflation $N_1\xto{f}B$, there exists an $\fs$-inflation $B\xto{b}N_2$ with $N_2\in\CN$.
Since $\CN$ is thick, the composite $b\circ f$ is an $\fs|_\CN$-inflation.
Thus, by \cref{prop:extriangulated_quotient_under_spade}, we have the extriangulated quotient $\CA/\CN$ which paraphrases \cref{ex:extri_quotient_by_biresolving}.
\end{enumerate}
\end{example}

As below, the class $\Sn$ is \emph{saturated}, i.e., a morphism $s\in\Mor\CA$ is in $\Sn$ if and only if $Q(s)$ is an isomorphism in $\CA/\CN$.

\begin{corollary}\label{cor:Sn_is_saturated}
The class $\Sn$ is saturated with respect to the localization $Q\colon \CA\to\CA/\CN$.
\end{corollary}
\begin{proof}
Let $A\xto{f}B$ be a morphism in $\CA$ such that $Q(f)$ is an isomorphism.
Since $\CA/\CN$ is defined as the localization $\CA[\ovl{\Sn}^{-1}]$, the inverse $Q(f)^{-1}$ is represented by a left fraction $B\xto{g}A'\xleftarrow{s}B$ where $s\in\Sn$.
We can chose such a roof together with a commutative diagram in $\overline{\CA}$:
\[
\begin{tikzcd}
&A'\arrow[equal]{d}&\\
A\arrow{ru}{gf}\arrow[equal]{rd}{}&A'&A\arrow{lu}[swap]{s}\arrow{l}[swap]{s}\arrow[equal]{ld}{}\\
&A\arrow{u}{s}&
\end{tikzcd}
\]
Thus we get $gf=s\in\overline{\Sn}$ in $\overline{\CA}$ and it can be factorized by $\Sn=\Def_\CN^{\sp}\circ\Inf_\CN$.
In particular, there exists a morphism $A\xto{h}N$ such that the induced morphism $A\xto{\begin{bsmallmatrix}gf \\
h
\end{bsmallmatrix}}A'\oplus N$ belongs to $\Inf_\CN$ and $N\in\CN$.
The ${\rm (WIC)}$ assumption shows that $\begin{bsmallmatrix}f \\
h
\end{bsmallmatrix}$ is also in $\Inf_\CN$ and $f\in\Sn$ as desired.
\end{proof}

As an advantage of the exact sequence $\Theta$ in \cref{prop:extriangulated_quotient_under_spade}, it restricts to an extension-closed subcategory $\CA'\sse\CA$ and produces an exact subsequence.

\begin{proposition}\label{prop:restriction_of_extri_subquotient}
Let $(\CA',\BE|_{\CA'},\fs|_{\CA'})\sse (\CA,\BE,\fs)$ be an extension-closed subcategory which contains $(\CN,\BE|_\CN,\fs|_\CN)$ and is closed under direct summands.
Then we have the following assertions.
\begin{enumerate}[label=\textup{(\arabic*)}]
\item 
$(\CA',\CN)$ satisfies the condition \ref{spade}.
\item 
The pair induces an exact subsequence $\Theta'\sse\Theta$, the situation of which is displayed as the following commutative diagram in $\ET$,
\begin{equation}\label{diag:PBII}
\begin{tikzcd}[column sep=1.5cm]
    (\CN,\BE|_\CN,\fs|_\CN) \arrow{r}\arrow[equal]{d}{}& (\CA',\BE|_{\CA'},\fs|_{\CA'}) \arrow{r}{(Q',\mu')}\arrow{d}&
    (\CB',\BF',\ft') \arrow{d}{(G,\psi)} \\
    (\CN,\BE|_\CN,\fs|_\CN) \arrow{r}&
    (\CA,\BE,\fs)\arrow{r}[swap]{(Q,\mu)} &
    (\CB,\BF,\ft)
\end{tikzcd}
\end{equation}
where we put $(\CB',\BF' ,\ft')= (\CA'/\CN,\wtil{\BE|_{\CA'}},\wtil{\fs|_{\CA'}})$.
Moreover, the image of $G$ is extension-closed in $(\CB,\BF,\ft)$.
\item 
If we identify $\CB'$ with $\Im G$, we have an exact equivalence $(\CB',\BF' ,\ft')\xto{\sim}(\CB',\BF|_{\CB'},\ft|_{\CB'})$.
\end{enumerate}
\end{proposition}
The assertion (1) is obvious, so we have an exact sequence $\Theta'$ associated to the admissible pair $(\CA',\CN)$.
To write down the proof of \cref{prop:restriction_of_extri_subquotient}, we declare that $\Sn$ and $\Sn'$ denote the defining classes of morphisms associated to the extriangulated quotient $\CA/\CN$ and $\CA'/\CN$, respectively.
Due to \cref{lem:factorization_of_Sn}, we have the following useful property of $\Sn$.

\begin{corollary}\label{cor:closed_under_weak_equiv}
Let $s\colon A\to B$ be a morphism in $\Sn\sse\Mor\CA$.
If one of $\{A, B\}$ belongs to $\CA'$, then so does the other.
Moreover we have an equality $\Sn'=\Mor\CA'\cap\Sn$.
\end{corollary}
\begin{proof}
Suppose that $A$ belongs to $\CA'$.
As we have seen in \cref{lem:factorization_of_Sn}, the morphism $s$ admits a factorization $s\colon A\xto{s_1}B'\xto{s_2}B$ with $s_1\in\Inf_\CN, s_2\in\Def_\CN^{\sp}$.
Since $\cone(s_1)\in\CN$ and $\CA'\sse\CA$ is extension-closed, we get $B'\in\CA'$.
Also, since $s_2$ is a retraction, we get $B\in\CA'$ as well.
The case of $B\in\CA'$ can be checked in a dual manner.

To show the inclusion $\Sn'\supseteq\Mor\CA'\cap\Sn$, we consider a morphism $s\colon A\to B$ with $A,B\in\CA'$.
By mimicking the above argument, we see $s\in\Sn'$ in a few moment. 
\end{proof}

\begin{proof}[Proof of \cref{prop:restriction_of_extri_subquotient}]
By the universality of $Q'\colon \CA'\to \CB'$, the exact functor $(G,\psi)$ exists and makes the diagram \eqref{diag:PBII} commute.
To check (2)(3), We have to only show that $(G,\psi)$ is fully faithful and the image of $G$ is extension-closed in $(\CB,\BF,\ft)$.

\underline{Faithfulness of $G$}:
We consider a morphism $Q'A\xto{\alpha}Q'B$ in $\CB'$ which is represented by a left fraction $A\xto{f}B'\xleftarrow{s}B$ in $\CA'$ with $s\in\Sn'$.
If we suppose $G(\alpha)=0$, we can deduce $Qf=0$ from $G(\alpha)=(Qs)^{-1}Qf$ in $\CB$.
Since $\Sn$ forms a multiplicative system in $\ovl{\CA}$ by \ref{MR2}, there is a morphism $A'\xto{t}A$ in $\Sn$ such that the composition $tf$ factors through an object in $\CN$.
By \cref{cor:closed_under_weak_equiv}, the morphism $tf$ indeed sits in $\CA'$.
Thus we can apply $Q'$ to get a zero morphism $Q'(tf)=0$ in $\CA'$.
Also, we see $t\in\Sn'$ by \cref{cor:closed_under_weak_equiv} and $Q'(t)$ is an isomorphism, which shows $Q'(f)=0$ and $\alpha=0$.

\underline{Fullness of $G$}:
We fix objects $A,B\in\CA'$ and consider a morphism $QA\xto{\alpha}QB$ in $\CB$ which is represented by a left fraction $A\xto{f}B'\xleftarrow{s}B$ in $\CA$ with $s\in\Sn$.
We know this left fraction sits in $\CA'$ and $s\in\Sn'$ by \cref{cor:closed_under_weak_equiv}.

\underline{Extension-closedness of $\Im G$}:
Let $A',C'$ be objects in $\CA'$ and consider a $\ft$-triangle $QA'\overset{x}{\lra} Y\overset{y}{\lra} QC'\overset{\delta}{\dra}$ in $\CB$.
We have to show $Y\in\Im G$.
By \cref{lem:inf_from_inf}, there is an $\fs$-triangle $A\overset{f}{\lra} B\overset{g}{\lra} C\overset{\eta}{\dra}$ in $\CA$ which is isomorphic to the above $\ft$-triangle under $Q$.
In particular, we have an isomorphism $QA\cong QA'$ and $A\in\CA'$ by \cref{cor:closed_under_weak_equiv}.
Also we get $C\in\CA'$ and $B\in\CA'$ by the extension-closedness of $\CA'$.
Since $G$ is fully faithful, the commutativity $Q=G\circ Q'$ shows that isomorphisms $\alpha\colon Q'A\xto{\simeq}Q'A'$ and $\beta\colon Q'C'\xto{\simeq}Q'C$ exist in $\CB'$.
Taking a pushout and a pullback along $\alpha$ and $\beta$ in succession, we get the following isomorphic $\ft'$-triangles in $\CB'$,
\[
\begin{tikzcd}[column sep=1.2cm]
    Q'A
        \arrow{r}{Q'f}
        \arrow{d}{\simeq}[swap]{\alpha}
    & Q'B
        \arrow{r}{Q'g}
        \arrow{d}{\simeq}
    & Q'C 
        \arrow[dashed]{r}{\delta'}
        \arrow[equal]{d}{}
    & {}\\
    Q'A'
        \arrow{r}{}
    & Y'
        \arrow{r}{}
    & Q'C
        \arrow[dashed]{r}{\alpha_*\delta'}
    & {}\\
     Q'A'
        \arrow{r}{}\arrow[equal]{u}{}
    & Y''
        \arrow{r}{}\arrow{u}{\simeq}
    & Q'C'
        \arrow[dashed]{r}{\beta^*\alpha_*\delta'}\arrow{u}{\simeq}[swap]{\beta}
    & {}
\end{tikzcd}
\]
where we put $\delta'=\mu_{C',A'}(\eta)$.
Applying $G$ to the above, we have a desired isomorphism of $\ft$-triangles:
\[
\begin{tikzcd}[column sep=1.2cm]
    QA'
        \arrow{r}{x}
        \arrow[equal]{d}{}
    & Y 
        \arrow{r}{y}
        \arrow{d}{\simeq}
    & QC' 
        \arrow[dashed]{r}{\delta}
        \arrow[equal]{d}{}
    & {}\\
    QA'
        \arrow{r}{Qf}
    & GY''
        \arrow{r}{Qg}
    & QC'
        \arrow[dashed]{r}{\delta}
    & {}
\end{tikzcd}
\]
which shows $\delta=\psi_{Q'C',Q'A'}(\beta^*\alpha_*\delta')$.
Hence we have seen $Y\in\CB'$ and that $\psi\colon \BF'\to\BF$ is surjective for each pair of objects.
Again by that fact that $G$ is fully faithful, we see $\psi$ is bijective.
In turn, the functor $(G,\psi)$ shows the exact equivalence $(\CB',\BF,\ft')\xto{\sim} (\Im G,\BF|_{\Im G},\ft|_{\Im G})$ by \cref{lem:exact_equivalence}.
\end{proof}

\subsection{The extriangulated subquotient}
\label{subsec:extriangulated_subquotient}
We will establish an exact sequence from any extension-closed subcategory $\CN\sse\CA$ by generalizing the localization developed in \cite{Oga24}.
In contrast to the previous subsection \S\ref{subsec:auxiliary_exact_sequence}, we do not assume that $\CA$ satisfies the condition \ref{spade}.
The setup in this subsection is this.

\begin{setup}
\label{setup:extriangulated_subquotient}
Let $(\CA,\BE,\fs)$ be an extriangulated category with ${\rm (WIC)}$ and consider an extension-closed subcategory $\CN\sse\CA$ which are closed under direct summands.
\end{setup}

Notice that $(\CA,\CN)$ is not assumed to be admissible.
However, we can make $(\CA,\CN)$ admissible by passing to the following extriangulated substructures of $(\CA,\BE,\fs)$ which are determined by the extension-closed subcategory $\CN$.
It is indeed a modification of Quillen's substructures $\BE_{[\CN]}$ and $\BE^{[\CN]}$ introduced in \cref{lem:Quillen_substructure}.

\begin{proposition}\label{prop:modified_Quillen_substructure}
\cite[Prop. A.4]{Che23}
By defining subsets of $\BE(C,A)$ for any objects $A,C\in\CA$ as follows, we have closed bifunctors of $\BE$.
\begin{enumerate}[label=\textup{(\arabic*)}]
\item $\BE^{R}_{\CN}(C,A) \deff \Set{\delta\in\BE(C,A) | a_*\delta\in\BE_{[\CN]}(C,N)\ \textnormal{for any}\ A\xto{a}N\ \textnormal{with}\ N\in\CN}$;
\item $\BE_{\CN}^L(C,A) \deff \Set{\delta\in\BE(C,A) | c^*\delta\in\BE_{[\CN]}(N,A)\ \textnormal{for any}\ N\xto{c}C\ \textnormal{with}\ N\in\CN}$.
\end{enumerate}
Moreover, putting $\BE_\CN:=\BE^R_\CN\cap\BE^L_\CN$, we have these extriangulated structures
\[
\CA_\CN^R\deff (\CA_\CN^R,\BE_\CN^R,\fs_\CN^R),\quad \CA_\CN^L\deff (\CA_\CN^L,\BE_\CN^L,\fs_\CN^L),\quad \CA_\CN\deff (\CA_\CN,\BE_\CN,\fs_\CN)
\]
relative to $(\CA,\BE,\fs)$.
Here $\fs_\CN$ is a restriction of $\fs$ to $\BE_\CN$ and other undefined symbols are used in similar meanings.
Note that all the above extriangulated subcategories share the same underlying category $\CA$.
\end{proposition}

\begin{remark}\label{rem:modified_Quillen_substructure}
Before stating the main theorem, we should make some remarks on \cref{prop:modified_Quillen_substructure}.
\begin{enumerate}[label=\textup{(\arabic*)}]
\item 
If $\CA$ corresponds to a triangulated category, the subfunctors $\BE_{\CN}^{R}, \BE_{\CN}^L$ and $\BE_{\CN}$ are same as the closed subfunctors defined in \cite[Prop. 2.1]{Oga24}.
\item 
The stabilities $\BE^R_{\CN}=(\BE^R_{\CN})^R_{\CN}$ and $\BE^L_{\CN}=(\BE^L_{\CN})^L_{\CN}$ is inherited from those of $\BE_{[\CN]}$ and $\BE^{[\CN]}$ that we verified in \cref{lem:stability_of_Quillen_substructure}, respectively.
In particular, we have $\BE_\CN=(\BE_\CN)_\CN$ too.
\end{enumerate}
\end{remark}

The main feature of the extriangulated substructure $\CA_\CN=(\CA_\CN,\BE_\CN,\fs_\CN)$ is that it satisfies the condition \ref{spade} in \cref{setup:spade}.
Due to \cref{prop:restriction_of_extri_subquotient}, we know that the pair $(\CA_\CN,\CN)$ is admissible and hence the following are straightforward.

\begin{theorem}\label{thm:subquotient_of_extri_cat}
Let $(\CA,\CN)$ be the pair of an extriangulated category $(\CA,\BE,\fs)$ and an extension-closed subcategory $\CN\sse\CA$ which is closed under direct summands.
We consider the extriangulated subcategory $\CA_\CN=(\CA_\CN,\BE_\CN,\fs_\CN)$ determined by $\CN$.
Then, we have the following.
\begin{enumerate}[label=\textup{(\arabic*)}]
\setcounter{enumi}{-1}
\item 
We have an identity $(\CN,\BE|_\CN,\fs|_\CN)=(\CN,\BE_\CN|_\CN,\fs_\CN|_\CN)$ as extriangulated categories.
\item 
The pair $(\CA_\CN,\CN)$ satisfies the condition \ref{spade}.
\item 
The pair $(\CA_\CN,\CN)$ is admissible, that is, there exists the associated exact sequence $\Theta$ in $\ET$:
\begin{equation}\label{seq:subquotient_of_extri_cat}
\begin{tikzcd}
    (\CN,\BE|_\CN,\fs|_\CN)
        \arrow{r}
    & (\CA_\CN,\BE_\CN,\fs_\CN)
        \arrow{r}{(Q,\mu)}
    & (\CA_\CN/\CN,\widetilde{\BE_\CN},\widetilde{\fs_\CN}),
\end{tikzcd}
\end{equation}
where we denote by $\CA_\CN/\CN$ the underlying category of the associated extriangulated quotient.
\item 
Consider the defining classes $\Inf_\CN, \Def_\CN$ and $\Sn$ associated to the quotient $\CA_\CN/\CN$.
Then we have identities $\Sn=\Def_\CN^{\sp}\circ \Inf_\CN=\Def_\CN\circ \Inf_\CN^{\sp}$.
Moreover, $\Sn$ is saturated with respect to the quotient $(Q,\mu)$.
\end{enumerate}
We call it the \emph{extriangulated subquotient} (or \emph{subquotient} for short) of an extriangulated category $\CA$ by an extension-closed subcategory $\CN$.
\end{theorem}
\begin{proof}
The item (0) is immediate from the definition of $\BE_\CN$.
If \ref{spade} holds for $(\CA_\CN,\CN)$, the assertions (2) and (3) are just a combination of \cref{prop:restriction_of_extri_subquotient}, \cref{cor:closed_under_weak_equiv} and \cref{cor:Sn_is_saturated}.
So it is enough to show the item (1).
To do this, we consider an $\fs_\CN$-conflation $N_1\overset{f}{\lra} B\overset{g}{\lra} C\overset{\delta}{\dra}$ with $N_1\in\CN$.
By the definition of $\BE^R_\CN$, we see $\delta\in\BE_{[\CN]}$.
Again by the definition of $\BE_{[\CN]}$, there is a morphism $C\xto{h}N_3$ and an $\BE$-extension $\eta\in\BE(N_3,N_1)$ such that $\BE(h,A)(\eta)=\delta$.
That is, we have the following morphism of $\fs$-triangles.
\begin{equation*}\label{diag:subquotient_of_extri_cat}
\begin{tikzcd}
    N_1 \arrow{r}{f}\arrow[equal]{d}{}& B \arrow{r}{g}\arrow{d}[swap]{b}&C \arrow{d}{h}\arrow[dashed]{r}{\delta} & {}\\
    N_1 \arrow{r}{bf}&N_2\arrow{r}{} &N_3\arrow[dashed]{r}{\eta}\wPB{lu} & {}
\end{tikzcd}
\end{equation*}
As $\CN$ is extension-closed in $\CA$, we get $N_2\in\CN$.
The morphism $b$ is a desired one, which shows \ref{spade} by combining the dual argument.
\end{proof}

Now, it turns out that a certain pair $(\CA,\CN)$ may provide two different types of localizations.
We should make a comparison of the extriangulated quotient and the subquotient.
If the pair $(\CA,\CN)$ is admissible, by the universality of $(Q',\mu')\colon \CA\to\CA_{\CN}/\CN$, we have a morphism $(F,\phi)\colon \CA_\CN/\CN\to\CA/\CN$ in $\ET$ which makes the following diagram commute.
\begin{equation}\label{diag:comparison}
\begin{tikzcd}[column sep=1.5cm]
    (\CN,\BE|_\CN,\fs|_\CN) \arrow{r}\arrow[equal]{d}{}& (\CA_\CN,\BE_\CN,\fs_\CN) \arrow{r}{(Q',\mu')}\arrow{d}{(\inc,\iota)}&(\CA_\CN/\CN,\wtil{\BE_\CN},\wtil{\fs_\CN}) \arrow[dotted]{d}{(F,\phi)} \\
    (\CN,\BE|_\CN,\fs|_\CN) \arrow{r}&(\CA,\BE,\fs)\arrow{r}{(Q,\mu)} &(\CA/\CN,\wtil{\BE},\wtil{\fs})
\end{tikzcd}
\end{equation}

We should remark that \eqref{diag:comparison} is very different to \eqref{diag:PBII}, because $\CA_\CN$ is never extension-closed in $(\CA,\BE,\fs)$ unless $\BE_\CN=\BE$ occurs.
We will reveal a necessary and sufficient condition for $(\CA,\CN)$ to produce the exact equivalence $(F,\phi)$, where the condition \ref{spade} again comes into play.

\begin{proposition}\label{prop:comparison_of_quotient_and_subquotient}
We keep the setup in \cref{thm:subquotient_of_extri_cat}.
Then the following conditions are equivalent.
\begin{enumerate}[label=\textup{(\arabic*)}]
\item 
The extriangulated subcategory $(\CN,\BE|_\CN,\fs|_\CN)$ is thick and satisfies the conditions \ref{MR1}--\ref{MR4}.
The functor $(F,\phi)\colon (\CA_\CN/\CN,\wtil{\BE_\CN},\wtil{\fs_\CN})\to (\CA/\CN,\wtil{\BE},\wtil{\fs})$ is an exact equivalence.
\item 
The identity $\BE=\BE_\CN$ holds.
\item
The pair $(\CA,\CN)$ satisfies \ref{spade} of \cref{setup:spade}.
\end{enumerate}
\end{proposition}
\begin{proof}
(3) $\Rightarrow$ (2):
Note that $\CN$ is a thick subcategory of $\CA$ by \cref{lem:spade_implies_thickness}.
To confirm the condition \ref{spade} forces the equality $\BE=\BE_\CN$, we shall check that any $\fs$-triangle $N_1\overset{f}{\lra} B\overset{g}{\lra} C\overset{\delta}{\dra}$ with $N_1\in\CN$ is an $\fs_\CN$-triangle.
In fact, due to \ref{spade}, we have a morphism $B\xto{b}N_2$ such that $N_1\overset{bf}{\lra}N_2$ is a $\fs$-inflation and a diagram of the same form as \eqref{diag:subquotient_of_extri_cat}.
This verifies $\delta\in\BE^R_{\CN}(C,N_1)$ and $\BE=\BE^R_{\CN}$.
Combining the dual, we have done.

(2) $\Rightarrow$ (1) and (3): Both items follow from \cref{thm:subquotient_of_extri_cat}(1)(2).

(1) $\Rightarrow$ (3):
Let us consider an $\fs$-triangle $N_1\overset{f}{\lra} B\overset{g}{\lra} C\overset{\delta}{\dra}$ with $N_1\in\CN$.
By the exact equivalence $(F,\phi)$, we know $Q'(g)$ is an isomorphism in $\CA_\CN/\CN$.
By \cref{thm:subquotient_of_extri_cat}(3), the $\fs$-deflation $g$ admits a factorization
\[
g\colon B\overset{\begin{bsmallmatrix}g \\
b
\end{bsmallmatrix}}\lra C\oplus N_2\overset{\begin{bsmallmatrix}\id_C & 0
\end{bsmallmatrix}}{\lra} C
\]
where $\begin{bsmallmatrix}g \\
b
\end{bsmallmatrix}$ is an $\fs_\CN$-inflation and $N_2\in\CN$.
Thus we get an $\fs_\CN$-inflation $\begin{bsmallmatrix}g \\
b
\end{bsmallmatrix}\circ f=\begin{bsmallmatrix}0 \\
bf
\end{bsmallmatrix}$.
Due to ${\rm (WIC)}$, we have an $\fs_\CN$-inflation $N_1\xto{bf}N_2$.
Since we can complete $bf$ into an $\fs_\CN$-conflation $N_1\xto{bf} N_2\to N_3$ in $\CN$, the morphism $b$ is a desired one and we omit the dual part of the remaining proof.
\end{proof}

In particular, under the condition $\ref{spade}$, the extriangulated quotient $\CA/\CN$ coincides with the extriangulated subquotient $\CA_\CN/\CN$.

\subsection{The extriangulated subquotients are algebraic}
\label{subsec:subquotients_are_algebraic}
It is well-known that the property of a triangulated category to be algebraic is stable under taking the Verdier quotient, e.g. \cite[\S 3.6]{Kel06}, \cite[\S 7.5]{Kra07}.
As stated in \cite{Che23}, this leads us to a natural question whether the algebraic extriangulated categories are closed under taking the extriangulated quotient.
As a partial answer to this question, we will prove our first main theorem, that is, they are closed under taking extriangulated subquotients.

From now on, the following setup is in play to the end of the section.

\begin{setup}\label{setup:algebraic_subquotient}
Let us consider the pair $(\CA,\CN)$ of an extriangulated category $(\CA,\BE,\fs)$ with ${\rm (WIC)}$ and an extension-closed subcategory $(\CN,\BE|_\CN,\fs|_\CN)$ which is closed under direct summands. In addition, we assume that $\CA$ admits an exact dg enhancement $(\A,\SS)$ which is connective and cofibrant.
\end{setup}

\emph{Keeping in mind \cref{prop:comparison_of_quotient_and_subquotient}, we may assume that the pair $(\CA,\CN)$ satisfies \ref{spade} in \cref{setup:spade}.}
If not in that case, we can take the extriangulated substructure $\CA_\CN=(\CA_\CN,\BE_\CN,\fs_\CN)$ by \cref{thm:subquotient_of_extri_cat}.
As stated in \cref{prop:comparison_of_quotient_and_subquotient}, we have an exact sequence \eqref{seq:mult_loc} in $\ET$ such as
\begin{equation}\label{seq:algebraic_subquotient}
\begin{tikzcd}
    \Theta\colon(\CN,\BE|_\CN,\fs|_\CN)
        \arrow{r}
    & (\CA,\BE,\fs)
        \arrow{r}{(Q,\mu)}
    & (\CA/\CN,\widetilde{\BE},\widetilde{\fs}).
\end{tikzcd}
\end{equation}

The universal embedding $F\colon \A\to\Db_{\dg}(\A)$ in \cref{thm:universal_embedding} shows that $H^0\A=(\CA,\BE,\fs)$ is an extension-closed subcategory of the bounded derived category $\Db(\A)$.
By considering the thick closure of $\CN$ in $\Db(\A)$, we have the Verdier quotient of $\Db(\A)$ by $\thick\CN$.
Thus, by the universality of \eqref{seq:algebraic_subquotient}, we have the following commutative diagram in $\ET$ consisting of exact sequences,
\begin{equation}
\label{diag:algebraic_subquotient}
\begin{tikzcd}[column sep=1.5cm]
    (\CN,\BE|_\CN,\fs|_\CN) \arrow{r}\arrow[hook]{d}{}& (\CA,\BE,\fs) \arrow{r}{(Q,\mu)}\arrow[hook]{d}&
    (\CA/\CN,\widetilde{\BE},\widetilde{\fs}) \arrow{d}{(G,\psi)} \\
    \thick\CN \arrow{r}&
    \Db(\A)\arrow{r}{L} &
    \dfrac{\Db(\A)}{\thick\CN}
\end{tikzcd}
\end{equation}
where we denote by $L$ the Verdier quotient.

We are going to show the extriangulated quotient $\CA/\CN$ is still algebraic by verifying that it forms an exact subsequence of the Verdier quotient.
In turn, what we will do amounts to the following statement.

\begin{proposition}
\label{prop:algebraic_subquotient}
The induced exact functor $(G,\psi)$ is fully faithful and the essential image $\Im G$ is extension-closed in $\frac{\Db(\A)}{\thick\CN}$.
Moreover, $(\CA/\CN,\wtil{\BE},\wtil{\fs})$ is exact equivalent to $\Im G$ and satisfies {\rm (WIC)}.
\end{proposition}

To prove \cref{prop:algebraic_subquotient}, we reveal a convenient description of $\thick\CN$.
Denote, for two classes $\CU$ and $\CV$ of objects in a triangulated category, by $\CU*\CV$ the closure with respect to taking the direct summands of the full subcategory consisting of $X$ occurring in a triangle $U\to X\to V\to U[1]$ with $U\in\CU$ and $V\in\CV$.
Now let $\CN_1$ be the full subcategory of all $N[n]$ with $N\in\CN$ and $n\in\BZ$. For $r>0$, let $\CN_r=\CN_1*\CN_1*\cdots *\CN_1$ be the product with $r$ factors.
Then we have the identity $\thick\CN=\bigcup_{r\geq 0}\CN_r$ in general.
Due to the condition \ref{spade}, we have a finer description.

\begin{lemma}
\label{lem:finer_description_of_thickN}
Any object $X\in\thick\CN$ admits integers $n\leq m$ such that $X\in\CN[n]*\CN[n+1]*\cdots *\CN[m]$.
\end{lemma}
\begin{proof}
There exists an integer $r$ such that $X\in\CN_r$.
Since the case of $r=1$ is obvious, we may assume that there exists a triangle
$X'\to X\to N[l]\to X'[1]$ with $N\in\CN$ and $X'\in\CN[n]*\CN[n+1]\cdots *\CN[m]$ for some $n\leq m$.
There is nothing to show provided $m\leq l$, so we suppose $m>l$.
We are reduced to show the following claim.

\begin{claim}\label{claim:finer_description_of_thickN}
We have $\CN[n]*\CA\sse \CN[n-1]*\CA*\CN[n]$ for any integer $n\geq 1$.
In particular, $\CN*\CN[1]*\cdots*\CN[n]*\CA\sse \CN*\CA*\CN[1]*\cdots*\CN[n]$ holds for any integer $n\geq 1$.
\end{claim}
\begin{proof}
Let $X$ be an object in l.h.s. and consider its defining triangle $N[n]\xto{f} X\xto{g} A\xto{h} N[n+1]$ with $N\in\CN$ and $A\in\CA$.
Since $\Db(\A)(?,-[n+1])=\Ext^{n+1}_{\Db(\A)}(?,-)$ is effaceable by \cref{prop:higher_extension_via_univ_emb}, the element $h\in\Ext^{n+1}_{\Db(\A)}(A,N)$ admits an $\fs$-conflation $N\xto{a} B\xto{b} C$ with $a[n+1]\circ h=0$ in $\Db(\A)$.
Now we employ \ref{spade} for the $\fs$-inflation $N\xto{a}B$, we get a morphism $B\xto{a'} N'$ in $\CA$ with the composite $a'a$ being an $\fs|_\CN$-inflation.
Thus we have an $\fs|_\CN$-conflation $N\xto{a'a}N'\xto{b'}N''$ in $\CN$ which induces the following commutative diagram in $\Db(\A)$ where the first and second columns and the bottom row are triangles.
\[
\begin{tikzcd}[row sep=0.8cm, column sep=1.6cm]
N''[n-1] \arrow{r}{}\arrow{d}[swap]{f'}\wPO{rd}
&N[n] \arrow{d}{f}
&
&
\\
X' \arrow{r}{}\arrow{d}[swap]{g'}
& X \arrow{d}{g}
&
&
\\
A\arrow[equal]{r}\arrow{d}[swap]{h'}
&
A \arrow{d}{h}\arrow[bend left, dotted]{rd}{0}
&
&
\\
N''[n]\arrow{r}{}
&
N[n+1]\arrow{r}{a'a[n+1]}
&
N'[n+1]\arrow{r}{b'[n+1]}
&
N''[n+1]
\end{tikzcd}
\]
Thus we see $X\in \CN[n-1]*\CA*\CN[n]$.
If $n-1\geq 1$, by repeating the same operations to $\CN[n-1]*\CA$, we have the latter inclusion.
\end{proof}
To be precise, we return to the argument on $X\in \CN[n]*\cdots *\CN[m]*\CN[l]$ with $m>l$.
By \cref{claim:finer_description_of_thickN}, we have $\CN[m]*\CN[l]\sse \CN[m-1]*\CN[l]*\CN[m]$ which explains well the following equations:
\begin{align*}
\CN[n]*\cdots *\CN[m]*\CN[l] &\sse \CN[n]*\cdots *\CN[m-1]*\CN[l]*\CN[m]\\
&\sse \CN[n]*\cdots *\CN[m-2]*\CN[l]*\CN[m-1]*\CN[m]\\
&\sse \CN[\min(n,l)]*\cdots*\CN[m]
\end{align*}
letting us close the proof.
\end{proof}

For the readability purpose, we proceed with the proof of \cref{prop:algebraic_subquotient} by dividing it in several steps.
In the rest of this subsection, we put $\wtil{\CD}\deff\frac{\Db(\A)}{\thick\CN}$ and $\CN[n,m]\deff\CN[n]*\cdots *\CN[m]$ for $n\leq m$ to save the space.
The following lemma is crucial and we need to use it many times not only in what follows but also in \S\ref{subsec:Drinfeld_subquotient}.

\begin{lemma}\label{lem:description_for_extensions}
Let us consider objects $A, B\in\CA$, an integer $k\geq 0$ and a morphism $A\xto{\alpha} B[k]$ in $\wtil{\CD}$.
Then, it is represented by a left fraction $s[k]\backslash a\colon A\xto{a}B'[k]\xleftarrow{s[k]}B[k]$ in $\Db(\A)$ with $B'\in\CA$.
In particular, if $k=0$, the left fraction is considered in $\CA$.
\end{lemma}
\begin{proof}
The morphism $\alpha$ is represented by a left fraction $s[k]\backslash a\colon A\xto{a}B'\xleftarrow{s[k]}B[k]$ in $\Db(\A)$.
so, we display such a triangle in $\Db(\A)$ with $\wtil{N}\in\thick\CN$ as follows,
\begin{equation}\label{diag:for_fullness1}
\begin{tikzcd}[column sep=1.2cm]
    {}
    & A \arrow[dotted]{ld}[swap]{\alpha}\arrow{d}{a}
    & {}
    & {}
    \\
    B[k] \arrow{r}[swap]{s[k]}
    & B' \arrow{r}[swap]{b}
    & \wtil{N} \arrow{r}
    & B[1]
\end{tikzcd}
\end{equation}
where the dotted arrow exists in $\wtil{\CD}$.
By \cref{lem:finer_description_of_thickN}, we have $\wtil{N}\in\CN[n,m]$.
We notice that, if $\widetilde{N}\in\CN[k]$, we have nothing to do.
We proceed with the proof by induction on $l\deff m-n\geq 1$.

(i) \underline{The case of $m>k$}:
By Lemma \ref{lem:finer_description_of_thickN}, the object $\widetilde{N}$ admits a triangle
\begin{equation}\label{diag:for_fullness2}
\wtil{N}'\overset{f}{\lra} \wtil{N}\overset{g}{\lra} N_1[m]\overset{h}{\lra} \wtil{N}'[1]
\end{equation}
with $\wtil{N}'\in\CN[n,m-1]$ and $N_1\in\CN$ with $n\leq m-1$.
The following argument will be frequently used, so we want to record it as a Claim.

\begin{claim}\label{claim:fulness}
There exists a modification of \eqref{diag:for_fullness2}, namely, a triangle
\begin{equation}\label{diag:for_fullness2-2}
\wtil{N}''\overset{f'}{\lra} \wtil{N}\overset{g'}{\lra} N_2[m]\overset{}{\lra} \wtil{N}''[1]
\end{equation}
such that $\wtil{N}''\in\CN[n,m-1], N_2\in\CN$ and $g'ba=0$.
\end{claim}
\begin{proof}
Since the composite $A\xto{gba}N_1[m]$ belongs to $\Ext^m(A,N_1)$, we have an $\fs$-triangle
\begin{equation}\label{diag:for_fullness3}
N_1\overset{x}{\lra} N_2\overset{y}{\lra} N_3\overset{z}{\dra}
\end{equation}
with $x[m]\circ gba=0$ by the effaceability.
Thanks to our assumption \ref{spade}, we may take the above $\fs$-triangle in $\CN$.
Applying the octahedral axiom to the composite $x[m]\circ g$, we have the following commutative diagram made of triangles in $\Db(\A)$.
\begin{equation}\label{diag:for_fullness4}
\begin{tikzcd}
    {}
    & N_2[m-1] \arrow{d}{}\arrow[equal]{r}{}
    & N_2[m-1] \arrow{d}{}
    & {}
    \\
    \wtil{N}' \arrow{r}{}\arrow[equal]{d}
    & \wtil{N}'' \arrow{r}{}\arrow{d}{f'}
    & N_3[m-1] \arrow{r}\arrow{d}{}
    & \wtil{N}'[1]\arrow[equal]{d}{}
    \\
    \wtil{N}' \arrow{r}{f}
    & \wtil{N} \arrow{r}{g}\arrow{d}[swap]{x[m]\circ g}
    & N_1[m] \arrow{r}\arrow{d}{x[m]}
    & \wtil{N}'[1]
    \\
    {} 
    & N_2[m] \arrow[equal]{r}{}
    & N_2[m] 
    & {}
\end{tikzcd}
\end{equation}
This gives a desired decomposition of $\wtil{N}$ which appears in the second column.
We notice that $\wtil{N}''$ still belongs to $\CN[n,m-1]$ and $x[m]\circ gba=0$.
\end{proof}

Thus we may suppose $gba=0$ for the triangle \eqref{diag:for_fullness2}.
We immediately see $ba$ factors through $\wtil{N}'$ and get the commutative square $(\circlearrowleft)$ as indicated below.
\[
\begin{tikzcd}[column sep=1.5cm]
    {}
    & A \arrow{r}{a}\arrow[dotted]{d}[swap]{a'}\arrow[mysymbol]{rd}[description]{(\circlearrowleft)}
    & B' \arrow{d}{b}
    & {}
    \\
    N_1[m-1]\arrow{r}{}
    & \wtil{N}' \arrow{r}[swap]{f}
    & \wtil{N} \arrow{r}[swap]{g}
    & N_1[m]
\end{tikzcd}
\]
Now, by taking a homotopy pullback of $b$ along $f$, we can modify the left fraction $s\backslash a$ to be $s'\backslash a''$ via the following commutative diagram.
\[
\begin{tikzcd}[column sep=1.5cm]
    {}
    & A\arrow[bend left]{rd}{a'}\arrow[bend right = 50]{dd}[yshift=4.5mm, swap]{a}\arrow[dotted]{d}{a''}
    & {}
    & {}
    \\
    B \arrow[equal]{d}{}\arrow[crossing over]{r}[swap]{s'}
    & B'_0 \arrow{r}{}\arrow{d}{t}
    & \wtil{N}' \arrow{d}{f}\arrow{r}{}
    & B[1] \arrow[equal]{d}
    \\
    B\arrow{r}{s}
    & B' \arrow{r}{b}
    & \wtil{N} \arrow{r}\wPB{lu}
    & B[1]
\end{tikzcd}
\]
More precisely, we see $s\backslash a= (ts')\backslash (ta'')=s'\backslash a''$ by the isomorphisms $\cone(t)\cong \cone(f')\cong N_2[m]$.
As $\wtil{N}'\in\CN[n,m-1]$, we have done.
The argument tells us we may suppose $\wtil{N}\in \CN[n,m]$ and $m\leq k$ for the triangle \eqref{diag:for_fullness1}.

(ii) \underline{The case of $m\leq k$}:
We can use a right fraction $a'/s'[k]$ corresponding to $\alpha$ instead of the left fraction $s[k]\backslash a$.
Actually, by taking a homotopy pullback of $s$ along $a$, we have a right fraction as follows.
\[
\begin{tikzcd}[column sep=1.2cm]
    \wtil{N}[-1] \arrow[equal]{d}{}\arrow{r}{b'}
    &B'' \arrow{r}{s'[k]}\arrow{d}[swap]{a'}
    & A \arrow[dotted]{ld}[description]{\alpha}\arrow{d}{a}\arrow{r}{}
    & \wtil{N} \arrow[equal]{d}{}
    \\
    \wtil{N}[-1] \arrow{r}{}
    &B[k] \arrow{r}[swap]{s[k]}
    & B' \arrow{r}[swap]{b}
    & \wtil{N}
\end{tikzcd}
\]
As $n<k$, we notice that $\wtil{N}$ admits a triangle 
\begin{equation}
N_1[n-1]\overset{f}{\lra} \wtil{N}[-1]\overset{g}{\lra} \wtil{N}'[-1]\overset{h}{\lra} N_1[n]
\end{equation}
with $\wtil{N}'\in \CN[n+1,m], N_1\in\CN$ and $a'b'f=0$ by the dual of \cref{claim:fulness}.
We can check that the dual argument to the case (i) works well and have a similar commutative diagram:
\[
\begin{tikzcd}[column sep=1.5cm]
    A[-1] \arrow[equal]{d}{}\arrow[crossing over]{r}
    & \wtil{N}[-1] \arrow{r}{b'}\arrow{d}[swap]{g}\wPO{rd}
    & B'' \arrow{d}[swap]{t}\arrow{r}[xshift=1.0mm]{s'[k]}\arrow[bend left = 50]{dd}[yshift=-4.5mm]{a'}
    & A \arrow[equal]{d}
    \\
    A[-1]\arrow{r}
    & \wtil{N}'[-1] \arrow{r}{}\arrow[bend right]{rd}[swap]{b''}
    & B_0'' \arrow[crossing over]{r}[xshift=1.0mm]{s''[k]}\arrow[dotted]{d}[swap]{a''}
    & A
    \\
    {}
    &{}
    &B
    &{}
\end{tikzcd}
\]
Thus we get $a'/s'[k]=(a''t)/(s''[k]t)=a''/s''[k]$.

By the arguments (i) and (ii), we are reduced to the case of $n=m=k$ and complete the proof.
\end{proof}

The fullness of $G$ is immediate from \cref{lem:description_for_extensions} if we consider the case of $k=0$ there.

\begin{lemma}\label{lem:fulness}
$(G,\psi)$ is full.
\end{lemma}
\begin{proof}
Thanks to \cref{lem:description_for_extensions}, we see in a few moment that any morphism $\alpha\colon A\to B$ in $\wtil{\CD}$ with $A,B\in\CA$ can be represented by a left fraction $ s\backslash a$ in $\CA$.
\end{proof}

\begin{lemma}\label{lem:faithfulness}
$(G,\psi)$ is faithful.
\end{lemma}
\begin{proof}
Let us consider a morphism $A\xto{\alpha}B$ in $\CA/\CN$ such that $G(\alpha)=0$.
We have a left fraction $s\backslash a$ in $\CA$ which represents $\alpha$.
Thanks to the factorization of $s\in\Sn=\Def_\CN^{\sp}\circ \Inf_\CN$ in \cref{lem:factorization_of_Sn}, we may suppose that there is a left fraction $s\backslash a$ appearing in the following diagram
\begin{equation}\label{diag:for_faithfulness1}
\begin{tikzcd}
    {}
    & A \arrow[dotted]{ld}[swap]{\alpha}\arrow{d}{a}
    & {}
    & {}
    \\
    B \arrow{r}{s}
    & B' \arrow{r}{b}
    & N_0 \arrow{r}
    & B[1]
\end{tikzcd}
\end{equation}
with $N_0\in\CN$.
Note that the bottom row is a triangle in $\Db(\A)$ by the extension-closedness of $\CA\sse\Db(\A)$.
Regarding \eqref{diag:for_faithfulness1} as a left fraction in $\Db(\A)$, we see that the morphism $A\xto{a}B'$ factors through an object $\wtil{N}\in\thick\CN$, say $A\xto{a_1}\wtil{N}\xto{a_2}B'$.
Thus it suffices to show that $A\xto{a}B'$ factors through an object $N\in\CN$.

Recall from \cref{lem:finer_description_of_thickN} that we have $\wtil{N}\in\CN[n,m]$ for some $n\leq m$.
Similarly to the proof of \cref{lem:fulness}, we proceed by the induction on $l=m-n\geq 0$.
We deal with the case of $l=0$ later and focus on the case of $l>0$:
\begin{itemize}
    \item 
    If $m>0$, then we have a decomposition \eqref{diag:for_fullness2} of $\wtil{N}$ such that $ga_1=0$ by \cref{claim:fulness}.
Thus $a_1$ factors through $\wtil{N}'\in\CN[n,m-1]$; 
    \item 
    If $m\leq 0$, by a similar argument, we see that $a_2$ factors through an object in $\CN[n+1,m]$.
\end{itemize}
In turn, the argument amounts to the case of $l=0$ and $m=n=0$ which shows $\alpha=0$ in $\CA/\CN$.
\end{proof}

\begin{lemma}\label{lem:extension-closedness}
The essential image $\Im G$ is extension-closed in $\wtil{\CD}$.
The quotient $(\CA/\CN,\wtil{\BE},\wtil{\fs})$ is exact equivalent to $\Im G$.
\end{lemma}
\begin{proof}
Let $A,C$ be objects in $\CA$ and consider a triangle $LA\overset{f}{\lra} Y\overset{g}{\lra} LC\overset{h}{\lra} LA[1]$ in $\wtil{\CD}$.
Note that the morphism $LC\xto{h}LA[1]$ is represented by a left fraction $(s[1])\backslash h'\colon C\xto{h'}A'[1]\xleftarrow{s[1]}A[1]$ in $\Db(\A)$ with $s[1]$ being an isomorphism in $\wtil{\CD}$ as below,
\[
\begin{tikzcd}
    {}
    & C \arrow[dotted]{ld}[swap]{h}\arrow{d}{h'}
    & {}
    & {}
    \\
    A[1] \arrow{r}[swap]{s[1]}
    & A'[1] \arrow{r}{}
    & \wtil{N} \arrow{r}
    & A[2]
\end{tikzcd}
\]
where we emphasize $A'[1]$ is an object in $\Db(\A)$.
In particular, we get $h=(Ls[1])^{-1}Lh'$ in $\wtil{\CD}$.
By \cref{lem:description_for_extensions}, we may suppose $\wtil{N}=N[1]$ for some $N\in\CN$.
Since $\CA$ is extension-closed in $\Db(\A)$, we get $A'\in\CA$ and an $\fs$-triangle $A'\overset{f'}{\lra} B'\overset{g'}{\lra} C\overset{h'}{\dra} $ in $\CA$ together with $s\in\Sn$.
Thus, it turns into isomorphic $\wtil{\fs}$-triangles under $Q$ and triangles under $G$ in succession,
\[
\begin{tikzcd}[column sep=1.2cm]
QA'\arrow{r}{Qf'}\arrow{d}{\simeq}[swap]{(Qs)^{-1}}&QB'\arrow{r}{Qg'}\arrow{d}{}&QC\arrow[dotted]{r}{\mu_{C,A'}(h')}\arrow[equal]{d}{}&{}\\
QA\arrow{r}{}&Y'\arrow{r}{}&QC\arrow[dotted]{r}{\delta}&{}
\end{tikzcd},\quad
\begin{tikzcd}[column sep=1.2cm]
LA'\arrow{r}{Lf'}\arrow{d}{\simeq}[swap]{(Ls)^{-1}}&LB'\arrow{r}{Lg'}\arrow{d}{}&LC\arrow[dotted]{r}{Lh'}\arrow[equal]{d}{}&{}\\
LA\arrow{r}{}&GY'\arrow{r}{}&LC\arrow[dotted]{r}{h}&{}
\end{tikzcd}
\]
where we put $\delta=(Qs^{-1})_*(\mu_{C,A'}(h'))$.
Note that the left/right-hand diagrams sit in $\CA/\CN$ and $\frac{\Db(\A)}{\thick\CN}$, respectively.
The bottom triangle in the right-hand diagram indeed represents the element $h$ due to $h=(Ls)^{-1}\circ Lh'$.
Thus we get $Y\cong GY'\in\Im G$ which shows $\Im G$ is extension-closed in $\wtil{\CD}$.
The bottom sequences show us that $\psi_{QC,QA}\colon \wtil{\BE}(QC,QA)\to \Ext_{\wtil{\CD}}^1(LC,LA)$ is surjective as it sends $\delta$ to $h$.
By Lemmas \ref{lem:fulness} and \ref{lem:faithfulness}, we see $\psi$ is indeed bijective and the functor $(G,\psi)$ gives a desired exact equivalence $\CA/\CN\xto{\sim}\Im G$.
\end{proof}

By the discussion so far, we summarize our first main result as follows.

\begin{corollary}\label{cor:algebraic_subquotient}
Let $(\CA,\BE,\fs)$ be an algebraic extriangulated category with ${\rm (WIC)}$ and $\CN$ an extension-closed subcategory $\CN\sse\CA$ which is closed under direct summands.
Then the subquotient $\CA_{\CN}/\CN$ is still an algebraic extriangulated category.
\end{corollary}
\begin{proof}
Consider the extriangulated substructure $\CA_\CN=(\CA_\CN,\BE_\CN,\fs_\CN)$ determined by $\CN$, see \cref{prop:modified_Quillen_substructure}.
Then we have an admissible pair $(\CA_\CN,\CN)$ by \cref{thm:subquotient_of_extri_cat}.
Thanks to \cref{prop:bijection_between_substructures}, we have an exact dg enhancement $\A_\N=(\A_\N,\SS_\N)$ of $\CA_\CN$ and consider the associated bounded derived category $\Db(\A_\N)$.
Lastly, by \cref{prop:algebraic_subquotient}, the extriangulated subquotient $\CA_\CN/\CN$ is exact equivalent to an extension-closed subcategory of $\frac{\Db(\A_\N)}{\thick\CN}$.
Hence $\CA_\CN/\CN$ is algebraic by \cref{prop:characterization_for_alg_ET}(1).
\end{proof}

We end the section by demonstrating that \cref{cor:algebraic_subquotient} generalizes the known stability of being algebraic.

\begin{example}\label{ex:algebraic_subquotient}
Let us fix an algebraic extriangulated category $(\CA,\BE,\fs)$.
\begin{enumerate}[label=\textup{(\arabic*)}]
\item 
In \cite[\S 3.5]{Che24b}, it is shown that the extriangulated ideal quotient $\CA/[\CN]$ by an full subcategory $\CN\sse\CA$ of projective-injective objects is still algebraic.
This is a special case of \cref{cor:algebraic_subquotient}, because $(\CA,\CN)$ satisfies the condition \ref{spade}, see \cref{ex:spade}(1).
\item 
Suppose that $(\CA,\BE,\fs)$ corresponds to an exact category and $\CN$ is biresolving in $\CA$.
We can deduce from the proof of \cite[Thm.~5]{Rum21} that the triangulated quotient $\CA/\CN$ is an extension-closed subcategory of a certain derived category, and hence algebraic.
Due to \cref{ex:spade}(2), \cref{cor:algebraic_subquotient} is an extriangulated version of such a statement.
\end{enumerate}
\end{example}

\section{A dg enhancement of the extriangulated subquotient}
\label{sec:subquotient}
As we have seen in \cref{prop:enhancement_of_Verdier_quotient}, the Drinfeld dg quotient enhances the Verdier quotient.
In this section,  we broaden the scope of this phenomenon by introducing a new notion, the \emph{exact dg subquotient} of an exact dg category, which yields an enhancement of the extriangulated subquotient.
To this end, in \S\ref{subsec:cohomological_envelope}, we introduce a notion of cohomological envelope $\A^\coh$ of an exact dg category $\A$ and establish basic properties.
\S\ref{subsec:dg_quotient_via_cohomological_envelope} collects fundamental facts concerning dg quotients of $\A^\coh$ and demonstrate their connections to the Serre quotient.
The section \S\ref{subsec:Drinfeld_subquotient} is devoted to stating the main result of this article, namely, we establish the enhanced extriangulated subquotient.

\subsection{Cohomological envelope of an exact dg category}
\label{subsec:cohomological_envelope}
Our notion of cohomological envelope makes a given exact dg category $\A$ to match better with the dg quotient.
We start with the following basic observation which is easily deduced from Tabuada's characterizations of the dg quotient.

\begin{lemma}\label{lem:failure_dg_quotient_of_connective_dg}
Let $\A$ be an arbitrary dg category and $\N\sse\A$ a full dg subcategory. 
If $\A$ is connective, then the natural morphism
$$H^0(\A/\N) \to H^0(\A)/[H^0(\N)]$$
is an isomorphism in $\Hqe$.
\end{lemma}
\begin{proof}
We denote by $h(\Cat_k) \sse \Hqe$ the full subcategory consisting of objects whose cohomologies of $\Hom$-complexes are concentrated in degree $0$. It suffices to show that a canonical morphism induces an isomorphism
$$h(\Cat_k)(H^0(\A/\N), -) \cong h(\Cat_k)(H^0\A/[H^0\N], -).$$
Notice that the functor $H^0\colon \A\to H^0\A$ is a morphism in $\dgcat$ by the connectivity of $\A$.
Thus, indeed we have the following isomorphisms for any $\C \in h(\Cat_k)$:
\begin{align*}
h(\Cat_k)(H^0(\A/\N), \C) 
&\cong \Hqe(\A/\N, \C) \\
&\cong \Hqe_{\N}(\A, \C) \\
&\cong \Hqe_{H^0(\N)}(H^0\A, \C) \\
&\cong h(\Cat_k)(H^0\A/[H^0\N], \C).
\end{align*}
The second isomorphism follows from \cite[Thm~4.0.1]{Tab10}, where $\Hqe_{\N}(\A, \C)$ denotes the set of morphisms which annihilate $\N$.
\end{proof}

In the rest of the section, we will work under this setup to investigate an exact dg enhancement for extriangulated quotient.

\begin{setup}\label{setup:cohomological_envelope}
Let $(\A,\SS)$ be an exact dg category and assume that it is cofibrant.
The universal embedding is denoted by $F\colon \A\to\Db(\A,\SS)$ as in \S\ref{subsec:universal_embedding}.
Moreover, we consider an extension-closed subcategory $\N\sse\A$.
\end{setup}

If, in addition to assumptions of \cref{lem:failure_dg_quotient_of_connective_dg}, we further assume that the pair $(H^0\A,H^0\N)=(\CA,\CN)$ is admissible in the sense that an exact sequence \eqref{seq:mult_loc} exists in $\ET$, then the following issue arises.
Since $\CA/\CN$ does not necessarily agree with $\CA/[\CN]$, we cannot in general expect the dg quotient of a connective dg category to enhance the extriangulated quotient.
To bypass such an obstruction, we introduce the following notion based on the existence of the universal embedding.

\begin{definition}\label{def:cohomological_envelope}
Let $\tau_{\leq 0}\A\to\A$ be the connective cover of $\A$ and consider the universal embedding $F\colon \tau_{\leq 0}\A\to\Db_{\dg}(\tau_{\leq 0}\A)$.
There exists an extension-closed dg subcategory $\A^\coh\deff\D'\sse \Db_{\dg}(\A)$ equipped with an exact quasi-equivalence $\tau_{\leq 0}\A\to\tau_{\leq 0}(\A^\coh)$ by \cref{thm:universal_embedding}.
We call the corresponding exact dg category $(\A^\coh,\SS)$ the \emph{cohomological envelope} of $\A$.
\end{definition}

It turns out that there exist canonical exact morphisms $\A^\coh\leftarrow\tau_{\leq 0}\A\to\A$ in $\Hqe$ which fall into exact equivalences in $\ET$ after applying $H^0$.
Recall that each conflation in $\SS$ involves only morphisms in degree $0$ and $-1$.
Hence the cohomological envelope $\A^\coh$ shares the class $\SS$ with the original dg category $\A$ via the above canonical morphisms above.
The stability of taking cohomological envelope is then immediate.
    
\begin{lemma}\label{lem:stability_of_cohomological_envelope}
The operation of taking the cohomological envelope is stable in the sense that there exists a natural isomorphism $\A^\coh\xto{\simeq}(\A^\coh)^\coh$ in $\Hqe$.
\end{lemma}
\begin{proof}
We may suppose that $\A$ is connective.
We consider the universal embedding $F\colon \A\to\Db_{\dg}(\A)$ and the extension-closed subcategory $\A^\coh\sse\Db_{\dg}(\A)$ with a quasi-equivalence $\A\xto{\sim}\tau_{\leq 0}(\A^\coh)$.
This yields a quasi-equivalence $\Db_{\dg}(\A)\xto{\sim}\Db_{\dg}(\tau_{\leq 0}(\A^\coh))$ which restricts to a desired quasi-equivalence $\A^\coh\xto{\sim}(\A^\coh)^\coh$.
\end{proof}

An exact dg category $(\A,\SS)$ is said to be \emph{cohomological} if it is exact quasi-equivalent to the cohomological envelope of an exact dg category.
We also say that $(\A,\SS)$ is a \emph{cohomological} exact dg enhancement of $\CA=H^0\A$.
Obviously, any algebraic extriangulated category admits a cohomological exact dg enhancement.

\begin{remark}\label{rem:cohomological_env}
\begin{enumerate}
\item 
The cohomological envelope already appears implicitly in \cite{Che23b}, where it is used---without being named---to prove the main theorem.
This is presumably because it arises in a rather primitive manner.
However, to highlight its significance in the present work, we introduce a specific terminology for it.
\item 
The terminology is inspired by \cite{CX12}, where a subcategory 
$\CX\sse\CD(B)$ is called homological if it is compatible with a homological ring epimorphism 
$B\to C$ in the sense of $\CX\simeq \CD(C)$, and hence with the induced higher extensions.
Our notion is derived from this idea, but we adopt the adjective ``cohomological'' to match the cohomological conventions used throughout this article.
\end{enumerate}

\end{remark}

For later use, we collect some basic properties of $(\A^\coh,\SS)$.

\begin{lemma}\label{lem:cohomological_envelope_as_functor}
The operation of taking cohomological envelope gives rise to a functor $(-)^\coh\colon \Hqe_{\rm ex}\to \Hqe_{\rm ex}$.
Moreover, it restricts to an equivalence $\Hqe_{\rm ex}^{\rm cn}\xto{\sim}\Hqe_{\rm ex}^{\rm coh}$, where we denote by $\Hqe_{\rm ex}^{\rm cn}$ (resp. $\Hqe_{\rm ex}^{\rm coh}$) the full subcategory of $\Hqe_{\rm ex}$ consisting of connective (resp. cohomological) exact dg categories.
\end{lemma}
\begin{proof}
We have already had two endofunctors on $\Hqe_{\rm ex}$, namely, taking the connective cover and constructing the bounded dg derived category that we denote by $\tau_{\leq 0}(-)$ and $\Db_{\dg}(-)$ respectively.
The composite sends a morphism $F\colon \A\to \B$ to $F'\colon \Db_{\dg}(\tau_{\leq 0}\A)\to \Db_{\dg}(\tau_{\leq 0}\B)$.
We see that $F'$ restricts to $F^\coh\colon \A^\coh\to\B^\coh$, because the quasi-essential image of $\A$ under $F$ is contained in $\B^\coh$.

It is easily confirmed that its restriction is an equivalence:
In fact we have a natural isomorphism $\tau_{\leq 0}(\A^\coh)\xto{\sim} \A$ exists for any $\A\in\Hqe_{\rm ex}^{\rm cn}$ as well as a natural isomorphism $(\tau_{\leq 0}\A)^\coh\xto{\sim}\A$ for any $\A\in\Hqe_{\rm ex}^{\rm coh}$.
We thus finish the proof.
\end{proof}

\begin{remark}
Let $F\colon \A\to\B$ be a morphism of cohomological exact dg categories in $\Hqe$.
Then, it must be exact.
In fact, it is the restriction of the canonical morphism $\pretr\A\to\pretr\B$ along the inclusions.
Hence we have $\Hqe_{\rm ex}^{\rm coh}=\Hqe^{\rm coh}$.
\end{remark}

The next proposition permits us to understand the bounded dg derived category $\Db_{\dg}(\A)$ of $\A$ as the pretriangulated hull of the cohomological envelope $\A^\coh$.

\begin{proposition}\label{prop:universal_embedding_and_cohomological_envelope}
Assume that $\A$ is connective.
Then there exists a natural isomorphism $\pretr(\A^\coh)\xto{\simeq} \Db_{\dg}(\A)$ in $\Hqe$.
Thus we also call the canonical morphism $F^\coh\colon \A^\coh\hookrightarrow\Db_{\dg}(\A)$ the \emph{pretriangulated hull} of $\A^\coh$.
\end{proposition}
\begin{proof}
Recall from \cref{thm:universal_embedding} that a universal exact morphism $F\colon \A\to\Db_{\dg}(\A)$ to a pretriangulated category exists.
We firstly check the exactness of the canonical morphism $\A\xto{\can} \pretr(\A^\coh)$.
Since $3$-term h-complexes in $\A^\coh$ belong to $\A$ via the canonical inclusion $\A\to \A^\coh$, we are reduced to check the exactness of $\A^\coh\hookrightarrow\pretr(\A^\coh)$.
By taking the pretriangulated hulls for the fully faithful inclusion $\A^{\coh}\hookrightarrow \Db_{\dg}(\A)$, we get the following commutative diagram in $\Hqe$ in which the arrows $\hookrightarrow$ denote the canonical inclusions, see \cref{rem:quasi-equivalence}.
\begin{equation}\label{diag:universal_embedding_and_cohomological_envelope}
\begin{tikzcd}[column sep=1.4cm, row sep=1.0cm]
    \A\arrow{r}{}\arrow{rd}[swap]{F}&\A^\coh \arrow[hook]{r}{}\arrow[hook]{d}[swap]{F^\coh}[swap]{}& \pretr(\A^\coh)\arrow[hook]{ld}{\inc}\\
    &\Db_{\dg}(\A) & {}
\end{tikzcd}
\end{equation}
Since both subcategories $\A^\coh$ and $\pretr(\A^\coh)$ are extension-closed in $\Db_{\dg}(\A)$, so is $\A^\coh$ in $\pretr(\A^\coh)$.
Thus \eqref{diag:universal_embedding_and_cohomological_envelope} gives a sequence $\A^\coh\sse\pretr(\A^\coh)\sse\Db_{\dg}(\A)$ of extension-closed subcategories.
In particular, a $3$-term h-complex $X$ in $\A^\coh$ is homotopy short exact if and only if so is it in $\Db_{\dg}(\A)$.

Next we shall that the above morphism $\inc\colon \pretr(\A^\coh)\to\Db_{\dg}(\A)$ is an isomorphism.
Since $\A\xto{\can}\pretr(\A^\coh)$ is an exact morphism to a pretriangulated dg category, by the universality of $F\colon \A\to\Db_{\dg}(\A)$, we have a morphism $G\colon \Db_{\dg}(\A)\to\pretr(\A^\coh)$ with $\can=G\circ F$ as depicted below.
\begin{equation*}
\begin{tikzcd}[column sep=1.2cm, row sep=1.0cm]
    \A \arrow{r}{\can}\arrow{d}[swap]{F}& \pretr(\A^\coh)\arrow[hook]{ld}[swap]{\inc}\\
        \Db_{\dg}(\A)\arrow[dotted,bend right]{ru}[swap]{G} & {}
\end{tikzcd}
\end{equation*}
Again by the universality of $F$, we conclude $\inc\circ G=\id$ in $\Hqe$ which forces $H^0(\inc) \colon \tr(\A^{\coh})\to \Db(\A)$ to be essentially surjective. Hence $\inc\colon \pretr(\A^{\coh})\to \Db_{\dg}(\A)$ is a quasi-equivalence.
\end{proof}

By the above discussion, we see that a cohomological exact dg category $\A$ is a ``generating'' extension-closed subcategory of a triangulated category in the following sense.

\begin{lemma}\label{lem:char_of_cohomological_exact_dg_cat}
Let $\T$ be a pretriangulated dg category and $\A\sse\T$ an extension-closed subcategory (with a natural exact structure $\SS$).
If $\tr(\A)=H^0\T$, then $(\A,\SS)$ is cohomological.
\end{lemma}
\begin{proof}
The inclusion $\A\sse\T$ can be factorized as the composite $\A\hookrightarrow \pretr(\A)\hookrightarrow \T$ in $\Hqe$.
If $\tr(\A)=H^0\T$, we have a quasi-equivalence $\pretr(\A)\hookrightarrow\T$.
Recall that the bounded derived category $\Db(\A)$ is the quotient of $\tr(\A)$ by the subcategory $\CM$ generated by the defective objects, see \cref{def:defective_object}.
Since any conflation in $\A$ is also a conflation in $\T$, we have $\CM=0$.
Hence $\A$ is cohomological.
\end{proof}

Thanks to the characterization in \cref{lem:char_of_cohomological_exact_dg_cat}, we can find examples of cohomological exact dg categories.
The first one is an incarnation of cohomological exact dg categories.

\begin{example}\label{ex:pvd_smc1}
Let us consider a finite dimensional $k$-algebra $A$.
The canonical inclusion $\mod A\hookrightarrow \Db(\mod A)$ from the module category to its derived category provides a prototypical example of cohomological exact dg category.
Precisely, we consider a natural enhancement of $\Db(\mod A)$ via the equivalence $\mathcal{K}^{{\rm b},-}(\proj A)\simeq \Db(\mod A)$, where $\mathcal{K}^{{\rm b},-}(\proj A)$ is the homotopy category of complexes bounded above with finite total cohomology.
Then, the natural dg enhancement of $\mod A\sse\mathcal{K}^{{\rm b},-}(\proj A)$ is a cohomological exact dg category.
\end{example}

A generalization of \cref{ex:pvd_smc1} for a suitable dg algebra exists under the notion of simple minded collection.

\begin{example}\label{ex:pvd_smc2}
Let $\Lambda$ be a proper connective dg $k$-algebra over a field $k$ and consider the perfect valued derived category $\pvd(\Lambda)\sse \CD(\Lambda)$, i.e.,
\[
\pvd(\Lambda)=\Set{X\in\CD(\Lambda) | \sum_i\dim_k H^i(X)<\infty}.
\]
Also, we consider the set $\CS$ of simple $H^0(\Lambda)$-modules in it.
The dg derived category $\Ddg(\Lambda)$ naturally defines an enhancement $\pvd_\dg(\Lambda)$ of $\pvd(\Lambda)$.
Since $\CS$ is a simple minded collection of $\pvd(\Lambda)$ in the sense of \cite[Def.~3.2]{KY14}, the thick closure of $\mod H^0(\Lambda)\sse\pvd(\Lambda)$ coincides with $\pvd(\Lambda)$.
By \cref{lem:char_of_cohomological_exact_dg_cat}, we have the cohomological exact dg enhancement of $\mod H^0(\Lambda)$ whose bounded derived category is $\pvd(\Lambda)$.

We should notice that, if we regard $\mod H^0(\Lambda)$ as a dg category concentrated in degree $0$, then its pretriangulated hull is the usual bounded derived category $\Db(\mod H^0(\Lambda))$ after taking $H^0$, see \cref{ex:univ_emb_exact_cat}.
It is not necessarily true that $\pvd(\Lambda)$ is equivalent to $\Db(\mod H^0(\Lambda))$.
\end{example}

Another example arises from Cohen-Macaulay dg modules studied in \cite{Jin20}.

\begin{example}\label{ex:Cohen-Macaulay1}
\cite[Prop.~3.29]{Che24b}
Let $\Lambda$ be a dg $k$-algebra as in \cref{ex:pvd_smc2}.
We additionally assume that $\Lambda$ is Gorenstein in the sense that $\per(\Lambda)$ is generated by $D\Lambda\in\CD(\Lambda)$, where $D=\Hom_k(-,k)$ is the standard $k$-dual.
We consider the full subcategory $\mathsf{CM}(\Lambda)\sse\CD(\Lambda)$ of \emph{Cohen-Macaulay} dg $\Lambda$-modules, that is,
\[
\mathsf{CM}(\Lambda)=\Set{X\in\CD(\Lambda) | H^i(X)=0\textup{\ and\ }\Hom_{\CD(\Lambda)}(M,A[i])=0\textup{\ for\ }i>0 }.
\]
Since $\mathsf{CM}(\Lambda)$ is extension-closed in $\CD(\Lambda)$, the dg derived category $\Ddg(\Lambda)$ defines an exact dg enhancement $\mathsf{CM}_{\dg}(\Lambda)$ to $\mathsf{CM}(\Lambda)$.
By \cite[Lem.~3.9(2)]{Jin20}, we see $\pvd(\Lambda)$ is generated by $\mathsf{CM}(\Lambda)$.
Hence $\mathsf{CM}_{\dg}(\Lambda)$ is cohomological and its pretriangulated hull is quasi-equivalent to $\pvd_{\dg}(\Lambda)$.
\end{example}

Passing to the cohomological envelope, we can interpret the higher extensions $\BE^n$ over the extriangulated category $H^0(\A)=H^0(\A^\coh)$ as the shifted $\Hom$-spaces via its pretriangulated hull $\Db_{\dg}(\A)=\pretr(\A^\coh)$ by \cref{prop:higher_extension_via_univ_emb}.

Lastly we would like to mention that $\Hqe_{\rm ex}^{\rm coh}$ admits a closed monoidal structure.
Chen has already proved that $\Hqe_{\rm ex}^{\rm cn}$ admits such a structure with respect to a certain tensor product \cite[\S 4]{Che24b}.
Thus, we have the statement by \cref{lem:cohomological_envelope_as_functor}.

\subsection{A dg quotient via the cohomological envelope}
\label{subsec:dg_quotient_via_cohomological_envelope}
Now we put the pretriangulated hull $F^\coh\colon \A^\coh\hookrightarrow \Db_{\dg}(\A)$ defined in \cref{prop:universal_embedding_and_cohomological_envelope} to use.
We establish a dg quotient of $\A^\coh$ compatible with that of $\Db_{\dg}(\A)$ depending on basic observations in \S\ref{subsec:Hqe}.

\begin{setup}\label{setup:dg_quotient_via_cohomological_envelope}
We keep \cref{setup:cohomological_envelope} and assume that $\A$ is connective through this subsection.
Also, we are given an extension-closed subcategory $\N\sse\A$ which defines the extension-closed subcategory $\CN\sse\CA$ under $H^0$ by \cref{ex:extension_closed_subcat}.
Thus, the sequence $\CN\sse\CA\sse\Db(\A)$ of extension-closed subcategories is available.
\end{setup}

Recall that we are assuming $\A$ to be cofibrant.
Hence the dg quotient $\A\to\A/\N$ is represented by a functor in $\dgcat$.
We can say more that the dg quotient $\A/\N$and the pretriagulated hull $\pretr(\A)$ remain flat.

Following the arguments in \S\ref{subsec:subquotients_are_algebraic}, we consider the thick closure $\thick\CN$ of $\CN\sse\Db(\A)$ and denote by $\thick_{\dg}\N$ its canonical dg enhancement in $\Db_{\dg}(\A)$.
Thus, we have the dg quotient which enhances the Verdier quotient of $\Db(\A)$ by $\thick\CN$ as the following diagram indicates (compare \eqref{seq:ambient_Verdier_quotient}).

\begin{equation}\label{seq:ambient_Verdier_quotient2}
\begin{tikzcd}[column sep=1.0cm]
    \thick_{\dg}\N \arrow{r}&
    \Db_{\dg}(\A)\arrow{r}{Q} &
    \dfrac{\Db_{\dg}(\A)}{\thick_{\dg}\N}
\end{tikzcd}
\quad
\overset{H^0}{\rightsquigarrow}
\quad
\begin{tikzcd}[column sep=1.0cm]
    \thick\CN \arrow{r}&
    \Db(\A)\arrow{r}{H^0Q} &
    \dfrac{\Db(\A)}{\thick\CN}
\end{tikzcd}
\end{equation}

Since $(\A^\coh,\SS)$ is an enhancement of $\CA$, the extension-closed dg subcategory $\N\sse\A$ naturally determines that of $\A^\coh$, namely, the canonical dg enhancement of $\CN\sse H^0(\A^\coh)$.
With abuse of notation, we still use the same symbol to present the situation such as $(\N,\SS|_\N)\sse (\A^\coh,\SS)$.
It has been already seen in \cref{cor:restriction_of_dg_quotient} that the dg quotient $\A^\coh/\N$ is realized as a dg subcategory of $\frac{\Db_{\dg}(\A)}{\thick_{\dg}\N}$.
We rephrase their interaction accordingly to our current context.

\begin{corollary}\label{cor:restriction_of_dg_quotient2}
Let us denote by $\Im(Q|_{\A^\coh})$ the quasi-essential image of $Q\circ F^\coh$ and consider the following commutative diagram in $\Hqe$.
\begin{equation}\label{diag:restriction_of_dg_quotient2}
\begin{tikzcd}[column sep=1.5cm]
\N\arrow{r}{}\arrow{d}{} & \A^\coh \arrow{r}{Q|_{\A^\coh}}\arrow[hook]{d}[swap]{F^\coh}&
    \Im(Q|_{\A^\coh}) \arrow[hook]{d}{\inc} \\
    \thick_{\dg}\N \arrow[hook]{r}&
    \Db_{\dg}(\A)\arrow{r}{Q} &
    \dfrac{\Db_{\dg}(\A)}{\thick_{\dg}\N}
\end{tikzcd}
\end{equation}
Then the following assertions hold.
\begin{enumerate}[label=\textup{(\arabic*)}]
\item 
The canonical inclusion $\inc$ is the pretriangulated hull of $\Im(Q|_{\A^\coh})$.
\item 
The restricted quasi-functor $Q|_{\A^\coh}\colon \A^\coh\to \Im(Q|_{\A^\coh})$ is a dg quotient of $\A^\coh$ by $\N$.
In particular, we have an isomorphism $\Im(Q|_{\A^\coh})\cong \A^\coh/\N$ in $\Hqe$.
\end{enumerate}
Moreover, if $\Im(Q|_{\A^\coh})$ is extension-closed in $\frac{\Db_{\dg}(\A)}{\thick_{\dg}\N}$, then it is cohomological and that the upper row yields a sequence $\CN\to\CA\to H^0(\A^\coh/\N)$ in $\ET$.
\end{corollary}
\begin{proof}
The only difference between our current situation and \cref{cor:restriction_of_dg_quotient} is that we have taken the thick closure $\thick\CN$ instead of $\Tria\CN$.
This does not cause any change on the quotient $\frac{\Db_{\dg}(\A)}{\thick_{\dg}\N}$, so the items (1) and (2) are immediate from \cref{cor:restriction_of_dg_quotient}.

We have to show the last assertion.
For simplicity, we still use the same symbol $Q$ to denote the restriction $Q|_{\A^\coh}$ and put $\B=\A^\coh/\N$.
We are currently regarding $\B$ as an extension-closed subcategory of $\frac{\Db_{\dg}(\A)}{\thick_{\dg}\N}$, so it is enough to show that the natural morphism $\tau_{\leq 0}\B\to\frac{\Db_{\dg}(\A)}{\thick_{\dg}\N}$ is a universal exact morphism to a pretriangulated dg category.
If we consider an exact morphism $G\colon \tau_{\leq 0}\B\to\T$ to a pretriangulated dg category $\T$, the composite $G\circ \tau_{\leq 0}Q\colon \tau_{\leq 0}(\A^\coh)=\A\to\T$ uniquely factors through $\Db_{\dg}(\A)$ by \cref{thm:universal_embedding}.
Since the induced unique morphism $\Db_{\dg}(\A)\to\T$ annihilates $\thick\CN$ under $H^0$, we have a desired unique morphism $G'\colon \frac{\Db_{\dg}(\A)}{\thick_{\dg}\N}\to\T$.
We omit the full details here, because the argument is done in the same fashion as \cref{cor:restriction_of_dg_quotient}.
\end{proof}

Lastly in this subsection, we apply \cref{cor:restriction_of_dg_quotient2} to the setting of Serre quotients.
Let $\CA$ be an abelian category and $\CN\sse\CA$ a Serre subcategory.
Let us consider trivial exact dg enhancements $(\N,\SS|_\N)\sse(\A,\SS)$, that is, dg categories concentrated in degree $0$.
They do not necessarily enhance the Serre quotient $\CA/\CN$, because the dg quotient $\A/\N$ yields the ideal quotient $\CA/[\CN]$ by \cref{lem:failure_dg_quotient_of_connective_dg}.
The cohomological envelope circumvents this obstruction.

\begin{proposition}\label{prop:enhanced_Serre}
Let us consider the above $\CN\sse\CA$ with the trivial enhancements $(\N,\SS|_\N)\sse(\A,\SS)$.
Then the exact dg subquotient $(\A^\coh/\N,\wtil{\SS_\N})$ defines a cohomological exact enhancement of the Serre quotient $\CA/\CN$.
\end{proposition}
\begin{proof}
By \cite[Theorem 3.2]{Miy91}, we have exact (sub)sequences in $\ET$ as follows,
\begin{equation}\label{diag:enhanced_Serre}
\begin{tikzcd}[column sep=1.5cm]
\CN\arrow{r}{}\arrow{d}{} & \CA \arrow{r}{Q'|_{\CA}}\arrow[hook]{d}[swap]{}&
    \CA/\CN \arrow[hook]{d}{\inc} \\
    \thick\CN \arrow[hook]{r}&
    \Db(\CA)\arrow{r}{Q'} &
    \Db(\CA/\CN)
\end{tikzcd}
\end{equation}
where we notice that $\thick\CN$ is the full subcategory of $\Db(\CA)$ consisting of objects whose cohomologies sit in $\CN$.
Notice that we can regard $\Im Q'|_\CA\cong \CA/\CN$ as an extension-closed subcategory in $\frac{\Db(\CA)}{\thick\CN}$.

Following \cref{ex:pvd_smc1}, we consider a canonical enhancement of $\Db(\CA)$.
By \cref{ex:univ_emb_exact_cat}, the inclusion $\CA\hookrightarrow\Db(\CA)$ is lifted to the universal embedding $F\colon \A\hookrightarrow\Db_{\dg}(\A)$ with $\Db(\A)\simeq\Db(\CA)$.
Thus we have $\Im F=\A^\coh$ by definition, and a natural enhancement $\N$ of $\CN$.
Since the bottom row is lifted to the dg quotient
\[
\begin{tikzcd}
\thick_{\dg}\N\arrow{r}{}&
\Db_{\dg}(\A)\arrow{r}{Q}&
\frac{\Db_{\dg}(\A)}{\thick_{\dg}\N}
\end{tikzcd},
\]
we have the diagram same as \eqref{diag:restriction_of_dg_quotient2} by \cref{cor:restriction_of_dg_quotient2}.

We have to only show that $\Im(Q|_{\A^{\coh}})$ is extension-closed in $\frac{\Db_{\dg}(\A)}{\thick_{\dg}\N}$.
However, it is immediate since $\Im Q'|_{\CA}\cong \CA/\CN$ is an extension-closed subcategory of $\Db(\CA/\CN)$ in \eqref{diag:enhanced_Serre}.
Hence $\Im(Q|_{\A^{\coh}})\cong \A^\coh/\N$ is an enhancement of the Serre quotient $\CA/\CN$.
\end{proof}

\subsection{An exact dg subquotient}
\label{subsec:Drinfeld_subquotient}
We now accomplish our goal of establishing
an exact dg enhancement of the extriangulated subquotient.
Roughly speaking, this will be achieved by placing the arguments in \S\ref{subsec:extriangulated_subquotient} into an exact dg setting.
We keep the situation of the previous subsection: a connective cofibrant exact dg category $(\A,\SS)$ and an extension-closed subcategory $(\N,\SS|_\N)\sse (\A,\SS)$ are fixed.
Their associated extriangulated categories are denoted by $(\CN,\BE|_\CN,\fs|_\CN)\sse (\CA,\BE,\fs)$.
Recall from \cref{thm:subquotient_of_extri_cat} that an extriangulated subquotient $\CA_\CN/\CN$ exists under the 
${\rm (WIC)}$ assumption.
To construct an enhancement of $\CA_\CN/\CN$, we shall introduce the notion of \emph{exact dg subquotient} of $\A$ by $\N$.

The first step of the construction is passing to an exact dg substructure determined by $\N$, as described in the next lemma.

\begin{lemma}\label{lem:dg_substructure}
We denote by $\A_\N=(\A_\N,\SS_\N)$ the exact dg substructure corresponding to the extriangulated substructure $(\CA_\CN,\BE_\CN,\fs_\CN)$ via \cref{prop:bijection_between_substructures}.
Then, both exact dg structures $(\N,\SS|_\N)$ and $(\N,\SS_\N|_\N)$ coincide.
In particular, $(\N,\SS_\N|_\N)$ is an extension-closed subcategory of $(\A_\N,\SS_\N)$.
\end{lemma}
\begin{proof}
By taking $H^0$, the canonical inclusion $(\N,\SS_\N|_\N)\hookrightarrow (\N,\SS|_\N)$ produces a coincidence $(\CN,\BE_\CN,\fs_\CN)=(\CN,\BE,\fs)$ by \cref{thm:subquotient_of_extri_cat}(0).
\cref{prop:bijection_between_substructures} guarantees $(\N,\SS_\N|_\N)=(\N,\SS|_\N)$. 
\end{proof}

The next step is taking the cohomological envelope of $(\A_\N,\SS_\N)$ which is denoted by $\A_\N\to (\A_\N^\coh,\SS_\N)$.
We should mention that the cohomological envelope depends on the exact dg structure that we focus on, so $\A_\N^\coh$ is different to the cohomological envelope $\A^\coh$ of $(\A,\SS)$ in general.
Since the cohomological envelope does not make any change on their homotopy categories, the extension-closed subcategory $(\N,\SS|_\N)\sse (\A_\N^\coh,\SS_\N)$ still produces the extension-closed subcategory $(\CN,\BE|_\CN,\fs|_\CN)\sse (\CA_\CN,\BE_\CN,\fs_\CN)$.
The following main theorem provides a way to realize the extriangulated subquotient $\CA_\CN/\CN$ via the dg quotient of $\A_\N^\coh$ by $\N$.

\begin{theorem}\label{thm:dg_subquotient}
Let us consider the exact dg substructure $(\A_\N,\SS_\N)$ determined by $\N$ and the cohomological envelope $\A_\N\hookrightarrow (\A_\N^\coh,\SS_\N)$.
If we additionally suppose that $\CA$ is ${\rm (WIC)}$ and the extension-closed subcategory $\CN\sse\CA$ is closed under direct summands, then the following hold.
\begin{enumerate}[label=\textup{(\arabic*)}]
\item 
The dg quotient of $\A_\N^\coh$ by $\N$ inherits a natural exact dg structure from $(\A_\N^\coh, \SS_\N)$.
In particular, the associated dg quotient gives rise to a sequence of exact morphisms in $\Hqe$:
\begin{equation*}
(\N,\SS|_{\N})\lra (\A_\N^\coh,\SS_\N)\overset{Q}{\lra} (\A_\N^\coh/\N,\wtil{\SS_\N}).
\end{equation*}
Moreover, the exact dg category $(\A_\N^\coh/\N,\wtil{\SS_\N})$ is cohomological.
\item 
The dg quotient $Q\colon \A_\N^\coh\to \A_\N^\coh/\N$ (equipped with exact structures) enhances the extriangulated subquotient of $\CA$ by $\CN$ as exhibited below:
\begin{equation}\label{seq:dg_subquotient}
\begin{tikzcd}
    (\N,\SS|_\N) \arrow{r}
    & (\A_\N^\coh,\SS_\N)\arrow{r}{Q}\arrow[Rightarrow, shorten >=0.5ex, shorten <=0.5ex]{d}{H^0} 
    & (\A_\N^\coh/\N,\wtil{\SS_\N})
    \\
    (\CN,\BE|_\CN,\fs|_\CN) \arrow{r}
    & (\CA_\CN,\BE_\CN,\fs_\CN)\arrow{r}{H^0Q} 
    & (\CA_\CN/\CN,\wtil{\BE_\CN},\wtil{\fs_\CN})
\end{tikzcd}
\end{equation}
\end{enumerate}
We call the upper sequence of \eqref{seq:dg_subquotient} the \emph{exact dg subquotient} of an exact dg category $(\A,\SS)$ by its extension-closed subcategory $(\N,\SS|_\N)$.
\end{theorem}

To simplify our discussions, we suppose that the extriangulated category $H^0\A=(\CA,\BE,\fs)$ satisfies \ref{spade} in \cref{setup:spade}.
Thanks to \cref{prop:comparison_of_quotient_and_subquotient}, we have nothing lost in our purpose.
In this case, we have $(\A_\N,\SS_\N)=(\A,\SS)$ and proceed the proof under such an assumption.

\begin{proof}[Proof of \cref{thm:dg_subquotient}]
Due to the condition \ref{spade}, the diagram \eqref{seq:dg_subquotient} is rewritten as
\[
\begin{tikzcd}
    (\N,\SS|_\N) \arrow{r}
    & (\A^\coh,\SS)\arrow{r}{Q}\arrow[Rightarrow, shorten >=0.5ex, shorten <=0.5ex]{d}{H^0} 
    & (\A^\coh/\N,\wtil{\SS})
    \\
    (\CN,\BE|_\CN,\fs|_\CN) \arrow{r}
    & (\CA,\BE,\fs)\arrow{r}{H^0Q} 
    & (\CA/\CN,\wtil{\BE},\wtil{\fs})
\end{tikzcd}
\]
We shall impose a natural exact dg structure on the dg quotient $\A^\coh/\N$.
We consider the pretriangulated hull $F^\coh\colon \A^\coh\hookrightarrow \Db_{\dg}(\A)$ of $\A^\coh$ such as \cref{prop:universal_embedding_and_cohomological_envelope} and obtain the following diagram \eqref{diag:restriction_of_dg_quotient} in \cref{cor:restriction_of_dg_quotient},
\begin{equation}\label{diag:dg_subquotient_proof}
\begin{tikzcd}[column sep=1.5cm]
    \N \arrow{r}{}\arrow{d}{}
    & \A^\coh \arrow{r}{\wtil{Q}|_{\A^\coh}}\arrow[hook]{d}[swap]{F^\coh}
    & \Im(\wtil{Q}|_{\A^\coh}) \arrow[hook]{d}{\inc}
    \\
    \thick_{\dg}\N \arrow{r}
    & \Db_{\dg}(\A)\arrow{r}{\wtil{Q}} 
    & \dfrac{\Db_{\dg}(\A)}{\thick_{\dg}\N}
\end{tikzcd}
\end{equation}
where $\Im(\wtil{Q}|_{\A^\coh})$ is isomorphic to the dg quotient $\A^\coh/\N$.
We have to verify that $\Im(Q|_{\A^\coh})$ is extension-closed in $\frac{\Db_{\dg}(\A)}{\thick_{\dg}\N}$.
Under taking $H^0$, the above diagram \eqref{diag:dg_subquotient_proof} falls into the following commutative diagram in $\ET$.
\begin{equation*}\label{diag:dg_subquotient_proof2}
\begin{tikzcd}[column sep=1.5cm]
    \CN \arrow{r}{}\arrow{d}{}
    & \CA \arrow{r}{(H^0\wtil{Q})|_{\CA}}\arrow[hook]{d}[swap]{}
    & \Im((H^0\wtil{Q})|_{\CA}) \arrow[hook]{d}{\inc}
    \\
    \thick\CN \arrow{r}
    & \Db(\A)\arrow{r}{H^0\wtil{Q}} 
    & \dfrac{\Db(\A)}{\thick\CN}
\end{tikzcd}
\end{equation*}
Since \ref{spade} is assumed, we know $\Im((H^0\wtil{Q})|_{\CA})$ is extension-closed in $\frac{\Db(\A)}{\thick\CN}$ and is exact equivalent to the extriangulated quotient $\CA/\CN$ by \cref{prop:algebraic_subquotient}.
We have thus explained the diagram displayed at the beginning of the proof.
It follows from \cref{cor:restriction_of_dg_quotient2}(2) that $\A^\coh/\N$ is a cohomological exact dg enhancement of $\CA/\CN$.
\end{proof}

As we have seen in the proof of \cref{thm:dg_subquotient}, the condition \ref{spade} dovetails the extriangulated (sub)quotient with the dg (sub)quotient.

\begin{corollary}\label{cor:dg_subquotient_closed_under_spade}
Let $(\A,\SS)$ be a cohomological exact dg category and $(\N,\SS|_\N)$ its extension-closed subcategory.
Assume that the corresponding pair $(\CA,\CN)=(H^0\A,H^0\N)$ satisfies the conditions in \cref{setup:spade} which contains $\ref{spade}$.
Then we have the following assertions.
\begin{enumerate}[label=\textup{(\arabic*)}]
\item 
The dg quotient $\A/\N$ inherits a natural exact structure from $(\A,\SS)$.
Moreover, the exact dg category $(\A/\N,\wtil{\SS})$ is still cohomological.
\item 
We have an exact equivalence $\CA/\CN\simeq H^0(\A/\N,\wtil{\SS})$ in $\ET$.
In particular, $\CA/\CN$ is algebraic.
\item 
We have the following diagram of dg quotients whose vertical arrows denote pretriangulated hulls.
\[
\begin{tikzcd}[column sep=1.5cm]
    \N \arrow{r}{}\arrow[hook]{d}{}
    & \A \arrow{r}{\wtil{Q}|_{\A}}\arrow[hook]{d}[swap]{F}
    & \A/\N \arrow[hook]{d}{}
    \\
    \thick_{\dg}\N \arrow{r}
    & \Db_{\dg}(\A)\arrow{r}{\wtil{Q}} 
    & \dfrac{\Db_{\dg}(\A)}{\thick_{\dg}\N}
\end{tikzcd}
\]
Under taking $H^0$, we have associated exact (sub)sequences in $\ET$.
\end{enumerate}
\end{corollary}
\begin{proof}
Since the subcategory $\CN\sse\CA$ satisfies \ref{spade} and $\A$ is cohomological, we have $(\A,\SS)=(\A,\SS_\N)=(\A_\N^{\coh},\SS_\N)$.
The first equality says that taking the substructure with respect to $\N$ does not make any changes on $(\A,\SS)$, and the second one follows from the stability of cohomological exact dg structures, see \cref{lem:stability_of_cohomological_envelope}.
The assertions are immediate from \cref{thm:dg_subquotient}.
\end{proof}

As a benefit of the cohomological enhancement of $\CA_\CN/\CN$, we have a good interaction of the higher extensions under $H^0Q\colon \CA_\CN\to\CA_\CN/\CN$.
We denote by $\wtil{\BE_\CN^n}(C,A)$ the higher extension in $(\CA_\CN/\CN,\wtil{\BE_\CN},\wtil{\fs_\CN})$.
The following corollary says that any element of $\wtil{\BE_\CN^n}(C,A)$ is represented by a `roof' in $\CA_\CN$.

\begin{corollary}\label{cor:interaction_of_higher_extensions}
We consider the extriangulated quotient $(H^0Q,\mu)\colon (\CA_\CN,\BE_\CN,\fs_\CN)\to (\CA_\CN/\CN,\wtil{\BE_\CN},\wtil{\fs_\CN})$ which is enhanced by the exact dg subquotient \eqref{seq:dg_subquotient}.
Let $A,C\in\CA_\CN$ be given.
Then, the higher extension $\wtil{\BE_\CN^n}(C,A)$ is bijectively corresponds to the coset of pairs
\begin{equation}\label{eq:interaction_of_higher_extensions}
\bigl\{ s\backslash\delta = [C\overset{\delta}{\dra} A'\overset{s}{\longleftarrow}A] \mid A'\in\CA_\CN, s\in\Sn, \delta\in\BE_\CN^n(C,A) \bigr\}/\sim,
\end{equation}
where the relation $s\backslash\delta\sim s'\backslash\delta'$ means the existence of a morphism $A'\xto{u}A''$ in $\Sn$ such that $us=s'$ and $u_*\delta=\delta'$.
The situation is depicted as follows.
\begin{equation}\label{diag:interaction_of_higher_extensions}
\begin{tikzcd}
C\arrow[dashed]{r}{\delta}\arrow[dashed]{rd}[swap]{\delta'}&A'\arrow{d}{u}&A\arrow{l}[swap]{s}\arrow{ld}{s'}\\
&A''&
\end{tikzcd}
\end{equation}
\end{corollary}
\begin{proof}
Since the enhancement of $\CA_\CN/\CN$ is cohomological, the canonical inclusion $\CA_\CN/\CN\hookrightarrow\wtil{\CD}=\frac{\Db(\A)}{\thick\CN}$ preserves the higher extensions in the sense that there exists a $\delta$-functor isomorphism $\wtil{\BE_\CN^n}(C,A)\xto{\simeq}\Ext_{\wtil{\CD}}^n(C,A)$.
Via this isomorphism, any element $\alpha\in\wtil{\BE_\CN^n}(C,A)$ is represented by a left fraction $s[n]\backslash \delta$ in $\Db(\A)$ with the following diagram
\[
\begin{tikzcd}[column sep=1.2cm]
    {}
    & C \arrow[dotted]{ld}[swap]{\alpha}\arrow{d}{\delta}
    & {}
    & {}
    \\
    A[n] \arrow{r}[swap]{s[n]}
    & A'[n] \arrow{r}[swap]{}
    & N[n] \arrow{r}
    & A[1]
\end{tikzcd}
\]
where the bottom sequence is a triangle and $N[n]\in\thick\CN$.
By \cref{lem:description_for_extensions}, we may assume $N\in\CN$ and $A'\in\CA_\CN$.
It turns out that $\alpha$ is represented by the pair $(s,\delta)$ of $s\in\Sn$ and $\delta\in\BE_\CN^n(C',A)$.
We can check that the assignment $\alpha\mapsto s\backslash\delta$ gives rise to an injective map.
The surjectivity is also immediate.
Actually, if we are given a pair $(s,\delta)$ of $s\in\Sn$ and $\delta\in\BE_\CN^n(C,A')$, we get the desired element $\alpha=(s^{-1}[n])\delta\in\Ext_{\wtil{\CD}}^n(C,A)$.
\end{proof}

An extriangulated quotient under the condition \ref{spade} contains the ideal quotient by projective-injectives and the quotient by biresolving subcategories (\cref{ex:algebraic_subquotient}), so we have their cohomological enhancements.

\begin{example}\label{ex:dg_subquotient_closed_under_spade}
Consider an exact dg category $(\A,\SS)$ and its extension-closed subcategory $(\N,\SS|_\N)$, and denote by $(\CN,\BE|_\CN,\fs|_\CN)\sse(\CA,\BE,\fs)$ the corresponding extriangulated categories.
In addition, we need to assume that $\CA$ is {\rm (WIC)} and $\CN$ is closed under direct summands.
\begin{enumerate}[label=\textup{(\arabic*)}]
\item 
Suppose that $\CN$ consists of projective-injective objects in $\CA$.
Then the pair $(\CA,\CN)$ satisfies \ref{spade}, and hence the ideal quotient $\CN\lra \CA\lra \CA/[\CN]$ is enhanced by the dg quotient $\N\lra \A^\coh\lra \A^\coh/\N$. 
We should remark that, in \cite[\S 3.5]{Che24b}, a connective exact dg enhancement for the ideal quotient $\CA/[\CN]$ was established in a similar fashion: 
If $\A$ is connective, the (usual) dg quotient $\A/\N$ indeed enhances $\CA/[\CN]$.
\item 
Suppose $\CN\sse\CA$ is a biresolving subcategory.
Then the pair $(\CA,\CN)$ satisfies \ref{spade}, and hence the extriangulated quotient $\CA/\CN$ admits a cohomological exact dg enhancement $\A^\coh/\N$.
\end{enumerate}
\end{example}

We give a closer look at Cohen-Macaulay dg modules of \cref{ex:Cohen-Macaulay1} in view of dg subquotient.

\begin{example}\label{ex:Cohen-Macaulay2}
Let $\Lambda$ be a proper connective Gorenstein dg $k$-algebra as in \cref{ex:Cohen-Macaulay1}.
Recall that $\mathsf{CM}_{\dg}(\Lambda)$ denotes a canonical dg enhancement determined by the ambient dg derived category $\Ddg(\Lambda)$.
Since the bounded dg derived category of $\mathsf{CM}_{\dg}(\Lambda)$ is $\pvd_{\dg}(\Lambda)$, it is cohomological.
Let us denote by $\add_{\dg}(\Lambda)$ the canonical enhancement of $\add\Lambda\sse\mathsf{CM}(\Lambda)$.
Recall from \cite[Thm.~4.2(1)]{Jin20} that $\add\Lambda$ is the subcategory of projective-injectives in the Frobenius extriangualated category $\mathsf{CM}(\Lambda)$.
Thus, the ideal quotient
\[
\begin{tikzcd}
\add(\Lambda)\arrow{r}{} &\mathsf{CM}(\Lambda)\arrow{r}{} &\underline{\mathsf{CM}}(\Lambda)
\end{tikzcd}
\]
is the extriangulated (sub)quotient of $\mathsf{CM}(\Lambda)$ by $\add(\Lambda)$.
By \cref{prop:comparison_of_quotient_and_subquotient}, the extriangulated subquotient $\underline{\mathsf{CM}}(\Lambda)$ is nothing other than the extriangulated quotient $\mathsf{CM}(\Lambda)/\add(\Lambda)$.
Note that the thick closure of $\add(\Lambda)$ in $\pvd(\Lambda)$ forms $\per(\Lambda)$.
By \cref{thm:dg_subquotient} we have the following dg quotient with their ambient bounded dg derived categories:
\[
\begin{tikzcd}[column sep=1.5cm]
    \add_{\dg}(\Lambda) \arrow{r}{}\arrow[hook]{d}{}
    & \mathsf{CM}_{\dg}(\Lambda) \arrow{r}{}\arrow[hook]{d}[swap]{}
    & \underline{\mathsf{CM}}_{\dg}(\Lambda) \arrow[hook]{d}{}
    \\
    \per_{\dg}(\Lambda) \arrow{r}
    & \pvd_{\dg}(\Lambda)\arrow{r}{} 
    & \frac{\pvd_{\dg}(\Lambda)}{\per_{\dg}(\Lambda)}
\end{tikzcd}
\]
We include more explanations on the above dg quotients.
The dg category $\underline{\mathsf{CM}}_{\dg}(\Lambda)$ denotes the dg (sub)quotient of $\mathsf{CM}_{\dg}(\Lambda)$ by $\add_{\dg}(\Lambda)$ and is an exact dg enhancement of the extriangulated (sub)quotient $\underline{\mathsf{CM}}(\Lambda)$.
Since the sigularity category $\CD_{\rm sg}(\Lambda)$ is defined to be the Verdier quotient $\frac{\pvd(\Lambda)}{\per(\Lambda)}$, the dg quotient $\frac{\pvd_{\dg}(\Lambda)}{\per_{\dg}(\Lambda)}$ is an exact dg enhancement of $\CD_{\rm sg}(\Lambda)$.
By \cref{cor:dg_subquotient_closed_under_spade}(3), the vertical arrows in the diagram are pretriangulated hulls.
We know from \cite[Thm.~2.4(3)]{Jin20} that the rightmost one is an isomorphism in $\Hqe$ which induces an equivalence $\underline{\mathsf{CM}}(\Lambda)\simeq\CD_{\rm sg}(\Lambda)$.
\end{example}

Closing the section, we show that Chen's enhanced ideal quotient can be obtained by using the exact dg (sub)quotient.
We need to recall the following from \cite[Thm.~3.23]{Che24b}.

\begin{proposition}\label{prop:Chen's_enahanced_ideal_quotient}
Let $(\CA,\CP)$ be the pair of an extriangulated category $(\CA,\BE,\fs)$ and its extension-closed subcategory $(\CP,\BE|_\CP,\fs|_\CP)$.
Assume that $\CA$ admits a connective exact dg enhancement $(\A,\SS)$ and denote by $(\P,\SS|_\P)$ the canonical enhancement of $\CP$.
Then, the dg quotient $\A/\P$ carries a canonical exact dg structure $(\A/\P,\overline{\SS})$ induced from $(\A,\SS)$ and its bounded dg derived category is quasi-eqivalent to the canonical enhancement of $\Db(\A)/\tr\P$.
\end{proposition}

First we should mention that $\tr\P$ is indeed a triangulated subcategory of $\Db(\A)$.
Recall from \cref{def:defective_object} that $\Db(\A)$ is defined to be the dg quotient of $\tr\A$ by the subcategory $\CM$ generated by the defective objects, namely, $\Db(\A)=\tr\A/\CM$.
By \cite[Lem.~3.20]{Che24b} we have $\Hom_{\tr\A}(\CP,\CM)=0$ and thus the inclusion $\tr\P\hookrightarrow\Db(\A)$ still exists as below.
\[
\begin{tikzcd}
\CM\arrow{r}{} &\tr\A\arrow{r}{} &\Db(\A)\\
&\tr\P\arrow[hook]{u}{}\arrow[hook]{ru}{}&
\end{tikzcd}
\]
Besides, the above diagram tells us $\tr\P=\Tria\CP$ holds in $\Db(\A)$.

Now we are able to consider the dg quotient $\Db(\A)/\tr\P$ as given in the statement.
Furthermore, taking cohomological envelope $\A^\coh$ produces the following diagram by \cref{thm:dg_subquotient} and \cref{ex:dg_subquotient_closed_under_spade}(1),
\[
\begin{tikzcd}
\P\arrow{r}{}\arrow{d}{}
&\A^\coh\arrow{r}{Q}\arrow[hook]{d}{F^\coh}
&\A^\coh/\P\arrow[hook]{d}{\inc}
\\
\Tria_{\dg}\P\arrow{r}{}
&\Db_{\dg}(\A)\arrow{r}{Q}
&\dfrac{\Db_{\dg}(\A)}{\Tria_{\dg}\P}
\end{tikzcd}
\]
where we use the same symbol $\P$ to denote the canonical enhancement of $\CP\sse H^0\A^\coh$.
Due to \cref{prop:Chen's_enahanced_ideal_quotient}, we know that the connective exact dg category $\A/\P$ admits the universal embedding $\A/\P\to \Db_{\dg}(\A)/\pretr\P=\Db_{\dg}(\A)/\Tria_{\dg}\P$.
Thus the above diagram guarantees an isomorphism $(\A/\P)^\coh\cong\A^\coh/\P$ that we want.

\section{Enhanced heart construction}
\label{sec:enhanced_heart_construction}
It is well-known that a $t$-structure $(\CT^{\leq 0},\CT^{\geq 0})$ over a triangulated category $\CT$ induces a cohomological functor $H$ from $\CT$ to the abelian heart $\CH=\CT^{\leq 0}\cap\CT^{\geq 0}$ \cite{BBD}.
Such a phenomenon was generalized in an extriangulated context in \cite{LN19}.
We will understand the heart construction via the notion of exact dg subquotient.
Throughout the section, $(\CA,\BE,\fs)$ be an extriangulated category.

\subsection{Hearts of cotorsion pairs}
\label{subsec:hearts_of_cotorsion_pairs}
We give a quick review on the heart $\underline{\CH}$ of a given cotorsion pair $(\CU,\CV)$ over $(\CA,\BE,\fs)$.
The notion of cotorsion pair goes back to \cite{Sal79} and was placed in triangulated categories \cite{IY08,Nak11}.
An extriangulated version is defined in the same fashion, that is, the pair $(\CU,\CV)$ of full additive subcategories of $\CA$ which are closed under direct summands is called a \emph{cotorsion pair} if the conditions 
$\BE(\CU,\CV)=0$ and $\cone(\CV,\CU)=\cocone(\CV,\CU)=\CA$ hold.

\begin{definition}\label{def:heart_of_cotorsion_pair}
\cite[\S 2]{LN19}
Given a cotorsion pair $(\CU,\CV)$, we put $\CW=\CU\cap\CV$ and
introduce the following associated subcategories in $\CA$:
\[
\CA^+=\cone(\CV,\CW),
\quad\CA^-=\cocone(\CW,\CU)
\quad\textup{and}\quad
\CH=\CA^+\cap\CA^-.
\]
The ideal quotient $\underline{\CH}=\CH/[\CW]$ is called the \emph{heart} of $(\CU,\CV)$.
\end{definition}

Like the case of $t$-structures, any cotorsion pair induces a natural cohomological functor to the heart.
The next theorem is an extriangulated generalization of \cite[Thm.~5.7]{AN12}.

\begin{theorem}\label{thm:heart}
\cite[Thm.~3.2, 3.5]{LN19}
Let $(\CU,\CV)$ be a cotorsion pair over $(\CA,\BE,\fs)$.
The heart $\underline{\CH}$ is an abelian category and there exists an associated cohomological functor $H\colon (\CA,\BE,\fs)\to \underline{\CH}$.
\end{theorem}

Since we are going to inspect the above heart construction via a lens of an exact dg category, we will recall a minimal amount to write down the proof.
We need to know how to constitute the above cohomological functor.

\begin{lemma}\label{lem:coreflection}\cite[Def.~2.14, Prop.~2.19]{LN19}
For any $X\in\CC$, there exists an $\fs$-triangle
\begin{equation}\label{diag:coreflection}
\begin{tikzcd}
V_X\arrow{r}{} &X^-\arrow{r}{\alpha_X} &X\arrow[dashed]{r}{} &{}
\end{tikzcd}
\end{equation}
where $V_X\in\CV, X^-\in\CC^-$ and $\BE(\CU,X^-)\xto{(\alpha_X)_*}\BE(\CU,X)$ is bijective.
Such a triangle is called a \emph{coreflection $\fs$-triangle} of $X$.
Dually, there exists a \emph{reflection $\fs$-triangle} of $X$:
\begin{equation}\label{diag:reflection}
\begin{tikzcd}
X\arrow{r}{\beta_X} &X^+\arrow{r}{} &U_X\arrow[dashed]{r}{} &{}
\end{tikzcd}.
\end{equation}
It is an $\fs$-triangle satisfying $U_X\in\CU,X^+\in\CC^+$ and $\BE(X^+,\CV)\xto{(\beta_X)^*}\BE(X,\CV)$ is bijective.
\end{lemma}

Notice that the morphism $X^-\xto{\alpha_X}X$ in \eqref{diag:coreflection} is a right $(\CC^-)$-approximation of $X$.
The dual statement holds for \eqref{diag:reflection}.
Also, both assignments $X\mapsto X^+$ and $X\mapsto X^-$ give rise to adjoint functors of the canonical inclusions \cite[\S 2.1]{LN19}.
The cohomological functor $H$ is defined to be the composite of them.

\begin{lemma}\label{lem:composite_of_adjoint_functors}
(cf. \cite[Def.~2.21, 2.34]{LN19})
We have the following assertions:
\begin{enumerate}[label=\textup{(\arabic*)}]
\item 
The canonical inclusion $\underline{\CA}^+\hookrightarrow \underline{\CA}$ admits a right adjoint $\sigma^+$.
\item 
Dually, the canonical inclusion $\underline{\CA}^-\hookrightarrow \underline{\CA}$ admits a left adjoint $\sigma^-$.
\item 
Furthermore, there exists a natural isomorphism $\sigma^+\circ \sigma^-\cong \sigma^-\circ \sigma^+$ which permits us to put $H\deff\sigma^-\circ \sigma^+$.
\end{enumerate}
\end{lemma}

The above lemma is presented as the following commutative diagram up to isomorphisms.
This also says that both functors $\sigma^+$ and $\sigma^-$ restrict to $\underline{\CA}^-\xto{\sigma^+}\underline{\CH}$ and $\underline{\CA}^+\xto{\sigma^-}\underline{\CH}$, respectively.

\[
\begin{tikzcd}[row sep = 0.5cm, column sep = 1.0cm]
&\underline{\CA}^+\arrow[hook]{rd}{}\arrow[bend right]{ld}[swap]{\sigma^-}&\\
\underline{\CH}\arrow[hook]{ru}{}\arrow[hook]{rd}{}&&\underline{\CA}\arrow[bend right]{lu}[swap]{\sigma^+}\arrow[bend left]{ld}{\sigma^-}\\
&\underline{\CA}^-\arrow[hook]{ru}{}\arrow[bend left]{lu}{\sigma^+}&
\end{tikzcd}
\]

For later use, we make some investigations on the kernel of the cohomological functor $H$.

\begin{lemma}\label{lem:coreflection2}
(cf. \cite[Lem.~5.9]{Oga24})
Let $f\colon A\to B$ be a morphism in $\CC$ and consider reflection triangles (\ref{diag:reflection}) of $A$ and $B$.
If the induced morphism $f^+\colon A^+\to B^+$ factors through an object in $\CW$, then $f$ factors through an object in $\CU$.
\end{lemma}
\begin{proof}
Since $\beta_X$ is a left $(\CC^+)$-approximation of $X=A,B$, we get a commutative diagram below.
\[
\xymatrix{
A\ar[r]^{\beta_A}\ar[d]_{f}&A^+\ar[d]^{f^+}\\
B\ar[r]^{\beta_B}&B^+
}
\]
If we resolve $B$ into an $\fs$-triangle $V_B\overset{v}{\lra} U_B\overset{u}{\lra} B\overset{\eta}{\dra}$ by $B\in\cone(\CV,\CU)$, we can verify that $f$ factors through $U_B\in\CU$:
To show this, it suffices to check $f^*\eta=0$ in $\BE(A,V_B)$.
The composite $\beta_B\circ f$ gives rise to a sequence $\BE(B^+,V_B)\xto{(\beta_B)^*}\BE(B,V_B)\xto{f^*}\BE(A,V_B)$ with the first morphism $(\beta_B)^*$ surjective by the definition of \eqref{diag:reflection}.
Thus we get an element $\eta'\in\BE(B^+,V_B)$ such that $(\beta_B\circ f)^*\eta'=(f^+\circ \beta_A)^*\eta'=f^*\eta$.
Now we are supposing $f^+$ factors through an object in $\CW$ and has just verified $f^*\eta=0$ by the orthogonality $\BE(\CW,\CV)=0$.
Finally, the exact sequence $\CA(A,U_B)\xto{u\circ -} \CA(A,B)\to \BE(A,V_B)$ gives a desired morphism in $\CA(A,U_B)$.
\end{proof}

\begin{proposition}\label{prop:kernel_heart}
(cf. \cite[Prop.~5.10]{Oga24})
The following hold for the kernel $\CN:=\Ker H$.
\begin{enumerate}
\item[\textnormal{(1)}]
$\CN=\add(\CU*\CV)$ is true.
\item[\textnormal{(2)}]
A morphism $f$ in $\CC$ factors through an object in $\CN$ if and only if $H(f)=0$.
\end{enumerate}
\end{proposition}
\begin{proof}
We can find the item (1) in \cite[Cor. 3.8]{LN19}.
We have to prove only the `if' part in the item (2), so $H(f)=0$ is supposed for a morphism $A\xto{f}B$ in $\CC$.
By taking coreflection $\fs$-triangles of $A, B$ and reflection $\fs$-triangles of $A^-, B^-$ successively, we get the following commutative squares labeled $(-)$ and $(+)$,
\[
\begin{tikzcd}[row sep=1.0cm, column sep=1.0cm]
A\arrow[mysymbol]{rd}[description]{(-)}\arrow{d}[swap]{f} &A^-\arrow{r}{\beta_{A^-}}\arrow{l}[swap]{\alpha_A}\arrow{d}[description]{f^-} &A^\pm\arrow{d}{f^\pm} \\
B &B^-\arrow{r}{\beta_{B^-}}\arrow{l}[swap]{\alpha_B} &B^\pm\arrow[mysymbol]{lu}[description]{(+)} 
\end{tikzcd}
\]
where we abbreviate as $X^\pm\deff({X^-})^+$.
By definition of $H$, $f^\pm$ factors through an object in $\CW$.
Lemma \ref{lem:coreflection2} tells us that $f^-$ factors through an object $U\in\CU$, say $f^-\colon A^-\xto{g}U\xto{h}B^-$.
Let us consider an $\fs$-triangle $A\overset{a}{\lra} V_A\overset{b}{\lra} U_A\overset{\delta}{\dra}$ associated to $\CA=\cocone(\CV,\CU)$.
The functor $\BE(U_A,-)$ transfers $(-)$ to the following commutative diagram,
\[
\begin{tikzcd}
\BE(U_A,A)\arrow{dd}[swap]{f_*} &\BE(U_A,A^-)\arrow{l}[swap]{(\alpha_A)_*}\arrow{dd}{f^-_*}\arrow{rd}{g_*} &{}
\\
& &\BE(U_A,U)\arrow{ld}{h_*}
\\
\BE(U_A,B) &\BE(U_A,B^-)\arrow{l}[swap]{(\alpha_B)_*} &{}
\end{tikzcd}
\]
where the horizontal arrows are isomorphisms by the definition of coreflection $\fs$-triangles.
By the factorization $f^-=h\circ g$, we get an $\fs$-triangle $U\overset{u_2}{\lra} U'\overset{u_1}{\lra} U_A\overset{\eta}{\dra} $ which corresponds to $\eta=g_*((\alpha_A)_*^{-1}\delta)$.
Thus we see $f_*\delta=(\alpha_B\circ h)_*\eta$ in $\BE(U_A,B)$.
Now, taking a weak pullback of $b$ along $u_1$, we have the commutative diagram made of $\fs$-triangles:
\[
\begin{tikzcd}
{}
&U\arrow[equal]{r}{}\arrow{d}[swap]{x}
&U\arrow{d}{u_2}
&{}
\\
A\arrow{r}{}\arrow[equal]{d}{}
&K\arrow{r}{}\arrow{d}[swap]{y}
&U'\arrow[dashed]{r}{u_1^*\delta}\arrow{d}{u_1}
&{}
\\
A\arrow{r}{a}
&V_A\arrow{r}{b}\arrow[dashed]{d}[swap]{b^*\eta}
&U_A\arrow[dashed]{r}{\delta}\arrow[dashed]{d}{\eta}\wPB{lu}
&{}
\\
&{}&{}&
\end{tikzcd}
\]
We have to show that $f$ factors through $K\in\CU*\CV$.
This follows from $f_*(u_1^*\delta)=u_1^*(f_*\delta)=u_1^*((\alpha_B\circ h)_*\eta)=(\alpha_B\circ h)_*(u_1^*\eta)=0$.
\end{proof}

Now we are ready to realize the heart $\underline{\CH}$ as a localization.
It is based on the well-known fact that a right/left adjoint $\sigma$ of a canonical inclusion is the Gabriel-Zisman localization with respect to a suitable morphisms, see \cite{GZ67}.
Thus, we have a localization description of the cohomological functor $H\colon \CA\to\underline{\CH}$.

\begin{proposition}\label{prop:heart_is_the_Gabriel_Zisman_localization}
Let $\CS_H$ be the class of morphisms $s$ such that $H(s)$ is an isomorphism in $\underline{\CH}$.
Then, we have a canonical equivalence $H'\colon \CA[\CS_H^{-1}]\xto{\sim}\underline{\CH}$.
\end{proposition}
\begin{proof}
As stated, the functor $\sigma^+\colon \underline{\CA}\to\underline{\CA}^+$ is the Gabriel-Zisman localization with respect to the morphisms $s$ such that $\sigma^+(s)$ is isomorphism in $\underline{\CA}^+$.
Since the ideal quotient is a localization, e.g. \cite[Ex.~2.6]{Oga22}, the assertion follows from $H=\sigma^-\circ\sigma^+$.
\end{proof}

\subsection{Hearts as an extriangulated subquotient}
\label{subsec:Hearts_as_an_extriangulated_subquotient}
We keep the setup in the previous subsection, namely, a cotorsion pair $(\CU,\CV)$ over $(\CA,\BE,\fs)$ are given.
We will show that the associated cohomological functor $H\colon \CA\to\underline{\CH}$ can be considered as the extriangulated subquotient of $\CA$.
To this end, we have to additionally assume the ${\rm (WIC)}$ condition on $\CA$.
Since the kernel of $H$ coincides with $\CN=\add(\CU*\CV)$ by \cref{prop:kernel_heart}(1), it is natural to consider the  extriangulated subcategory $\CA_\CN=(\CA_\CN,\BE_\CN,\fs_\CN)$ determined by the extension-closed subcategory $\CN$, see \cref{prop:modified_Quillen_substructure}.

The first theorem in \S\ref{sec:enhanced_heart_construction} is stated as follows.

\begin{theorem}\label{thm:heart_is_the_subquotient}
There exists a natural exact equivalence $(F,\phi)$ from the extriangulated subquotient $\CA_\CN/\CN$ to the heart $\underline{\CH}$.
\end{theorem}

Due to \cref{thm:subquotient_of_extri_cat}, we know the extriangulated subquotient $\CA_\CN\xto{(Q,\mu)}\CA_\CN/\CN$ always exists.
In addition, the universality of the quotient $(Q,\mu)$ shows the desired exact functor exists as below.

\begin{lemma}\label{lem:cohomological_to_exact}
The cohomological functor $H$ gives rise to an exact functor $(H,\psi)$ with respect to the substructure $(\CA_\CN,\BE_\CN,\fs_\CN)$.
In particular, we have a unique exact functor $(F,\phi)\colon \CA_\CN/\CN\to \underline{\CH}$ with the commutativity $(H,\psi)=(F,\phi)\circ (Q,\mu)$.
\end{lemma}
\begin{proof}
As stated, we have to only show that $H$ sends any $\fs_\CN$-triangle $A\overset{f}{\lra} B\overset{g}{\lra} C\overset{\delta}{\dra} $ to the short exact sequence in $\underline{\CH}$.
Since $H$ is cohomological, we shall only check that $H(g)$ is epic.
The monomorphicity of $H(f)$ is just the dual.

By the definition of $(\CU,\CV)$, we can resolve $A$ in the sense of considering an $\fs$-triangle $A\overset{a}{\lra} V\lra U\dra$ with $U\in\CU$ and $V\in\CV$.
Obviously we get $a_*\delta\in\BE_\CN(C,V)$.
Due to the definition of the subbifunctor $\BE_\CN$, we obtain the following morphisms of $\fs_\CN$-triangles,
\begin{equation*}\label{diag:cohomological_to_exact}
\begin{tikzcd}[column sep=1.2cm]
A\arrow{d}{a}\arrow{r}{f}\wPO{rd} 
&B\arrow{d}{b}\arrow{r}{g} 
&C\arrow[equal]{d}{}\arrow[dashed]{r}{\delta} 
&{}
\\
V\arrow[equal]{d}{}\arrow{r}{} 
&B'\arrow{d}{b'}\arrow{r}{g'} 
&C\arrow{d}{c}\arrow[dashed]{r}{a_*\delta=c^*\eta} 
&{}
\\
V\arrow{r}{} 
&N'\arrow{r}{} 
&N\arrow[dashed]{r}{\eta}\wPB{ul} 
&{}
\end{tikzcd}
\end{equation*}
where $N,N'\in\CN$.
This tells us that $H(g)$ is an epimorphism.
Actually we get a factorization $g=g'\circ b$ and the middle column forces $H(b)$ to be an epimorphism by $H(N')=0$.
Also, we know the square labeled by ${\rm (wPB)}$ is folded as an $\fs$-triangle by \cref{lem:wPO_wPB} and $H(g')$ is an epimorphism.
\end{proof}

We notice that the underlying category of the subquotient $\CA_\CN/\CN$ is constituted as the Gabriel-Zisman localization at the class $\CS_\CN=\Def_\CN^\sp\circ\Inf_\CN$, see \cref{lem:factorization_of_Sn}.
As below, we confirm that $\CS_\CN$ is the class $\CS_H$ of all morphisms which become isomorphisms in the heart $\underline{\CH}$ and the additive equivalence $F\colon \CA_\CN/\CN\xto{\sim}\underline{\CH}$ exists.

\begin{lemma}\label{lem:heart_is_the_subquotient}
The above obtained functor $F\colon \CA_\CN/\CN\to\underline{\CH}$ is an equivalence of additive categories.
\end{lemma}
\begin{proof}
The inclusion $\CS_\CN\sse\CS_H$ is obvious by the commutativity $H=F\circ Q$ in \cref{lem:cohomological_to_exact}, we consider a morphism $A\xto{s}A'$ in $\CS_H$ to check the converse.
Using the cotorsion pair $(\CU,\CV)$, we resolve the object $A$ to the $\fs$-triangle $A\overset{f}{\lra} V\overset{g}{\lra} U\overset{\delta}{\dra} $ with $U\in\CU$ and $V\in\CV$.
By taking a weak pushout of $f$ along $s$, we have the folded $\fs$-triangle $A\overset{\begin{bsmallmatrix}f \\
s
\end{bsmallmatrix}}{\lra} V\oplus A'\overset{h}{\lra} B\dra $.
Since $H\begin{bsmallmatrix}f \\
s
\end{bsmallmatrix}$ is still an isomorphism in $\underline{\CH}$, we get $H(h)=0$ and that $h$ factors through an object $N\in\CN$ by \cref{prop:kernel_heart}(2), say $h\colon V\oplus A'\xto{h_1}N\xto{h_2}B$.
A weak pullback of $h$ along $h_2$ yields the following morphism of $\fs$-triangle,
\[
\begin{tikzcd}
A\arrow{r}{s'}\arrow[equal]{d}{} &A''\arrow{r}{}\arrow{d}[swap]{s''} &N\arrow[dashed]{r}{}\arrow{d}{h_2} &{}\\
A\arrow{r}{\begin{bsmallmatrix}f \\
s
\end{bsmallmatrix}} &V\oplus A'\arrow{r}{h} &B\arrow[dashed]{r}{}\wPB{ul} &{}
\end{tikzcd}
\]
where $s''$ is a retraction.
Since $H$ is cohomological, $H(s')$ is an epimorphism by $H(N)=0$.
By the factorization $\begin{bsmallmatrix}f \\
s
\end{bsmallmatrix}=s''\circ s'$ on the leftmost square, we have both $H(s')$ and $H(s'')$ are isomorphisms.
It is obvious the kernel of $s''$ belongs to $\CN=\Ker H$, which shows $s''\in\Def_\CN^{\sp}$.
Last, we have only to show $s'\in\Inf_\CN$.
Similarly to the above, we again consider a weak pushout of $f$ along $s'$ and get the following commutative diagram made of $\fs$-triangles.
\[
\begin{tikzcd}
A\arrow{r}{f}\arrow{d}[swap]{s'}\wPO{rd} 
&V\arrow{r}{g}\arrow{d}{} 
&U\arrow[dashed]{r}{}\arrow[equal]{d}{} 
&{}
\\
A''\arrow{r}{}\arrow{d}{}
&N'\arrow{r}{}\arrow{d}{} 
&U\arrow[dashed]{r}{} 
&{}
\\
N\arrow[equal]{r}{}\arrow[dashed]{d}{}&N\arrow[dashed]{d}{}&&
\\
{}&{}&{}&{}
\end{tikzcd}
\]
We notice that $N'\in\CN$ and the second column shows the first column is indeed an $\fs_\CN$-triangle.
We thus conclude $s\in\Sn$ and $\Sn=\CS_H$.
\end{proof}

\begin{proof}[Proof of \cref{thm:heart_is_the_subquotient}]
It remains to show that the exact functor $(F,\phi)\colon \CA_\CN/\CN\to\underline{\CH}$ is an exact equivalence.
The extriangulated structure on $\underline{\CH}$ corresponds to the abelian category and is obviously maximal (see \cite{Rum11, Eno18, Eno21} for more details).
Due to the additive equivalence $F\colon\CA_\CN/\CN\xto{\sim}\underline{\CH}$ and \cref{lem:exact_equivalence}, we only have to show that the extriangulated structure $(\CA_\CN/\CN,\wtil{\BE_\CN},\wtil{\fs_\CN})$ is maximal.
Let us consider a short exact sequence $\epsilon\colon 0\to A\xto{\alpha} B\xto{\beta} C\to 0$ in the abelian category $\CA_\CN/\CN$.
To show it is isomorphic to the image of an $\fs_\CN$-conflation under the quotient $\CA_\CN\xto{Q}\CA_\CN/\CN$, we consider a left fraction $B\xto{b}C'\xleftarrow{s}C$ with $s\in\Sn$ which represents $\beta$.
By the cotorsion pair $(\CU,\CV)$, there is an $\fs$-conflation $V\overset{v}{\lra} U\overset{u}{\lra} C'$ with $U\in\CU$ and $V\in\CV$.
Taking a weak pushout of $u$ along $b$, we have an (folded) $\fs$-conflation $A'\lra U\oplus B\overset{\begin{bsmallmatrix}u \\
b
\end{bsmallmatrix}}{\lra} C'$ which is isomorphic to the short exact sequence $\epsilon$ under the quotient.
In turn, we should show the following claim.

\begin{claim}\label{claim:heart_is_the_subquotient}
Let $A\overset{f}{\lra} B\overset{g}{\lra} N\overset{\delta}{\dra}$ be an $\fs$-triangle in $\CA$.
If its image under the quotient $Q$ forms a short exact sequence in $\CA_\CN/\CN$, then it is an $\fs_\CN$-triangle.
\end{claim}
\begin{proof}
By the definition of $\fs$-triangle, we are reduced to the case of $N\in\CN$.
Under such an assumption, we see $Q(f)$ is an isomorphism and $f\in\Sn$ by the saturatedness stated in \cref{cor:Sn_is_saturated}.
Also, $f$ is factorized as $f=f_2\circ f_1$ with $f_2\in\Def_\CN$ and $f_1\in\Inf_\CN^{\sp}$ by \cref{lem:factorization_of_Sn}.
We then apply \cite[dual of Prop.~3.17]{NP19} to obtain the following commutative diagram made of $\fs$-triangles,
\[
\begin{tikzcd}
&N''\arrow[equal]{r}{}\arrow{d}{h}&N''\arrow{d}{}&\\
A\arrow[equal]{d}{}\arrow{r}{f_1}&A\oplus N'\arrow{d}{f_2}\arrow{r}{}&N'\arrow[dashed]{r}{}\arrow{d}{}&{}\\
A\arrow{r}{f}&B\arrow{r}{g}\arrow[dashed]{d}{}&N\arrow[dashed]{d}{\eta}\arrow[dashed]{r}{\delta}&{}\\
&{}&{}&
\end{tikzcd}
\]
with $h_*\eta=(f_1)_*\delta$.
Since $f_1=\begin{bsmallmatrix}1 \\
0
\end{bsmallmatrix}$ is a section, taking a weak pushout along the corresponding retraction $\begin{bsmallmatrix}1 & 0
\end{bsmallmatrix}$ yields $\begin{bsmallmatrix}1 & 0
\end{bsmallmatrix}_*h_*\eta=\begin{bsmallmatrix}1 \\ 0
\end{bsmallmatrix}_*\begin{bsmallmatrix}1 & 0
\end{bsmallmatrix}_*\delta=\delta$.
More precisely, the induced morphism $(\begin{bsmallmatrix}1 \\ 0
\end{bsmallmatrix}\circ h)_*\colon \BE(N,N'')\to \BE(N,A)$ maps $\eta$ to $\delta$, which shows $\delta\in\BE_\CN(N,A)$.
Combining the dual argument, we have $\delta\in\BE_\CN(N,A)$.
\end{proof}
Therefore any short exact sequence in $\CA_\CN/\CN$ is isomorphic to an $\wtil{\fs_\CN}$-triangle, which shows $(F,\phi)$ is an exact equivalence.
\end{proof}

Immediately, we have the following commutative diagram up to isomorphisms,
\begin{equation}\label{diag:cohomological_functor_via_quotient}
\begin{tikzcd}
(\CA,\BE,\fs)\arrow{r}{H}&\underline{\CH}\\
(\CA_\CN,\BE_\CN,\fs_\CN)\arrow[hook]{u}{}\arrow{r}{(Q,\mu)}&(\CA_\CN/\CN,\wtil{\BE_\CN},\wtil{\fs_\CN})\arrow{u}{\sim}[swap]{(F,\phi)}
\end{tikzcd}
\end{equation}
where $H$ is cohomological and the others are exact.
In such a sense, Liu-Nakaoka's heart construction is an extriangulated quotient.

\subsection{Hearts as an exact dg subquotient}
\label{subsec:Hearts_as_an_exact_dg_subquotient}
In addition to the circumstance in the previous subsection, we assume that $(\CA,\BE,\fs)$ admits an exact dg enhancement $(\A,\SS)$ which is connective and cofibrant.
We also have a canonical exact dg enhancement $(\N,\SS)$ of the extension-closed subcategory $\CN=\add(\CU*\CV)$.
Let us consider the exact dg subquotient $(\N,\SS_\N)\lra (\A_\N^\coh,\SS_\N)\overset{Q}{\lra}(\A_\N^\coh/\N,\wtil{\SS_\N})$ as in \cref{thm:dg_subquotient}.

The following theorem is constituted as a combination of results in \S\ref{subsec:dg_quotient_via_cohomological_envelope} and \S\ref{subsec:Hearts_as_an_extriangulated_subquotient}.

\begin{theorem}\label{thm:enhanced_heart}
The induced dg functor $Q\colon(\A,\SS)\to(\A_\N^\coh/\N,\wtil{\SS_\N})$ enhances the cohomological functor $H\colon (\CA,\BE,\fs)\to\underline{\CH}$ associated to the cotorsion pair $(\CU,\CV)$.
\end{theorem}
\begin{proof}
As the exact dg substructure $\A_\N$ is still connective, we have a natural exact morphism $\A_\N\to\A_\N^\coh$ to the cohomological envelope.
Also, since the exact structure $(\A_\N,\SS_\N)$ is a weaker one of $(\A,\SS)$, the identity $\id_\A$ gives rise to an exact functor $\id_\A\colon(\A_\N,\SS_\N)\to(\A,\SS)$.
The all exact morphisms obtained so far are displayed as follows.
\[
\begin{tikzcd}[row sep=0.4cm]
&(\A,\SS)&\\
&(\A_\N,\SS_\N)\arrow[hook]{u}{\id_\A}\arrow{d}{}&\\
(\N,\SS|_\N)\arrow[hook]{r}{}&(\A_\N^\coh,\SS_\N)\arrow{r}{Q}&(\A_\N^\coh/\N,\wtil{\SS_\N})
\end{tikzcd}
\]
By Theorems \ref{thm:heart_is_the_subquotient} and \ref{thm:dg_subquotient}, we know the bottom dg (sub)quotient enhances the extriangulated subquotient $(\CN,\BE|_\CN,\fs|_\CN)\hookrightarrow(\CA_\CN,\BE_\CN,\fs_\CN)\xto{H^0Q}(\CA_\CN/\CN,\wtil{\BE_\CN},\wtil{\fs_\CN})$.
\end{proof}

Then, by \eqref{diag:cohomological_functor_via_quotient}, we see the induced (non-exact) dg functor $(\A,\SS)\to (\A_\N^\coh/\N,\wtil{\SS_\N})$ enhances the cohomological functor $H\colon (\CA,\BE,\fs)\to\underline{\CH}$.

\medskip
\noindent
{\bf Acknowledgement.}
Y.\ O.\ is supported by JSPS KAKENHI (grant JP22K13893).
%

\end{document}